\documentclass[11pt]{article}

\usepackage{enumitem} 

\usepackage[T1]{fontenc}
\usepackage[utf8]{inputenc}
\usepackage[letterpaper,margin=1in]{geometry}
\usepackage{amsthm,amsmath,amssymb}
\usepackage[round,authoryear]{natbib}
\usepackage{microtype}
\usepackage{graphicx}
\usepackage{xcolor}
\usepackage{hyperref}
\newenvironment{keywords}{\par\small\noindent\textbf{Keywords: }}{\par\normalsize}
\usepackage{algorithm}
\usepackage{algorithmic}
\usepackage{booktabs}
\usepackage{multirow}

\usepackage{bm}

\usepackage{xcolor}

\definecolor{royalblue}{RGB}{65, 105, 225}
\definecolor{forestgreen}{RGB}{34, 139, 34}

\hypersetup{
    colorlinks=true,
    linkcolor=royalblue,   
    citecolor=forestgreen, 
    urlcolor=royalblue     
}

\usepackage{tikz}

\definecolor{mplblue}{HTML}{1F77B4}
\definecolor{mplorange}{HTML}{FF7F0E}
\definecolor{mplgreen}{HTML}{2CA02C}
\definecolor{mplred}{HTML}{D62728}
\definecolor{mplpurple}{HTML}{9467BD}

\newcommand{\legenditem}[2]{%
    \tikz[baseline=-0.5ex]{
        \draw[#1, line width=1pt] (0,0) -- (0.5,0);
    }%
    \hspace{0.3em}#2%
}
\usepackage{lastpage}

\usepackage{cleveref} 
\usepackage{thmtools}
\usepackage{thm-restate}

\usepackage{graphicx}
\usepackage{caption}
\usepackage{subcaption}

\newtheorem{theorem}{Theorem}
\newtheorem{lemma}{Lemma}

\newtheorem{definition}{Definition}
\newtheorem{proposition}{Proposition}

\newcommand{\vol}{\mathrm{dvol}}

\DeclareMathOperator*{\argmin}{arg\,min}

\begin{document}
\date{}

\title{Wasserstein Bounds for Ambient and Latent Smoothing on Data Manifolds}

\author{Wonjun Lee \and Wenyan Luo}
\date{Department of Mathematics, The Ohio State University\\
\href{mailto:lee.8222@osu.edu}{\texttt{lee.8222@osu.edu}}\quad
\href{mailto:luo.1409@osu.edu}{\texttt{luo.1409@osu.edu}}}

\maketitle

\begin{abstract}%
Smoothing empirical data can fill gaps between observations, while noise in the surrounding space can move mass away from the manifold supporting the data. We derive Wasserstein error bounds for ambient Gaussian smoothing and for smoothing through an encoder--decoder representation. For a compact connected smooth manifold without boundary and a smooth positive target density, established heat-flow and entropy estimates imply that sufficiently small intrinsic heat smoothing improves the squared intrinsic Wasserstein error. Coupling this estimate with ambient noise yields a bound whose quadratic term includes ambient dimension and a curvature correction. Our main latent bound tracks reconstruction error, tangential distortion, and the decoder's response to orthogonal latent noise. Its quadratic surrogate gives bandwidth rules and explicit conditions under which the latent upper-bound surrogate is smaller than the ambient one. These surrogate comparisons do not assert an ordering of the exact Wasserstein errors. Synthetic experiments examine the predicted scaling and latent geometry on spheres and a product of spheres.
\end{abstract}


\section{Introduction}
\label{sec:intro}
High-dimensional data in modern machine learning, such as high-resolution images, biological gene expression profiles, and 3D geometric meshes, is often modeled as concentrating near a low-dimensional manifold $M \subset \mathbb{R}^D$ with intrinsic dimension $m \ll D$. A foundational objective in statistical learning is to approximate the unknown target probability distribution $\mu$ supported on $M$ using a discrete empirical measure $\hat{\mu}_n = \frac{1}{n} \sum_{i=1}^n \delta_{x_i}$ formed by $n$ independent observed samples.

To evaluate estimation fidelity, we employ the 2-Wasserstein
distance. Existing results characterize the convergence
of empirical measures in Wasserstein distance and explain how
the intrinsic dimensional structure of the target distribution
can govern sample complexity
\citep{boissard2014mean,weed2019sharp}.
The precise rates depend on the dimension, distributional
assumptions, and probabilistic formulation of the error bound.

Complementary to analyses of convergence as the sample size
increases, we study the effect of smoothing at a fixed sample
size $n$. Specifically, we analyze
the 2-Wasserstein distance between the smoothed empirical measure and the target measure, denoted as
$W_2(k_\sigma\ast\hat{\mu}_n,\mu)$ relative to the unsmoothed
manifold-supported target measure $\mu$, with explicit
dependence on the Gaussian smoothing bandwidth $\sigma$.

Defined in ambient space $\mathbb{R}^D$ as $k_\sigma(x) = (2\pi\sigma^2)^{-D/2} \exp(-\|x\|^2 / 2\sigma^2)$ for bandwidth $\sigma > 0$, convolving the discrete measure $\hat{\mu}_n$ with $k_\sigma$ redistributes empirical mass to yield the continuous probability density
\begin{equation}
    (k_\sigma \ast \hat{\mu}_n)(x) = \frac{1}{n} \sum_{i=1}^n k_\sigma(x - x_i).
\end{equation}
Smoothing spreads mass around each observation and can help fill gaps between samples. When the data lie on a manifold, however, the direction of that movement matters. Movement along the manifold can bring mass closer to unobserved parts of the target distribution, whereas movement away from it introduces distortion. Ambient Gaussian noise produces both effects. As the surrounding dimension increases while the manifold stays fixed, more of the injected noise acts in directions away from the manifold.

We use intrinsic heat smoothing as a reference. On a compact connected smooth manifold without boundary and for a smooth positive target density, a dimensional evolution variational inequality (EVI) of \citet{erbar2015equivalence} and heat-kernel entropy asymptotics \citep{menkovski2024entropy} imply that sufficiently small positive intrinsic smoothing improves the squared intrinsic Wasserstein error. This heat-flow step is an application of existing results. Our ambient and latent bounds quantify what changes when the noise is applied in the observation space or through a learned representation.

The resulting ambient error bound suggests a smaller smoothing bandwidth as the surrounding dimension increases. More precisely, for fixed intrinsic geometry, target distribution, and sample, the bandwidth minimizing its quadratic surrogate scales inversely with ambient dimension; the improvement in that surrogate also shrinks. The detailed coefficients and the restrictions needed to apply this approximation are given in \Cref{sec:main_bounds}.

This motivates smoothing through an encoder--decoder representation: encode the observations in a lower-dimensional space, add Gaussian noise there, and decode the perturbed points. Reducing the noise dimension can help, but the encoder--decoder geometry also matters. A useful representation must reconstruct the data accurately, preserve movement along the manifold, and control the decoder's response to noise in other latent directions. Our latent bound quantifies these effects and gives sufficient conditions for its optimized quadratic surrogate to be smaller than the ambient counterpart.

The comparison concerns quadratic surrogates of upper bounds, rather than the exact Wasserstein errors. Applying it to the complete bounds requires controlling the remainder terms. Moreover, the common reference is the intrinsic empirical error, which can exceed the ambient empirical error. The numerical experiments in \Cref{sec:experiments} examine whether the predicted trends also appear in the computed transport errors.

\subsection{Summary of Contributions}

Our main contributions are as follows:
\begin{itemize}
        \item We state the intrinsic heat-flow consequence of existing estimates with the normalization and constants needed for the ambient and latent comparisons (\Cref{thm:intrinsic_integrated}).
        \item We decompose ambient smoothing into intrinsic and normal-noise effects, yielding a small-bandwidth bound with explicit ambient-dimensional dependence (\Cref{thm:decomposition}).
    \item We derive a latent smoothing bound that quantifies reconstruction error, tangential distortion, and orthogonal-noise amplification by the decoder. Its quadratic surrogate provides bandwidth-selection rules and an explicit criterion for a smaller optimized latent upper-bound surrogate (\Cref{thm:latent_decomposition}).
        \item We test the predicted bandwidth and geometric trends on controlled synthetic examples (\Cref{sec:experiments}).
\end{itemize}

\subsection{Related Work}
\label{sec:related_work}
Our work intersects statistical measure estimation on submanifolds, kernel regularization under optimal transport, score-based generative modeling, and geometric representation learning.

\paragraph{Statistical Estimation on Submanifolds.}
An extensive literature studies empirical and nonparametric
measure estimation under Wasserstein loss
\citep{dudley1969speed,fournier2015rate,weed2019sharp,
niles2022minimax}. When the target is supported on or near
a smooth low-dimensional manifold, convergence rates can depend
on the intrinsic geometry and density regularity rather than
directly on the ambient dimension. In particular,
\citet{divol2022measure} constructs kernel-based estimators for
manifold-supported measures and establishes intrinsic-dimensional
minimax rates. Our analysis addresses a different question:
for a fixed empirical measure, we quantify the local dependence
of intrinsic, ambient, and encoder--decoder smoothing errors on
the bandwidth and representation geometry.

\paragraph{Gaussian-Smoothed Optimal Transport.}
Building on classical non-parametric density estimation \citep{rosenblatt1956density, parzen1962estimation}, recent studies analyze the statistical convergence of smoothed measures under optimal transport \citep{goldfeld2020convergence,nietert2021smooth,goldfeld2024statistical}. \citet{chen2022asymptotics} study contraction when both measures are smoothed, and \citet{ding2022asymptotics} analyze the small-noise regime for two discrete measures. For suitable target distributions and smoothing scales, the distance $W_2(k_\sigma\ast\hat\mu_n,k_\sigma\ast\mu)$ has a parametric $n^{-1/2}$ rate. This rate does not hold uniformly over all targets at every fixed $\sigma$: \citet{block2026rate} identify a noise-scale threshold and give slower-rate examples below it. These results smooth both marginals, whereas we compare a smoothed empirical measure with the unsmoothed target $\mu$ and study its fixed-sample, small-bandwidth behavior.

\paragraph{Score-Based Diffusion Models.}
Score-based diffusion models use a family of noisy data distributions across noise scales $\sigma>0$ \citep{song2020score,ho2020denoising}. Their reverse dynamics use learned scores of these distributions to generate data; score matching and denoising objectives motivate how such scores are learned \citep{hyvarinen2005estimation,vincent2011connection}. Analyses under manifold assumptions study when such dynamics recover distributions concentrated on low-dimensional structure \citep{pidstrigach2022score,de2022convergence,oko2023diffusion}.
Recent work also establishes intrinsic-dimensional $W_1$ rates for score-based generative models on smooth manifolds \citep{fu2026intrinsic}.

To mitigate the computational cost of high-dimensional ambient score matching, Latent Diffusion Models (LDMs) apply this noisy forward-reverse mechanism within a learned latent space \citep{rombach2022high}. Our latent smoothing bound addresses a complementary question: how reconstruction and decoder geometry enter a local $W_2^2$ upper bound when Gaussian noise is added to encoded empirical samples.

\paragraph{Wasserstein Bounds for Encoder--Decoder Models.}
\citet{tolstikhin2018wasserstein} formulate generative autoencoding through optimal transport and relate the transport objective to reconstruction and the encoded latent law. Subsequent work gives Wasserstein regeneration guarantees for Wasserstein autoencoders \citep{chakrabarty2021statistical} and statistical $W_1$ bounds for variational autoencoders \citep{mbacke2023statistical}. \citet{loaiza2024manifold} show that two-step generative models optimize a possibly regularized Wasserstein upper bound. Our result treats a local question: starting from fixed encoded observations, it bounds the squared Wasserstein error under small Gaussian latent perturbations, separating reconstruction error, tangential distortion, and the decoder's response to orthogonal noise. This geometry-resolved bandwidth comparison builds on the established optimal-transport view of autoencoders.

\paragraph{Heat Flow and Intrinsic Smoothing.}
An evolution variational inequality (EVI) characterizes heat flow as the Wasserstein gradient flow of entropy: it bounds how the transport distance from the evolving measure to any fixed comparison measure changes in time. The finite-dimensional heat-flow EVI under a lower Ricci bound is established by \citet{erbar2015equivalence}; its precise form in our setting appears in \eqref{eq:dimensional_evi}. The small-time entropy expansion of a Riemannian heat kernel is given by \citet{menkovski2024entropy}. Heat-kernel smoothing also enters Wasserstein risk bounds for empirical measures on compact manifolds \citep{borda2023empirical}. Together the EVI and entropy estimates imply the intrinsic smoothing bound used here; we state the consequence with our heat-time normalization and empirical mixture. Our new analysis concerns the ambient and latent perturbations built on that input.

\paragraph{Geometric Regularization in Deep Representations.}
Mapping high-dimensional data into low-dimensional coordinate spaces while controlling geometric distortion is a foundational objective in manifold learning \citep{tenenbaum2000global, belkin2003laplacian}. Modern deep representation models rely on architectural constraints, such as spectral normalization \citep{miyato2018spectral}, and denoising objectives \citep{vincent2008extracting, alain2014regularized}, to regulate representation geometry. Our local bound makes the roles of the decoder's tangential and orthogonal responses explicit in a squared Wasserstein upper-bound surrogate.

\subsection{Organization of the Paper.}
Section~\ref{sec:setup} fixes the geometric and probabilistic setup. Section~\ref{sec:main_bounds} states the intrinsic heat-flow input and the ambient and latent bounds. Section~\ref{sec:proofs_overview} gives the proofs, and Section~\ref{sec:experiments} reports the numerical experiments.

\section{Problem Setup and Geometric Foundations}
\label{sec:setup}

We formalize the geometric and measure-theoretic setup used throughout the paper.

\paragraph{Submanifold Geometry and Reach}
Let $M \subset \mathbb{R}^D$ be a smooth, compact, connected $m$-dimensional Riemannian manifold without boundary, embedded in ambient Euclidean space $\mathbb{R}^D$, where $m \ll D$.
We equip $M$ with the induced Riemannian metric $g_M$: for
$u,v\in T_xM$, $g_M|_x(u,v)=\langle u,v\rangle_{\mathbb R^D}$.
The Ricci curvature $\operatorname{Ric}_M$ summarizes intrinsic curvature
in each tangent direction: for a unit vector $v\in T_xM$ completed to
an orthonormal basis $(v,e_2,\ldots,e_m)$, $\operatorname{Ric}_M|_x(v,v)$
is the sum of the sectional curvatures of the tangent planes spanned by
$v$ and $e_j$, $j=2,\ldots,m$. $\operatorname{Ric}_M\geq-\kappa g_M$ means
$\operatorname{Ric}_M|_x(v,v)\geq-\kappa g_M|_x(v,v)$ for every
$x\in M$ and $v\in T_xM$; individual sectional curvatures may still
be negative.

\begin{definition}[Reach \citep{federer1959curvature}]
The \emph{reach} of $M$, denoted by $\tau = \mathrm{reach}(M) > 0$, is the supremum over all $r > 0$ such that every point $y \in \mathbb{R}^D$ within an open $r$-neighborhood $U_r(M) = \{y \in \mathbb{R}^D : \mathrm{dist}(y, M) < r\}$ has a unique metric projection $\pi_M(y) \in M$.
\end{definition}

For any $x \in M$, the ambient space decomposes into the orthogonal sum $\mathbb{R}^D = T_x M \oplus N_x M$, where $T_x M \cong \mathbb{R}^m$ is the tangent space and $N_x M \cong \mathbb{R}^{D-m}$ is the normal space. The reach $\tau$ bounds the local curvature of $M$: for any $x, y \in M$, the normal displacement satisfies $\|\pi_{N_x M}(y - x)\| \le \frac{\|y - x\|^2}{2\tau}$.

\paragraph{Noise Model and Measures.}
Let $\mu=\rho\,\vol_M$ be an unknown probability measure on $M$ with a
smooth, strictly positive density $\rho$ satisfying
\[
    0<
    \rho(x)
    \leq
    \rho_{\max}
    <\infty
    \qquad
    \text{for every }x\in M.
\]
We start from an empirical sample $X_n = \{x_1, \dots, x_n\}$, defining the clean empirical measure $\hat{\mu}_n = \frac{1}{n} \sum_{i=1}^n \delta_{x_i}$. Adding isotropic Gaussian noise of scale $\sigma > 0$ to this measure yields the extrinsically smoothed empirical measure
\begin{equation}
    \mu_{\mathrm{amb}} = k_\sigma \ast \hat{\mu}_n, \quad \text{where } k_\sigma(y) = (2\pi \sigma^2)^{-D/2} \exp\left(-\frac{\|y\|^2}{2\sigma^2}\right).
\end{equation}
Note that we use $k_\sigma$ to denote the isotropic Gaussian kernel operating in either the ambient space $\mathbb{R}^D$ or the latent space $\mathbb{R}^d$; despite this slight abuse of notation, the intended dimension will be clear from context throughout. Conversely, the ideal intrinsically smoothed measure on $M$ is defined via the Riemannian heat kernel $k_\sigma^M(x, y)$
\begin{equation}
    \mu_\sigma = k_\sigma^M \ast \hat{\mu}_n = \frac{1}{n} \sum_{i=1}^n k_\sigma^M(x_i, \cdot).
\end{equation}

\paragraph{Wasserstein distance convention.}
We write $W_2$ for the $2$-Wasserstein distance associated with the
ambient Euclidean cost
\[
    c_{\mathrm{amb}}(x,y)=\|x-y\|^2,
\]
whereas $W_{2,M}$ denotes the intrinsic $2$-Wasserstein distance on
$M$ associated with the squared geodesic cost
\[
    c_M(x,y)=d_M^2(x,y).
\]
Since
$
    \|x-y\|\leq d_M(x,y),
$
$x,y\in M,$
any two probability measures $\nu_1,\nu_2$ supported on $M$ satisfy
\begin{equation}
\label{eq:ambient_intrinsic_w2_comparison}
    W_2(\nu_1,\nu_2)
    \leq
    W_{2,M}(\nu_1,\nu_2).
\end{equation}
We use $W_{2,M}$ to analyze intrinsic heat smoothing on $M$ and
$W_2$ whenever at least one of the measures may be supported outside
$M$, as occurs under ambient or latent-space smoothing.

\section{Main Theoretical Bounds}
\label{sec:main_bounds}

We study a compact, connected, smooth $m$-dimensional manifold $M$ without
boundary and a target probability measure $\mu=\rho\,\vol_M$ with smooth
density $0<\rho\leq\rho_{\max}$. Let
$\hat\mu_n=n^{-1}\sum_{i=1}^n\delta_{x_i}$, and let
$\mu_\sigma=k_\sigma^M\ast\hat\mu_n$ denote intrinsic heat smoothing at
time $t=\sigma^2$ for the generator $\frac12\Delta_M$.

An evolution variational inequality (EVI) describes heat flow as the
Wasserstein gradient flow of entropy by comparing the evolving measure
with any fixed measure. Here is the finite-dimensional form used below.
Let $H_s=e^{s\Delta_M}$, and for $\nu=h\,\vol_M$ define
$\operatorname{Ent}(\nu)=\int_M h\log h\,d\vol_M$ and
$U_m(\nu)=e^{-\operatorname{Ent}(\nu)/m}$. If
$\operatorname{Ric}_M\geq-\kappa g_M$ for $\kappa\geq0$, set
\[
\mathfrak s_{-\kappa/m}(r)=
\begin{cases}
\sinh(\sqrt{\kappa/m}\,r)/\sqrt{\kappa/m},&\kappa>0,\\
r,&\kappa=0.
\end{cases}
\]
Then \citet[Theorem~2, Eq.~(1.4)]{erbar2015equivalence} give,
for every initial measure $\nu_0$ and fixed comparison measure $\eta$
of finite entropy, for almost every $s>0$,
\begin{equation}\label{eq:dimensional_evi}
\frac{d}{ds}\mathfrak s_{-\kappa/m}\!\left(
\frac{W_{2,M}(H_s\nu_0,\eta)}{2}\right)^2
-\kappa\,\mathfrak s_{-\kappa/m}\!\left(
\frac{W_{2,M}(H_s\nu_0,\eta)}{2}\right)^2
\leq\frac m2\left(1-\frac{U_m(\eta)}{U_m(H_s\nu_0)}\right).
\end{equation}
The EVI controls the change in Wasserstein distance through an entropy
comparison. Combining it with the heat-kernel entropy asymptotic of
\citet[Theorem~4.4]{menkovski2024entropy} gives the intrinsic estimate
below. We derive it in \Cref{sec:intrinsic_heat_derivative_proof} with
our heat-time normalization and constants.

\begin{theorem}[\cite{erbar2015equivalence,menkovski2024entropy}]
\label{thm:intrinsic_integrated}
Choose $\kappa\geq0$ such that $\operatorname{Ric}_M\geq-\kappa g_M$,
and write $R=\operatorname{diam}_M(M)$. Define
\begin{align*}
a&:=\sqrt{\kappa/m},\quad 
\Lambda_{\kappa,m,R}:=\frac{\sinh(aR/2)}{aR/2}\cosh(aR/2),\quad
C_{\mathrm{int}}:=\frac{m\rho_{\max}^{-1/m}}
{\sqrt{2\pi e}\,\Lambda_{\kappa,m,R}}.
\end{align*}
At $\kappa=0$, use the continuous extension $\Lambda_{0,m,R}=1$.
Write $B_{\mathrm{int}}:=m+\kappa R^2/2$ for the intrinsic
quadratic coefficient. This notation will also be used in the latent
bound below.
For every fixed empirical measure
$\hat\mu_n$ and all sufficiently small $\sigma>0$,
\begin{equation}\label{eq:integrated_w2_bound}
W_{2,M}^2(\mu_\sigma,\mu)
\leq W_{2,M}^2(\hat\mu_n,\mu)
-C_{\mathrm{int}} n^{-1/m}\sigma+B_{\mathrm{int}}\sigma^2.
\end{equation}
\end{theorem}
The theorem holds for every compact smooth boundaryless manifold, as each has a finite lower Ricci bound. The later ambient and latent bounds still require a smooth embedding and positive reach; they do not claim the full nonsmooth metric-measure generality of the cited EVI.

The ambient and latent results below use this established intrinsic
heat-flow input to quantify how the surrounding representation changes
the smoothing error. The next decomposition itself only requires the
smooth embedding and positive reach; the Ricci lower bound enters
its explicit combined bound through \Cref{thm:intrinsic_integrated}.

\begin{theorem}[Ambient smoothing error bound]\label{thm:decomposition}
Let $M\subset\mathbb R^D$ be a compact smooth embedded
$m$-dimensional manifold with reach $\tau>0$. Let
$k_\sigma\ast\hat\mu_n$ be ambient isotropic Gaussian smoothing
with covariance $\sigma^2\mathbf I_D$. For sufficiently small
$0<\sigma<\tau/\sqrt{mD}$,
\begin{equation}\label{eq:ambient_bound_main}
W_2^2(k_\sigma\ast\hat\mu_n,\mu)
\leq W_2^2(\mu_\sigma,\mu)
+(D-m)\sigma^2
+\frac{4m}{\tau}W_2(\hat\mu_n,\mu)\sigma^2
+\mathcal O\!\left(\frac{\sigma^3}{\tau}\right)
+C_{\mathrm{tail}}.
\end{equation}
The nonnegative $C_{\mathrm{tail}}$ is specified in
\eqref{eq:tail_bound_proof} and decays exponentially in $1/\sigma^2$
for fixed $m,D,M$.

With the intrinsic assumptions above, for sufficiently small
$\sigma$, the intrinsic estimate gives the explicit combined bound
\begin{align}
W_2^2(k_\sigma\ast \hat\mu_n,\mu)
&\leq W_{2,M}^2(\hat\mu_n,\mu)
-C_{\mathrm{int}} n^{-1/m}\sigma \nonumber\\
&\quad+
\left(D+\frac{\kappa R^2}{2}+\frac{4m}{\tau}W_2(\hat\mu_n,\mu)\right)\sigma^2
+\mathcal O\!\left(\frac{\sigma^3}{\tau}\right)
+C_{\mathrm{tail}}.
\label{eq:ambient_combined_bound}
\end{align}
\end{theorem}

The $(D-m)\sigma^2$ term in \eqref{eq:ambient_bound_main} is the expected squared normal
displacement. Combining it with $\mathcal{O}(\sigma^2)$ term from
intrinsic smoothing gives $(D+\kappa R^2/2)\sigma^2$ in
\eqref{eq:ambient_combined_bound}. The codimension identifies the
normal-noise contribution, while $D$ governs the dimensional growth
of the combined quadratic coefficient.

To reduce the ambient-dimensional contribution to the combined bound, modern generative models perform smoothing within a lower-dimensional latent space $\mathbb{R}^d$, where $m \le d \ll D$. Mass is smoothed in latent space via an encoder $f: \mathbb{R}^D \to \mathbb{R}^d$ and mapped back to ambient space through a decoder $g: \mathbb{R}^d \to \mathbb{R}^D$.

To characterize the fidelity of the autoencoder pair $(f, g)$ near data points $x_i \in M$, we introduce three groups of geometric quantities:
\begin{itemize}
    \item {Reconstruction Error ($\delta$):} The root-mean-square reconstruction error
    \begin{equation}\label{eq:def-delta}
        \delta = \left( \frac{1}{n} \sum_{i=1}^n \|x_i - g(f(x_i))\|^2 \right)^{1/2},
    \end{equation}
    quantifying how accurately the autoencoder preserves empirical data points on average.
    \item {Tangential Gain ($\rho_\parallel$) and Isometry Defect ($\varepsilon_{\mathrm{iso}}$):} Fix one global gain $\rho_\parallel>0$ and let
    $A_i:=\left.J_g(f(x_i))J_f(x_i)\right|_{T_{x_i}M}:T_{x_i}M\to\mathbb R^D$.
    With $\iota_i:T_{x_i}M\hookrightarrow\mathbb R^D$ denoting inclusion, define
    \begin{equation}\label{eq:def-isometry-defect}
        \varepsilon_{\mathrm{iso}}
        :=\max_{1\le i\le n}\min_{O\in O(T_{x_i}M)}
        \|A_i-\rho_\parallel\,\iota_iO\|_{\mathrm{op}}.
    \end{equation}
    Here $O(T_{x_i}M)$ is the orthogonal group of the tangent space. The defect measures the worst-case local departure of the encoder--decoder composition from a conformal tangential isometry.
    \item {Orthogonal Expansion Factor ($L_{\mathrm{orth}}$):} Let $V_i=J_f(x_i)(T_{x_i}M)$ and choose one $L_{\mathrm{orth}}\geq0$ such that $\|J_g(f(x_i))v^\perp\|\leq L_{\mathrm{orth}}\|v^\perp\|$ for every $i$ and $v^\perp\in V_i^\perp$. This factor measures how much the decoder amplifies noise added in directions perpendicular to the manifold representation in latent space.
\end{itemize}

\begin{theorem}[Latent-Space Smoothing Error Bound]\label{thm:latent_decomposition}
Let $M \subset \mathbb{R}^D$ be a compact, connected, smooth, boundaryless $m$-dimensional submanifold with reach $\tau > 0$. Choose $\kappa\geq0$ with $\operatorname{Ric}_M\geq-\kappa g_M$, and use the intrinsic constants from \Cref{thm:intrinsic_integrated}. Let $\mu=\rho\,\vol_M$ have a smooth, strictly positive density bounded above by $\rho_{\max}$, and let $\hat{\mu}_n = \frac{1}{n}\sum_{i=1}^n \delta_{x_i}$ be an empirical measure. Let $f$ and $g$ be the encoder and decoder above, with $m\leq d\leq D$.
Use $\delta$ from \eqref{eq:def-delta}, $\rho_\parallel$ and
$\varepsilon_{\mathrm{iso}}$ from \eqref{eq:def-isometry-defect}, and
$L_{\mathrm{orth}}$ and $V_i$ as defined above. For every empirical
point $x_i$, assume that
$\left.J_f(x_i)\right|_{T_{x_i}M}:T_{x_i}M\to V_i$
is a linear isometry. Assume also that
$g\in C^2(\mathbb R^d;\mathbb R^D)$ has uniformly bounded first and
second derivatives.
Set $W_0:=W_{2,M}(\hat\mu_n,\mu)$ and define
\begin{equation}\label{eq:decoder_uniform_C2_constant}
\begin{aligned}
 C_g&:=\max\left\{1,\sup_z\|J_g(z)\|_{\mathrm{op}},
                 \sup_z\|D^2g(z)\|_{\mathrm{op}}\right\}.
\end{aligned}
\end{equation}
and
\begin{equation}\label{eq:latent_Q_main}
\begin{aligned}
    Q&:=\rho_\parallel^2B_{\mathrm{int}}
     \left(1+\frac{\delta}{W_0}\right)
     +2\left(\frac{2m\rho_\parallel^2}{\tau}+\frac{C_g}{2}\sqrt{d(d+2)}\right)(W_0+\delta).
\end{aligned}
\end{equation}
For sufficiently small $\sigma>0$ such that
    $\sigma<\frac{1}{2\sqrt m\,\rho_\parallel}
    \min\{W_0,2\tau\},$
the generated pushforward measure
$\mu_{\mathrm{out}}
=g_\#\!\left((f_\#\hat{\mu}_n)\ast
\mathcal N(0,\sigma^2\mathbf I_d)\right)$ satisfies
\begin{equation}\label{eq:latent_bound_main}
    W_2^2(\mu_{\mathrm{out}}, \, \mu) \le \mathcal{E}_{\mathrm{const}} - \Gamma_{\mathrm{eff}} \sigma + \beta_{\mathrm{lat}} \sigma^2 + \mathcal{O}(\sigma^3),
\end{equation}
where
\begin{itemize}
    \item $\mathcal{E}_{\mathrm{const}} = (1+\varepsilon_{\mathrm{iso}})W_0^2 + 2\delta W_0 + \delta^2$ is the constant term in the derived upper bound.  The contribution
    $\varepsilon_{\mathrm{iso}}W_0^2$ is introduced by the Young inequality used to control the tangential-defect cross term and should not be interpreted as part of the exact zero-noise reconstruction error,
    \item $\Gamma_{\mathrm{eff}} = \rho_\parallel C_{\mathrm{int}} n^{-1/m} \left(1 + \frac{\delta}{W_0} \right)$ is the effective tangential variance reduction rate, driven by global tangential gain $\rho_\parallel$ with the factor $1+\delta/W_0$ arising from the upper-bound estimate for the reconstruction cross term. This factor does not mean that reconstruction error improves the exact transport cost. Here $C_{\mathrm{int}}>0$ is the intrinsic rate constant,
    \item $\beta_{\mathrm{lat}} = L_{\mathrm{orth}}^2(d-m)
    +(\rho_\parallel^2B_{\mathrm{int}}+m)\varepsilon_{\mathrm{iso}}
    +m\varepsilon_{\mathrm{iso}}^2
    +Q$
    is the second-order coefficient. It displays the orthogonal-noise
    contribution and the quadratic and linear isometry-defect contributions
    separately from the remaining terms $Q$ in \eqref{eq:latent_Q_main}.
\end{itemize}
\end{theorem}

\Cref{thm:latent_decomposition_effective_gain} shows that latent smoothing
replaces the ambient-dimensional noise penalty with a second-order coefficient
governed by the latent geometry.  
Combining \Cref{thm:intrinsic_integrated,thm:decomposition,thm:latent_decomposition_effective_gain}
gives the small-bandwidth upper bounds
\begin{align}
    W_2^2(k_\sigma\ast\hat\mu_n,\mu)
    &\leq
    \Psi_{\mathrm{amb}}(\sigma)
    +
    \mathcal R_{\mathrm{amb}}(\sigma),
    \label{eq:ambient_surrogate_remainder}\\
    W_2^2(\mu_{\mathrm{out}}(\sigma),\mu)
    &\leq
    \Psi_{\mathrm{lat}}(\sigma)
    +
    \mathcal R_{\mathrm{lat}}(\sigma),
    \label{eq:latent_surrogate_remainder}
\end{align}
where
\begin{align}
    \Psi_{\mathrm{amb}}(\sigma)
    &:=
    W_0^2-\alpha_n\sigma+\beta_{\mathrm{amb}}\sigma^2,
    \label{eq:psi_amb}\\
    \Psi_{\mathrm{lat}}(\sigma)
    &:=
    \mathcal E_{\mathrm{const}}
    -\Gamma_{\mathrm{eff}}\sigma
    +\beta_{\mathrm{lat}}\sigma^2,
    \label{eq:psi_lat}
\end{align}
and
\[
    \mathcal R_{\mathrm{amb}}(\sigma)
    =
    \mathcal O(\sigma^3)
    +
    C_{\mathrm{tail}}(\tau,\sigma,m,D,M),
    \qquad
    \mathcal R_{\mathrm{lat}}(\sigma)
    =
    \mathcal O(\sigma^3).
\]
Here
\[
    \alpha_n:=C_{\mathrm{int}}n^{-1/m},
    \qquad
    \Gamma_{\mathrm{eff}}
    :=
    \rho_\parallel\alpha_n
    \left(1+\frac{\delta}{W_0}\right),
\]
and we take
\[
    \beta_{\mathrm{amb}}
    :=
    D+\frac{\kappa R^2}{2}+\frac{4m}{\tau}W_2(\hat\mu_n,\mu).
\]
For fixed intrinsic geometry, sample, and target measure,
under an isometric inclusion into a larger ambient space,
$\beta_{\mathrm{amb}}$ therefore grows linearly with $D$.
The notation $\mu_{\mathrm{out}}(\sigma)$ emphasizes the
dependence of the generated measure on the latent bandwidth.

The functions
$\Psi_{\mathrm{amb}}$ and $\Psi_{\mathrm{lat}}$ are the
quadratic, second-order surrogates obtained by omitting the remainder terms
from \eqref{eq:ambient_surrogate_remainder} and
\eqref{eq:latent_surrogate_remainder}. Their unconstrained
minimizers over $\sigma\geq0$ are
\begin{align}
    \sigma_{\mathrm{amb}}^{\ast}
    &:=
    \argmin_{\sigma\geq0}
    \Psi_{\mathrm{amb}}(\sigma)
    =
    \frac{\alpha_n}{2\beta_{\mathrm{amb}}},
    \label{eq:ambient_surrogate_optimizer}\\
    \sigma_{\mathrm{lat}}^{\ast}
    &:=
    \argmin_{\sigma\geq0}
    \Psi_{\mathrm{lat}}(\sigma)
    =
    \frac{\Gamma_{\mathrm{eff}}}{2\beta_{\mathrm{lat}}}.
    \label{eq:latent_surrogate_optimizer}
\end{align}
These expressions are relevant when the resulting bandwidths lie inside the
small-bandwidth regimes of the corresponding theorems.  The minimum
surrogate values are
\begin{align}
    \min_{\sigma\geq0}
    \Psi_{\mathrm{amb}}(\sigma)
    &=
    W_0^2-\frac{\alpha_n^2}{4\beta_{\mathrm{amb}}},
    \label{eq:ambient_surrogate_minimum}\\
    \min_{\sigma\geq0}
    \Psi_{\mathrm{lat}}(\sigma)
    &=
    \mathcal E_{\mathrm{const}}
    -
    \frac{\Gamma_{\mathrm{eff}}^2}{4\beta_{\mathrm{lat}}}.
    \label{eq:latent_surrogate_minimum}
\end{align}
Accordingly, the rigorous bounds evaluated at these surrogate bandwidths
retain the remainder terms:
\begin{align}
    W_2^2
    \left(
        k_{\sigma_{\mathrm{amb}}^{*}}\ast\hat\mu_n,
        \mu
    \right)
    &\leq
    W_0^2
    -
    \frac{\alpha_n^2}{4\beta_{\mathrm{amb}}}
    +
    \mathcal R_{\mathrm{amb}}
    \left(
        \sigma_{\mathrm{amb}}^{*}
    \right),
    \label{eq:ambient_bound_at_surrogate_optimizer}\\
    W_2^2
    \left(
        \mu_{\mathrm{out}}
        \left(
            \sigma_{\mathrm{lat}}^{*}
        \right),
        \mu
    \right)
    &\leq
    \mathcal E_{\mathrm{const}}
    -
    \frac{\Gamma_{\mathrm{eff}}^2}{4\beta_{\mathrm{lat}}}
    +
    \mathcal R_{\mathrm{lat}}
    \left(
        \sigma_{\mathrm{lat}}^{*}
    \right).
    \label{eq:latent_bound_at_surrogate_optimizer}
\end{align}
Thus, $\sigma_{\mathrm{amb}}^{*}$ and
$\sigma_{\mathrm{lat}}^{*}$ are minimizers of the quadratic surrogates,
not necessarily minimizers of the exact Wasserstein errors.

Both bounds express the same competition: smoothing first reduces the
intrinsic transport error at a rate proportional to $\sigma$, while
geometric distortion contributes at order $\sigma^2$. In the ambient
bound, the quadratic coefficient grows with $D$ because isotropic noise
also moves points away from $M$. In the latent bound,
$L_{\mathrm{orth}}^2(d-m)$ isolates the contribution from noise in the
extra latent directions. The full quadratic coefficient also contains
$Q$ from \eqref{eq:latent_Q_main}, including a decoder Taylor-remainder
term proportional to $C_g\sqrt{d(d+2)}(W_0+\delta)$. Thus the benefit
of smoothing in fewer dimensions depends on the full encoder--decoder
geometry, not only on $d-m$.

For $\rho_\parallel>0$, the latent surrogate initially decreases from
its constant term $\mathcal E_{\mathrm{const}}$; the complete asymptotic
upper bound has the same initial decrease because its remainder is
$\mathcal O(\sigma^3)$. The question is whether its best surrogate
value is also below the best ambient surrogate value. By
\eqref{eq:ambient_surrogate_minimum} and
\eqref{eq:latent_surrogate_minimum}, this
holds exactly when
\begin{equation}
\label{eq:general_latent_dominance}
    \underbrace{\varepsilon_{\mathrm{iso}}W_0^2+2W_0\delta+\delta^2}_{
        \text{latent starting penalty}}
    <
    \underbrace{\frac{\Gamma_{\mathrm{eff}}^2}{4\beta_{\mathrm{lat}}}}_{
        \text{latent surrogate gain}}
    -
    \underbrace{\frac{\alpha_n^2}{4\beta_{\mathrm{amb}}}}_{
        \text{ambient surrogate gain}}.
\end{equation}
Thus, a larger reduction along the latent surrogate must compensate for
its higher starting bound. The starting penalty
$\varepsilon_{\mathrm{iso}}W_0^2+2W_0\delta+\delta^2$ records the
effects of tangential distortion and reconstruction error in this bound.

The coefficients show what the representation must accomplish.
Tangential gain $\rho_\parallel$ carries the intrinsic smoothing
benefit into the decoded space, but it also enters the quadratic
coefficient. The isometry defect $\varepsilon_{\mathrm{iso}}$ increases
both the starting penalty and $\beta_{\mathrm{lat}}$.
The decoder's orthogonal response $L_{\mathrm{orth}}$ multiplies the
penalty from the $d-m$ extra latent directions. Reconstruction error
$\delta$ likewise raises the starting penalty and enters $Q$. These
competing effects mean that reducing $d$ alone does not guarantee a
better latent bound.

To make the criterion concrete, consider unit tangential gain
$\rho_\parallel=1$ and exact reconstruction $\delta=0$. Unit gain
means that the reference map $\iota_iO$ in
\eqref{eq:def-isometry-defect} preserves the length of each tangent
displacement instead of expanding or contracting it. If also
$\varepsilon_{\mathrm{iso}}=0$, then $A_i=\iota_iO_i$ for some
$O_i\in O(T_{x_i}M)$ at every sample point: to first order, the
encoder--decoder composition is a tangential isometry, up to an
orthogonal transformation of the tangent space. With
$\rho_\parallel=1$ and $\delta=0$, $\Gamma_{\mathrm{eff}}=\alpha_n$,
so latent and ambient
smoothing have the same linear coefficient. Writing out $Q$ from
\eqref{eq:latent_Q_main}, the latent quadratic coefficient becomes
\begin{align*}
    Q_0
    &:=B_{\mathrm{int}}
      +2\left(\frac{2m}{\tau}+\frac{C_g}{2}\sqrt{d(d+2)}\right)W_0,\\
    \beta_{\mathrm{lat}}^{(0)}
    &:=L_{\mathrm{orth}}^2(d-m)+Q_0
      +(B_{\mathrm{int}}+m)\varepsilon_{\mathrm{iso}}
      +m\varepsilon_{\mathrm{iso}}^2.
\end{align*}
The unit-gain, exact-reconstruction version of
\eqref{eq:general_latent_dominance} is therefore
\begin{equation}
\label{eq:ideal_latent_dominance}
    \varepsilon_{\mathrm{iso}}W_0^2
    <\frac{\alpha_n^2}{4}
      \left(\frac{1}{\beta_{\mathrm{lat}}^{(0)}}
      -\frac{1}{\beta_{\mathrm{amb}}}\right).
\end{equation}
If tangential isometry is also exact
($\varepsilon_{\mathrm{iso}}=0$), this simplifies to
\[
    \beta_{\mathrm{amb}}
    >L_{\mathrm{orth}}^2(d-m)+Q_0.
\]
With the same linear improvement, the comparison of quadratic
surrogates is governed by their full quadratic coefficients. The
ambient coefficient grows with $D$; the latent coefficient includes
both $L_{\mathrm{orth}}^2(d-m)$ and $Q_0$, which itself depends on $d$
and $C_g$. For fixed intrinsic geometry and uniformly controlled reach,
latent dimension, and decoder quantities, the surrogate condition can hold as
$D$ grows. At fixed other coefficients, a
small nonzero isometry defect uses part of that advantage twice: it
raises the left side of
\eqref{eq:ideal_latent_dominance} and increases its latent quadratic
coefficient.

Exact reconstruction is the useful benchmark, but the comparison also
shows why reconstruction accuracy matters in practice. When
$\delta>0$, the latent starting penalty acquires
$2W_0\delta+\delta^2$, and $Q$ increases with $\delta$. Although
$\Gamma_{\mathrm{eff}}$ also contains $1+\delta/W_0$, that factor
comes from the bound and does not make reconstruction error desirable.
For example, if $W_0\asymp n^{-1/m}$ and the other coefficients are
controlled, a reconstruction error of order $n^{-1/m}$ contributes at
the same $n^{-2/m}$ scale as the surrogate smoothing gain. The
encoder--decoder pair must therefore reconstruct accurately enough for
that cost to stay below the gain available from latent smoothing.

These conditions give a quantitative way to assess when latent
smoothing has the smaller optimized quadratic upper-bound surrogate:
preserve tangent directions, control decoder response outside them, and keep
reconstruction error small. The selected bandwidths must lie in the
small-bandwidth regime; for the complete bounds, the gap in
\eqref{eq:general_latent_dominance} must also exceed the remainder
terms in \eqref{eq:ambient_bound_at_surrogate_optimizer} and
\eqref{eq:latent_bound_at_surrogate_optimizer}. The numerical
experiments examine how these predicted geometric effects appear in
the computed Wasserstein errors.

\section{Proof of Main Theorems}
\label{sec:proofs_overview}
This section proves the three main results. The intrinsic estimate follows
from the dimensional heat-flow inequality and a heat-kernel entropy expansion.
The ambient proof separates normal noise from intrinsic smoothing, while the
latent proof tracks reconstruction and decoder distortion.

Throughout this section, $k_\sigma^M$ denotes the intrinsic heat kernel at
diffusion time $t=\sigma^2$, normalized so that its tangent-space covariance at
time $t$ is $t\mathbf I_m$. Equivalently, the heat generator is
$\frac12\Delta_M$. When the diffusion time itself is used as the
subscript, $k_t^M$ is shorthand for $k_{\sqrt t}^M$. Thus
$k_t^M\ast\hat\mu_n$ and $k_\sigma^M\ast\hat\mu_n$ refer to the same heat flow
under the relation $t=\sigma^2$.

\subsection{Intrinsic heat-flow input}
\label{sec:intrinsic_heat_derivative_proof}
The proof below specializes the dimensional EVI to the paper's
heat-time convention and uses the entropy of a finite heat-kernel
mixture.

\begin{proof}[Proof of \Cref{thm:intrinsic_integrated}]
Apply \eqref{eq:dimensional_evi} with initial measure
$\nu_0=\hat\mu_n$ and comparison measure $\eta=\mu$.
In our normalization, $\mu_t=H_{t/2}\hat\mu_n$.
The ratio of two $U_m$ values is unchanged if Riemannian volume is
normalized to a probability measure. Set
$F(t):=W_{2,M}^2(\mu_t,\mu)$ and $r(t):=\sqrt{F(t)}/2$.
After the time change $s=t/2$, the dimensional EVI reads, for almost
every $t>0$,
\[
\frac{d}{dt}\mathfrak s_{-a^2}(r(t))^2
-\frac{\kappa}{2}\mathfrak s_{-a^2}(r(t))^2
\leq\frac m4\left(1-\frac{U_m(\mu)}{U_m(\mu_t)}\right),
\]
Here $a=\sqrt{\kappa/m}$, with the $\kappa=0$ case understood by
continuity. 
Define
\[
b(r):=\frac{\sinh(ar)}{ar}\cosh(ar),
\qquad b(0):=1.
\]
Direct differentiation gives
$4\frac{d}{dF}\mathfrak s_{-\kappa/m}(\sqrt F/2)^2=b(\sqrt F/2)$.
Because $0\leq r(t)\leq R/2$ and $b$ is increasing,
$1\leq b(r(t))\leq\Lambda_{\kappa,m,R}$. Also,
\[
\frac{2\kappa\mathfrak s_{-\kappa/m}(r)^2}{b(r)}
=2\kappa r\frac{\tanh(ar)}a
\leq 2\kappa r^2\leq\frac{\kappa R^2}{2},
\]
with the zero-curvature expression understood by continuity.
Consequently the EVI implies
\begin{equation}\label{eq:intrinsic_evi_derivative}
F'(t)\leq B_{\mathrm{int}}
-\frac{m}{\Lambda_{\kappa,m,R}}
 \frac{U_m(\mu)}{U_m(\mu_t)},
\end{equation}
where $B_{\mathrm{int}}=m+\frac{\kappa R^2}{2}$.
The EVI applies with the atomic initial measure $\hat\mu_n$.
Since $\rho\leq\rho_{\max}$,
$U_m(\mu)\geq\rho_{\max}^{-1/m}$.

Write $p_{i,t}=k_t^M(\cdot,x_i)$, the heat-kernel density for
generator $\frac12\Delta_M$. Let $\operatorname{Scal}(x)$ be the
scalar curvature at $x$: for any $g_M$-orthonormal basis
$(e_1,\ldots,e_m)$ of $T_xM$,
$\operatorname{Scal}(x)=\sum_{j=1}^m\operatorname{Ric}_M|_x(e_j,e_j)$.
The entropy expansion
\citep[Theorem 4.4]{menkovski2024entropy} gives
\[
\operatorname{Ent}(p_{i,t}\,\vol_M)
=-\frac m2\log(2\pi e t)
+\frac14\operatorname{Scal}(x_i)t+o(t).
\]
The Ricci lower bound implies $\operatorname{Scal}\geq-m\kappa$.
For any fixed $\eta>0$ and the fixed empirical configuration, there
is $t_{\mathrm{ent}}>0$ such that, for every $i$ and
$0<t<t_{\mathrm{ent}}$,
\[
\operatorname{Ent}(p_{i,t}\,\vol_M)
\geq-\frac m2\log(2\pi e t)-C_{\mathrm{ent}}t,
\qquad C_{\mathrm{ent}}:=\frac{m\kappa}{4}+\eta.
\]
This specifies an admissible exponential coefficient but the cited
asymptotic does not give a numerical value of $t_{\mathrm{ent}}$.
The mixture inequality $\mu_t\geq p_{i,t}/n$ gives
\[
\operatorname{Ent}(\mu_t)
\geq\frac1n\sum_i\operatorname{Ent}(p_{i,t}\,\vol_M)-\log n,
\quad
U_m(\mu_t)\leq n^{1/m}\sqrt{2\pi e t}\,
e^{C_{\mathrm{ent}}t/m}.
\]
Substitution into \eqref{eq:intrinsic_evi_derivative} yields
\begin{equation}\label{eq:differential_w2_bound_explicit}
F'(t)\leq B_{\mathrm{int}}-
\frac{m\rho_{\max}^{-1/m}}
{\Lambda_{\kappa,m,R}\sqrt{2\pi e}}\,
e^{-C_{\mathrm{ent}}t/m}n^{-1/m}t^{-1/2}
\quad\text{for a.e. }0<t<t_{\mathrm{ent}}.
\end{equation}
Decrease the time threshold so that
$e^{-C_{\mathrm{ent}}t/m}\geq\frac12$ there. For example, when
$C_{\mathrm{ent}}>0$ one may require
$t\leq(m/C_{\mathrm{ent}})\log2$ in addition to
$t<t_{\mathrm{ent}}$. Then
$F'(t)\leq B_{\mathrm{int}}-\frac12C_{\mathrm{int}}
n^{-1/m}t^{-1/2}$.
The heat flow converges to $\hat\mu_n$ in $W_{2,M}$ as $t\downarrow0$
because $M$ is compact. Integrating from $\varepsilon$ to $t$ and
letting $\varepsilon\downarrow0$ gives
\[
F(t)\leq W_{2,M}^2(\hat\mu_n,\mu)
-C_{\mathrm{int}}n^{-1/m}\sqrt t+B_{\mathrm{int}}t.
\]
Setting $t=\sigma^2$ proves \eqref{eq:integrated_w2_bound}.
\end{proof}

\subsection{Proof of the Ambient Smoothing Error Bound}
\label{sec:ambient_smoothing_proof}

Rather than comparing $k_\sigma\ast\hat\mu_n$ directly with $\mu$, we isolate
three geometric contributions
\begin{enumerate}
    \item the displacement caused by projecting ambient Gaussian noise to
    tangent spaces of $M$;
    \item the discrepancy between extrinsically projected noise and intrinsic
    heat smoothing; and
    \item the intrinsic estimation error on $M$.
\end{enumerate}
The lemmas below control these terms in the order in which they enter the final
coupling argument.

To establish a baseline bound on the intrinsic error before introducing ambient noise leakage or latent distortion, we analyze the distance between the ambiently smoothed measure and its metric projection onto tangent space of each point.

\begin{lemma}[Extrinsic Projection Distance Bound]\label{lem:projW}
Assume $\tau > \sigma \sqrt{D}$.
Let $Y \sim k_\sigma \ast \hat{\mu}_n$ be the ambiently smoothed empirical distribution on $M \subset \mathbb{R}^D$ with reach $\tau > 0$, and let $\mu^{\mathrm{tan}}_\sigma = \frac{1}{n}\sum_i (\pi_{x_i+T_{x_i}M})_{\#} \mathcal{N}(x_i, \sigma^2 \mathbf{I}_D)$ denote the tangentially smoothed distribution. The squared 2-Wasserstein distance is bounded by
\begin{equation*}
    W_2^2\big(k_\sigma \ast \hat{\mu}_n, \, \mu^{\mathrm{tan}}_\sigma\big) \le \sigma^2(D-m).
\end{equation*}
\end{lemma}

\begin{proof}
Under the suboptimal coupling induced by the projection map $Y \mapsto \pi_{X_n+T_{X_n}M}(Y)$ with $X_n\sim \hat\mu_n$, we bound $W_2^2$ by the expected squared displacement in ambient space
\begin{equation*}
    W_2^2\big(k_\sigma \ast \hat{\mu}_n, \, \mu^{\mathrm{tan}}_\sigma\big) \le \mathbb{E}\left[ \|Y - \pi_{X_n+T_{X_n}M}(Y)\|^2 \right].
\end{equation*}
Expressing $Y = X_n + W$ and ambient Gaussian noise $W \sim \mathcal{N}(0, \sigma^2 \mathbf{I}_D)$, we decompose $W = W^\parallel + W^\perp$ into orthogonal projections onto the tangent space $T_{X_n}M \cong \mathbb{R}^m$ and normal space $N_{X_n}M \cong \mathbb{R}^{D-m}$. Then,
\begin{align*}
    \mathbb{E}\left[ \|Y - \pi_{X_n+T_{X_n}M}(Y)\|^2 \right]
    &= \mathbb{E}\left[ \|W^\perp\|^2 \right] = \sigma^2(D-m).
\end{align*}
This completes the proof.
\end{proof}

To complete the intrinsic error channel, we analyze the local behavior of the true intrinsic heat-smoothed distribution $\mu_\sigma = k_{\sigma}^M \ast \hat{\mu}_n$ and bridge the gap to its tangential Gaussian approximation $\mu^{\mathrm{tan}}_\sigma$. The following lemma constructs a synchronous coupling between Riemannian Brownian motion on $M$ and its flat tangent-space approximation, establishing explicit $L^2$ bounds on both the total diffusion displacement $\|X - Z_n\|_{L^2}$ from the initial sample $Z_n$ and the approximation discrepancy $\|X - \widetilde{X}\|_{L^2}$ as functions of the intrinsic dimension $m$, reach $\tau$, and scale parameter $\sigma$.

\begin{lemma}[Intrinsic Heat Flow vs.\ Tangent Gaussian Coupling]
\label{lem:intrinsic_one_step}
Let $M\subset\mathbb R^D$ be a compact $m$-dimensional smooth submanifold with reach $\tau>0$, and let
    $\hat\mu_n=\frac1n\sum_{i=1}^n\delta_{x_i}$
be an empirical measure on $M$. Let $X\sim\mu_\sigma=k_\sigma^M\ast\hat\mu_n$ denote the endpoint at diffusion time $t=\sigma^2$ of Riemannian Brownian motion initialized at $Z_n\sim\hat\mu_n$. Let $\widetilde X=Z_n+W^\parallel$ with
    $W^\parallel \sim \mathcal N\!\left(0,\sigma^2\mathbf I_{T_{Z_n}M} \right)$.
Under the synchronous coupling driven by ambient Brownian motion, if $0<\sigma\leq\frac{\tau}{\sqrt m}$, then
\begin{equation}
\label{eq:intrinsic_tangent_coupling}
    \|X- Z_n\|_{L^2}
    \leq
    \left(\sqrt{m} + \frac{2m\sigma}{\tau} \right) \sigma,\quad
    \|X-\widetilde X\|_{L^2}
    \leq
    \frac{2m}{\tau}\sigma^2.
\end{equation}
\end{lemma}

\begin{proof}
Condition on $Z_n=x\in M$. Let $P(y):\mathbb R^D\to T_yM$ denote the orthogonal projection onto the tangent space at $y$. The extrinsic representation of Riemannian Brownian motion with generator $\frac12\Delta_M$ is
\[
    X_t
    =
    x
    +
    \int_0^t P(X_s)\,dW_s
    +
    \frac12\int_0^t H(X_s)\,ds,
\]
where $W_t$ is standard Brownian motion in $\mathbb R^D$ and $H$ is the mean-curvature vector of the embedding. Adding and subtracting $P(x)W_t$ gives
\begin{equation}\label{eq:proof-Xt}
    X_t
    =
    x+P(x)W_t
    +
    \int_0^t\bigl(P(X_s)-P(x)\bigr)\,dW_s
    +
    \frac12\int_0^t H(X_s)\,ds.
\end{equation}

We first establish a short-time second-moment bound for $X_s-x$. Set
\[
    u(s):=\mathbb E\|X_s-x\|^2.
\]
Applying It\^o's formula to the function $q\mapsto\|q-x\|^2$ gives
\[
    u(s) = \int^s_0 \big[ m+ \mathbb E\bigl[ \langle X_r-x,H(X_r)\rangle \bigr]\big]dr,
\]
and thus
\[
    u'(s)
    =
    m+ \mathbb E\bigl[ \langle X_s-x,H(X_s)\rangle \bigr].
\]
A reach of $\tau$ implies $\|H(y)\|\leq\frac{m}{\tau}$ for all $y\in M$. It also implies that for $a,b\in M$ and every unit normal $n\in N_aM$,
\[
    |\langle b-a,n\rangle|
    \leq
    \frac{\|b-a\|^2}{2\tau}.
\]
Applying this inequality at $a=X_s$ with $n=H(X_s)/\|H(X_s)\|$ when $H(X_s)\neq0$ yields
\[
    \bigl|
        \langle X_s-x,H(X_s)\rangle
    \bigr|
    \leq
    \frac{m}{2\tau^2}\|X_s-x\|^2.
\]
Consequently,
\[
    u'(s)
    \leq
    m+\frac{m}{2\tau^2}u(s),
    \qquad
    u(0)=0.
\]
Gronwall's inequality gives
\[
    u(s)
    \leq
    2\tau^2
    \left(
        \exp\left(\frac{ms}{2\tau^2}\right)-1
    \right).
\]
For $0\leq s\leq\tau^2/m$, we have $\exp\left(\frac{ms}{2\tau^2}\right)-1 \leq \frac{ms}{\tau^2}$, and therefore
\begin{equation}
\label{eq:brownian_short_time_second_moment}
    \mathbb E\|X_s-x\|^2
    \leq
    2ms.
\end{equation}

Let $\theta_1,\ldots,\theta_m$ be the principal angles between
$T_{y_1}M$ and $T_{y_2}M$, and let $\theta_{\max}$ be their maximum.
The tangent-angle estimate for a manifold of reach at least $\tau$
\citep[Corollary~3]{boissonnat2019reach} gives
$\sin(\theta_{\max}/2)\leq\|y_1-y_2\|/(2\tau)$.
The principal-angle identity for the rank-$m$ orthogonal projections
$P(y_1)$ and $P(y_2)$ therefore yields
\[
    \|P(y_1)-P(y_2)\|_F^2
    =2\sum_{j=1}^m\sin^2\theta_j
    \leq 8m\sin^2(\theta_{\max}/2)
    \leq \frac{2m}{\tau^2}\|y_1-y_2\|^2,
    \qquad y_1,y_2\in M.
\]
Hence, by It\^o's isometry and \eqref{eq:brownian_short_time_second_moment},
\begin{align}
    \left\|
        \int_0^t
        \bigl(P(X_s)-P(x)\bigr)\,dW_s
    \right\|_{L^2}^2
    &=
    \int_0^t
        \mathbb E
        \left[
            \|P(X_s)-P(x)\|_F^2
        \right]ds
    \nonumber\\
    &\leq
    \frac{2m}{\tau^2}
    \int_0^t
        \mathbb E\|X_s-x\|^2\,ds
    \nonumber\\
    &\leq
    \frac{2m}{\tau^2}
    \int_0^t 2ms\,ds
    =
    \frac{2m^2}{\tau^2}t^2 \label{eq:ito_iso}.
\end{align}
For the mean-curvature term, Minkowski's integral inequality gives
\begin{equation}\label{eq:H-bound}
    \left\|
        \frac12\int_0^t H(X_s)\,ds
    \right\|_{L^2}
    \leq
    \frac12\int_0^t
        \|H(X_s)\|_{L^2}\,ds
    \leq
    \frac{m}{2\tau}t.
\end{equation}
Combining \eqref{eq:ito_iso} and \eqref{eq:H-bound} in \eqref{eq:proof-Xt}, we have
\[
    \|X_t- x\|_{L^2}
    \leq
    \|P(x) W_t\|_{L^2} +
    \left(\sqrt2+\frac12\right)
    \frac{m}{\tau}t
    \leq
    \sqrt{mt} + \frac{2m}{\tau} t
\]
where the last inequality uses $P(x)W_t \sim \mathcal N(0,t\mathbf I_{T_xM})$ and $\sqrt{2}+\frac{1}{2} \le 2$.

Define $\widetilde X_t:=x+P(x)W_t.$ Then,
\begin{equation}
\label{eq:heat_tangent_error_process}
    X_t-\widetilde X_t
    =
    \int_0^t\bigl(P(X_s)-P(x)\bigr)\,dW_s
    +
    \frac12\int_0^t H(X_s)\,ds.
\end{equation}
Combining \eqref{eq:ito_iso} and \eqref{eq:H-bound} in \eqref{eq:heat_tangent_error_process}, we obtain
\[
    \|X_t-\widetilde X_t\|_{L^2}
    \leq
    \left(\sqrt2+\frac12\right)
    \frac{m}{\tau}t
    \leq
    \frac{2m}{\tau}t.
\]
The condition $\sigma\leq\tau/\sqrt m$ ensures that $t=\sigma^2\leq\tau^2/m$, so the preceding short-time estimates apply. Setting $t=\sigma^2$ and then averaging over $Z_n\sim\hat\mu_n$ concludes the proof.
\end{proof}

We are now ready to prove \Cref{thm:decomposition}.
For completeness, we first restate Theorem \ref{thm:decomposition} with explicit constants and non-asymptotic residual terms before giving the proof.
\begingroup
\renewcommand{\thetheorem}{\ref{thm:decomposition}}
\begin{theorem}[Extrinsic $W_2^2$ Error Decomposition under Ambient Smoothing]
Let $M \subset \mathbb{R}^D$ be an $m$-dimensional compact submanifold with reach $\tau > 0$. Let $\mu$ be a smooth probability measure supported on $M$, and let $k_\sigma \ast \hat{\mu}_n$ denote the ambient Gaussian-smoothed empirical measure with bandwidth parameter $\sigma > 0$ satisfying $0 < \sigma < \tau/\sqrt{mD}$. Let $\mu_\sigma = k_{\sigma}^M \ast \hat{\mu}_n$ denote the intrinsic heat-smoothed measure on $M$.

Then, the squared 2-Wasserstein distance satisfies
\begin{equation*}
\begin{aligned}
    W_2^2(k_\sigma\ast\hat\mu_n,\mu)
    &\leq W_2^2(\mu_\sigma,\mu)
    +\sigma^2(D-m)
    +\frac{4m}{\tau}W_2(\hat\mu_n,\mu)\sigma^2\\
    &\quad
    +\frac{4m}{\tau}
        \bigl(\sqrt{D-m}+\sqrt m\bigr)\sigma^3
    +\frac{12m^2}{\tau^2}\sigma^4
    +C_{\mathrm{tail}},
\end{aligned}
\end{equation*}
where $C_{\mathrm{tail}}$ is the explicit nonnegative
tail bound defined in \eqref{eq:tail_bound_proof}. Under the
intrinsic assumptions of \Cref{thm:intrinsic_integrated},
\eqref{eq:ambient_combined_bound} follows.
\end{theorem}
\addtocounter{theorem}{-1}
\endgroup

\begin{proof}
Let $Z_n\sim\hat\mu_n$.  Conditionally on $Z_n=x\in M$, let
$(B_s)_{s\geq0}$ be a standard Brownian motion in $\mathbb R^D$ and
set
$W^\parallel:=P(x)B_{\sigma^2} \sim \mathcal N\!\left(     0,\sigma^2\mathbf I_{T_xM} \right),$
where $P(x)$ denotes the orthogonal projection onto $T_xM$.
Independently of the entire Brownian path $(B_s)_{s\geq0}$, sample
    $W^\perp
    \sim
    \mathcal N\!\left(
        0,\sigma^2\mathbf I_{N_xM}
    \right).$
Define
    $\widetilde X:=x+W^\parallel$ and
    $Z:=x+W^\parallel+W^\perp.$
Because $T_xM$ and $N_xM$ are orthogonal complementary subspaces and
$W^\parallel$ and $W^\perp$ are independent, conditionally on
$Z_n=x$,
    $W^\parallel+W^\perp \sim \mathcal N(0,\sigma^2\mathbf I_D).$
Consequently,
    $Z\sim k_\sigma\ast\hat\mu_n.$
Moreover, $\widetilde X$ is the orthogonal projection of $Z$ onto the
affine tangent plane $x+T_xM$:
\[
    \widetilde X
    =
    \pi_{x+T_xM}(Z)
    =
    x+P(x)(Z-x).
\]
Thus $(Z,\widetilde X)$ is the ambient-to-tangent coupling used in
\Cref{lem:projW}.

Using the same Brownian path $(B_s)_{s\geq0}$ that defines
$W^\parallel=P(x)B_{\sigma^2}$, construct the Riemannian Brownian
motion $(X_s)_{s\geq0}$ initialized at $X_0=x$ as in
\Cref{lem:intrinsic_one_step}, and set $X:=X_{\sigma^2}.$
Then $X\sim\mu_\sigma$, and \Cref{lem:intrinsic_one_step} gives
\[
    \|X-\widetilde X\|_{L^2} \leq \frac{2m}{\tau} \sigma^2.
\]
By construction, conditionally on $Z_n=x$, the normal Gaussian vector
$W^\perp$ is independent of $(\widetilde X,X)$.

Finally, let $\pi^\star$ be an optimal coupling of $\mu_\sigma$ and
$\mu$ for the ambient Euclidean cost.  Disintegrate it as
\[
    \pi^\star(dq,dy)
    =
    \mu_\sigma(dq)\,K(q,dy).
\]
After constructing $X$, sample $Y$ conditionally according to
$K(X,dy)$, using randomness independent of $W^\perp$.  Then
$Y\sim\mu$ and
\[
    \mathbb E\|X-Y\|^2
    =
    W_2^2(\mu_\sigma,\mu).
\]
Furthermore, conditionally on $Z_n=x$, $W^\perp$ is independent of
$(\widetilde X,X,Y)$ and satisfies
\[
    \mathbb E[W^\perp\mid Z_n=x]=0.
\]

Next, we decompose the total expectation over the event
$\mathcal{E}=\{\|W\|<\tau/\sqrt{m}\}$ and its complement
\begin{equation}\label{eq:total_expectation_split}
    \mathbb{E}[\|Z - Y\|^2] = \mathbb{E}[\|Z - Y\|^2 \mathbf{1}_{\mathcal{E}}] + \mathbb{E}[\|Z - Y\|^2 \mathbf{1}_{\mathcal{E}^c}].
\end{equation}
Since $\mathcal E^c=\{\|W\|\ge\tau/\sqrt{m}\}$, the inequality
$\|Z-Y\|^2\le2\|W\|^2+2\operatorname{diam}(M)^2$ and
\Cref{lemma:gaussian_tail} give
\begin{align}\label{eq:tail_bound_proof}
    \mathbb{E}[\|Z - Y\|^2 \mathbf{1}_{\mathcal{E}^c}] &\le 2 \mathbb{E}\left[ \|W\|^2 \mathbf{1}(\|W\| \ge \tau/\sqrt{m}) \right] + 2 \operatorname{diam}(M)^2 \mathbb{P}\left( \|W\| \ge \tau/\sqrt{m} \right) \nonumber \\
    &\le 2\sqrt{3}\sigma^2 D \exp\left( -\frac{(\tau/\sqrt{m} - \sigma\sqrt{D})^2}{4\sigma^2} \right)\nonumber \\
    &\hspace{2cm}+ 2\operatorname{diam}(M)^2 \exp\left( -\frac{(\tau/\sqrt{m} - \sigma\sqrt{D})^2}{2\sigma^2} \right) \nonumber \\
    &= C_{\mathrm{tail}}(\tau,\sigma,m,D,M).
\end{align}
On $\mathcal E$, we expand the squared Euclidean distance into three direct
terms and three interaction terms
\begin{align}\label{eq:exact_vector_expansion}
    \|Z - Y\|^2 \mathbf{1}_{\mathcal{E}} &= \left( \|Z - \tilde{X}\|^2 + \|\tilde{X} - {X}\|^2 + \|{X} - Y\|^2 \right) \mathbf{1}_{\mathcal{E}} \nonumber \\
    &\quad + \underbrace{2\langle Z - \tilde{X}, \, \tilde{X} - {X} \rangle \mathbf{1}_{\mathcal{E}}}_{T_1 \mathbf{1}_{\mathcal{E}}} + \underbrace{2\langle \tilde{X} - {X}, \, {X} - Y \rangle \mathbf{1}_{\mathcal{E}}}_{T_2 \mathbf{1}_{\mathcal{E}}} + \underbrace{2\langle Z - \tilde{X}, \, {X} - Y \rangle \mathbf{1}_{\mathcal{E}}}_{T_3 \mathbf{1}_{\mathcal{E}}}.
\end{align}
Taking expectations over $\mathcal{E}$, the three squared distance terms are bounded as follows
\begin{align*}
    \mathbb{E}[\|Z - \tilde{X}\|^2 \mathbf{1}_{\mathcal{E}}] &\le \sigma^2(D-m)  \quad \text{(\Cref{lem:projW})}, \\
    \mathbb{E}[\|\tilde{X} - {X}\|^2 \mathbf{1}_{\mathcal{E}}]
    &\le \frac{4m^2}{\tau^2}\sigma^4
    \quad \text{(\Cref{lem:intrinsic_one_step})}, \\
    \mathbb{E}[\|{X} - Y\|^2 \mathbf{1}_{\mathcal{E}}]
    &\le W_{2}^2(\mu_\sigma, \mu).
\end{align*}
Using Cauchy--Schwarz and the $L^2$ bound from
\Cref{lem:intrinsic_one_step} gives
\begin{equation}\label{eq:T1_bound}
    \left|\mathbb E[T_1\mathbf 1_{\mathcal E}]\right|
    \leq
    2\|Z-\widetilde X\|_{L^2}
     \|\widetilde X-X\|_{L^2}
    \leq
    \frac{4m\sqrt{D-m}}{\tau}\sigma^3.
\end{equation}
For the second cross-term $T_2$, applying Cauchy--Schwarz  gives
\begin{equation}\label{eq:T2_bound}
    \left| \mathbb{E}[T_2 \mathbf{1}_{\mathcal{E}}] \right| \le 2 \sqrt{\mathbb{E}[\|\tilde{X} - {X}\|^2]} \sqrt{\mathbb{E}[\|{X} - Y\|^2]} \le \frac{4 m}{\tau} \sigma^2 W_{2}(\mu_\sigma, \mu).
\end{equation}
By the triangle inequality and \Cref{lem:intrinsic_one_step},
\[
    W_2(\mu_\sigma,\mu)
    \leq W_2(\hat\mu_n,\mu)
    +\sqrt m\,\sigma+\frac{2m}{\tau}\sigma^2.
\]
Substituting into \eqref{eq:T2_bound} yields
\[
    \left|\mathbb E[T_2\mathbf 1_{\mathcal E}]\right|
    \leq
    \frac{4m}{\tau}W_2(\hat\mu_n,\mu)\sigma^2
    +\frac{4m\sqrt m}{\tau}\sigma^3
    +\frac{8m^2}{\tau^2}\sigma^4.
\]
For the third cross-term, condition on
$(Z_n,W^\parallel,X,Y)$. The normal noise $W^\perp$ remains
a centered Gaussian under this conditioning. Writing
$W=W^\parallel+W^\perp$, orthogonality gives
$\mathcal E
=
\left\{
    \|W^\parallel\|^2+\|W^\perp\|^2<\tau^2/m
\right\}$.
Gaussian symmetry yields
\[
    \mathbb E\!\left[
        W^\perp\mathbf 1_{\mathcal E}
        \,\middle|\,Z_n,W^\parallel,X,Y
    \right]=0,
\]
and therefore
$\mathbb E[T_3\mathbf 1_{\mathcal E}]=0$.
Summing the bounds for the direct terms, cross terms,
and tail contribution establishes \eqref{eq:ambient_bound_main}.
For the combined bound, use
$W_2(\mu_\sigma,\mu)\leq W_{2,M}(\mu_\sigma,\mu)$
and substitute \Cref{eq:integrated_w2_bound}. The quadratic terms
$(D-m)\sigma^2$ and $B_{\mathrm{int}}\sigma^2$ add to
$(D+\kappa R^2/2)\sigma^2$, giving
\eqref{eq:ambient_combined_bound}.
\end{proof}

The bound retains the realized empirical error
$W_2(\hat\mu_n,\mu)$ explicitly and does not require an
empirical convergence-rate estimate. A bound in expectation
or with high probability can be obtained by combining this
result with an empirical Wasserstein estimate valid under
the relevant dimensional and distributional assumptions.

\subsection{Proof of the Latent-Space Smoothing Bound}
\label{sec:latent_smoothing}

The combined ambient bound contains
$(D+\kappa R^2/2)\sigma^2=B_{\mathrm{int}}\sigma^2+(D-m)\sigma^2$,
together with geometric corrections. The decomposition distinguishes the intrinsic contribution
from the additional normal-noise contribution.

Latent smoothing introduces noise in $\mathbb R^d$ and maps it back through
the decoder. Its explicit orthogonal contribution is
$L_{\mathrm{orth}}^2(d-m)\sigma^2$, while tangential gain, isometry defect,
reconstruction error, and geometry also enter the quadratic coefficient.
Thus the comparison uses the full coefficients
$\beta_{\mathrm{lat}}$ and $\beta_{\mathrm{amb}}$, rather than codimensions
alone. The decoder bounds quantify when reducing the noise dimension
can yield a tighter optimized surrogate.

For completeness, we restate Theorem \ref{thm:latent_decomposition} with all explicit geometric conditions, local Jacobian bounds, and non-asymptotic residual terms before presenting the proof.


\begingroup
\renewcommand{\thetheorem}{\ref{thm:latent_decomposition}}
\renewcommand{\theHtheorem}{latent-restatement}
\begin{theorem}[Latent-space smoothing bound]
\label{thm:latent_decomposition_effective_gain}
Let $M$ be a compact, connected, smooth, boundaryless $m$-dimensional
Riemannian manifold embedded in $\mathbb R^D$ with reach $\tau>0$.
Choose $\kappa\geq0$ with $\operatorname{Ric}_M\geq-\kappa g_M$,
and use the intrinsic constants from \Cref{thm:intrinsic_integrated}.
Let $\mu=\rho\,\vol_M$ have smooth, strictly positive density
bounded above by $\rho_{\max}$, and let
$\hat\mu_n=n^{-1}\sum_{i=1}^n\delta_{x_i}$.

Let $f:\mathbb R^D\to\mathbb R^d$ and
$g:\mathbb R^d\to\mathbb R^D$, where $m\leq d\leq D$, and define
$\delta$ as in \Cref{eq:def-delta}. Assume that $f$ is differentiable
at each $x_i$ and that
$g\in C^2(\mathbb R^d;\mathbb R^D)$ has uniformly bounded first and
second derivatives. The finite constant $C_g$ is defined in
\eqref{eq:decoder_uniform_C2_constant}.
For each $i$, assume that
\[
    U_i:=\left.J_f(x_i)\right|_{T_{x_i}M}:T_{x_i}M\to V_i
\]
is a linear isometry onto an $m$-dimensional subspace
$V_i\subset\mathbb R^d$, and define
\[
    A_i
    :=
    \left.J_g(f(x_i))J_f(x_i)\right|_{T_{x_i}M}
    :T_{x_i}M\to\mathbb R^D.
\]
Assume that there exist a global tangential gain
$\rho_\parallel>0$ and an effective isometry defect
$\varepsilon_{\mathrm{iso}}\geq0$ such that
\begin{equation}
\label{eq:effective_gain_conformal_assumption}
    \varepsilon_{\mathrm{iso}}
    =
    \max_{1\leq i\leq n}
    \min_{O\in O(T_{x_i}M)}
    \left\|A_i-\rho_\parallel\iota_iO\right\|_{\mathrm{op}},
\end{equation}
where $\iota_i:T_{x_i}M\hookrightarrow\mathbb R^D$ is the inclusion
map. Let $O_i$ be a minimizing orthogonal map for each $i$. Assume also
that
\[
    \|J_g(f(x_i))v^\perp\|
    \leq
    L_{\mathrm{orth}}\|v^\perp\|,
    \qquad v^\perp\in V_i^\perp.
\]
Define
    $\mu_{\mathrm{out}}
    :=
    g_\#\!\left(
        (f_\#\hat\mu_n)\ast
        \mathcal N(0,\sigma^2\mathbf I_d)
    \right)$
and
    $W_0:=W_{2,M}(\hat\mu_n,\mu).$
For sufficiently small $\sigma>0$ such that
    $0<\sigma
    <
    \frac{1}{2\sqrt m\,\rho_\parallel}
    \min\{W_0,2\tau\},$
one has
\begin{equation}
\label{eq:effective_gain_final_bound}
    W_2^2(\mu_{\mathrm{out}},\mu)
    \leq
    \mathcal E_{\mathrm{const}}
    -\Gamma_{\mathrm{eff}}\sigma
    +L_{\mathrm{latent}}\sigma^2
    +\mathcal O(\sigma^3),
\end{equation}
where
\begin{align*}
    \mathcal E_{\mathrm{const}}
    &:=(1+\varepsilon_{\mathrm{iso}})W_0^2
       +2\delta W_0+\delta^2,\\
    \Gamma_{\mathrm{eff}}
    &:=\rho_\parallel C_{\mathrm{int}}n^{-1/m}
       \left(1+\frac{\delta}{W_0} \right),\\
    L_{\mathrm{latent}}
    &:=L_{\mathrm{orth}}^2(d-m)
       +m\varepsilon_{\mathrm{iso}}^2
       +(\rho_\parallel^2B_{\mathrm{int}}+m)\varepsilon_{\mathrm{iso}}
       +Q.
\end{align*}
Here $C_{\mathrm{int}}$ is the constant from
\Cref{thm:intrinsic_integrated}, and $Q$ is defined in
\eqref{eq:latent_Q_main}. The remainder estimate, including its dependence
on $d$, is made explicit in the proof below.
\end{theorem}
\addtocounter{theorem}{-1}
\endgroup

\begin{proof}
Set $\gamma_n:=C_{\mathrm{int}}n^{-1/m}$, with $Q$ defined in
\eqref{eq:decoder_uniform_C2_constant} and \eqref{eq:latent_Q_main} and $C_{\mathrm{geom}}:=\frac{2m}{\tau}$.
Let
\[
    M_i:=A_i-\rho_\parallel\iota_iO_i,
    \qquad
    \|M_i\|_{\mathrm{op}}\leq\varepsilon_{\mathrm{iso}},
\]
and set
\[
    r:=\rho_\parallel\sigma,
    \qquad
    \mu_r:=k_r^M\ast\hat\mu_n,
    \qquad
    W_r:=W_{2,M}(\mu_r,\mu).
\]

We define a joint coupling to estimate
$W_2^2(\mu_{\mathrm{out}},\mu)$. Let
$X_n\sim\hat\mu_n$, so that $X_n=x_i$ with probability $1/n$.
Conditionally on $X_n=x_i$, sample
    $\xi\sim\mathcal N(0,\mathbf I_d)$
and decompose it orthogonally as
\[
    \xi=\xi^\parallel+\xi^\perp,
    \qquad
    \xi^\parallel\in V_i,
    \qquad
    \xi^\perp\in V_i^\perp.
\]
Since $U_i:T_{x_i}M\to V_i$ is an isometry,
\[
    W^\parallel:=U_i^{-1}\xi^\parallel
    \sim\mathcal N(0,\mathbf I_{T_{x_i}M}).
\]
The rotated vector
    $v:=O_iW^\parallel$
has the same distribution.

By \Cref{lem:intrinsic_one_step}, we may couple the flat tangent step
$x_i+r\iota_iv$ to a point $z_i$ with conditional law
$k_r^M(x_i,\cdot)$ so that
\begin{equation}
\label{eq:effective_gain_one_step_L2}
    \left\|
        z_i-(x_i+r\iota_iv)
    \right\|_{L^2}
    \leq
    C_{\mathrm{geom}}r^2
\end{equation}
After averaging over $i$, the variable $z_i$ has law $\mu_r$.
Disintegrate an optimal coupling for the intrinsic cost between
$\mu_r$ and $\mu$, and conditionally on $z_i$ sample $y$ from this
coupling. Then
\[
    \mathbb E\big[d_M^2(z_i,y)\big]=W_r^2
\]
and, since the ambient Euclidean distance is bounded by the intrinsic
distance on $M$,
\begin{equation}
\label{eq:effective_gain_T1_L2}
    \|y-z_i\|_{L^2}\leq W_r.
\end{equation}
Conditionally on $X_n=x_i$, the orthogonal noise $\xi^\perp$ is
independent of $(z_i,y)$.

Let
\[
    Z_{\mathrm{out}}
    :=g(f(x_i)+\sigma\xi).
\]
We use the following four-term decomposition:
\begin{align}
\label{eq:effective_gain_exact_expansion}
    y-Z_{\mathrm{out}}
    ={}&
    (y-z_i)
    +\bigl(z_i-(x_i+r\iota_iv)\bigr)\nonumber\\
    &+\bigl((x_i+r\iota_iv)
        -(g(f(x_i))+r\iota_iv)\bigr)\nonumber\\
    &+\bigl((g(f(x_i))+r\iota_iv)
        -g(f(x_i)+\sigma\xi)\bigr)\nonumber\\
    =:{}&T_1+T_2+T_3+T_4.
\end{align}
Writing
\[
    d_i:=g(f(x_i))-x_i,
\]
the third term is simply
\[
    T_3=-d_i,
    \qquad
    \mathbb E\|T_3\|^2=\delta^2.
\]

Taylor's theorem and \Cref{eq:decoder_uniform_C2_constant} give
\[
    g(f(x_i)+\sigma\xi)
    =g(f(x_i))+\sigma J_g(f(x_i))\xi+R_g(\sigma\xi),
\]
where
\begin{equation}
\label{eq:decoder_taylor_remainder_global}
    \|R_g(\sigma\xi)\|
    \leq
    \frac{C_g}{2}\sigma^2\|\xi\|^2.
\end{equation}
Since $\xi^\parallel=U_iW^\parallel$,
\[
    J_g(f(x_i))\xi^\parallel
    =A_iW^\parallel
    =\rho_\parallel\iota_iv+M_iW^\parallel.
\]
Consequently,
\begin{equation}
\label{eq:effective_gain_T4_expansion}
    T_4
    =
    -\sigma M_iW^\parallel
    -\sigma J_g(f(x_i))\xi^\perp
    -R_g(\sigma\xi).
\end{equation}

We first estimate the four squared terms. From
\Cref{eq:effective_gain_T1_L2,eq:effective_gain_one_step_L2},
\[
    \mathbb E\|T_1\|^2\leq W_r^2,
    \qquad
    \mathbb E\|T_2\|^2
    \leq C_{\mathrm{geom}}^2r^4,
    \qquad
    \mathbb E\|T_3\|^2=\delta^2.
\]
The two linear terms in \Cref{eq:effective_gain_T4_expansion} are
orthogonal in expectation because $W^\parallel$ and $\xi^\perp$ are
independent and centered conditionally on $X_n=x_i$. Moreover,
\[
    \mathbb E\|M_iW^\parallel\|^2
    \leq
    \varepsilon_{\mathrm{iso}}^2m,
    \qquad
    \mathbb E\|J_g(f(x_i))\xi^\perp\|^2
    \leq
    L_{\mathrm{orth}}^2(d-m).
\]
The Gaussian identity
$\mathbb E\|\xi\|^4=d(d+2)$ and
\Cref{eq:decoder_taylor_remainder_global} imply
    \begin{equation}\label{eq:G_2}
        \|R_g(\sigma\xi)\|_{L^2}
        \leq 
        \frac{C_g}{2}\sigma^2\sqrt{d(d+2)} 
        =:
        G_2\sigma^2.
    \end{equation}
It follows that
\begin{equation}
\label{eq:effective_gain_T4_square}
    \mathbb E\|T_4\|^2
    \leq
    \left(
        \varepsilon_{\mathrm{iso}}^2m
        +L_{\mathrm{orth}}^2(d-m)
    \right)\sigma^2
    +\mathcal O(\sigma^3).
\end{equation}

We next estimate the six cross terms. By Cauchy--Schwarz,
\begin{align*}
    2\mathbb E\langle T_1,T_2\rangle
    &\leq
    2C_{\mathrm{geom}}r^2W_r,\quad
    2\mathbb E\langle T_1,T_3\rangle
    \leq
    2\delta W_r.
\end{align*}
The component of $2\mathbb E\langle T_1,T_4\rangle$ containing
$\xi^\perp$ vanishes by conditional independence and centering. For
the tangential defect, Cauchy--Schwarz and Young's inequality give
\begin{align}
\label{eq:effective_gain_defect_cross}
    2\mathbb E
    \left\langle
        y-z_i,-\sigma M_iW^\parallel
    \right\rangle
    &\leq
    2\varepsilon_{\mathrm{iso}}\sqrt m\,\sigma W_r
    \leq
    \varepsilon_{\mathrm{iso}}W_r^2
    +m\varepsilon_{\mathrm{iso}}\sigma^2.
\end{align}
Together with \eqref{eq:G_2}, the Taylor remainder satisfies
\[
    2\mathbb E
    \left\langle
        y-z_i,-R_g(\sigma\xi)
    \right\rangle
    \leq
    2G_2\sigma^2W_r.
\]
The remaining cross terms satisfy
\begin{align*}
    2\mathbb E\langle T_2,T_3\rangle
    &\leq
    2\delta C_{\mathrm{geom}}r^2,\\
    2\mathbb E\langle T_2,T_4\rangle
    &=\mathcal O(\sigma^3),\\
    2\mathbb E\langle T_3,T_4\rangle
    &\leq
    2\delta G_2\sigma^2.
\end{align*}
For the last line, the two linear components of $T_4$ have zero
conditional mean given $X_n=x_i$, so only the Taylor remainder
contributes.

Combining the direct and cross-term estimates yields
\begin{align}
\label{eq:effective_gain_pre_intrinsic}
    W_2^2(\mu_{\mathrm{out}},\mu)
    \leq{}&
    (1+\varepsilon_{\mathrm{iso}})W_r^2
    +2\delta W_r+\delta^2\nonumber\\
    &+\Bigl[
        \varepsilon_{\mathrm{iso}}^2m
        +m\varepsilon_{\mathrm{iso}}
        +L_{\mathrm{orth}}^2(d-m)
    \Bigr]\sigma^2\nonumber\\
    &+2\left(
        C_{\mathrm{geom}}\rho_\parallel^2+G_2
    \right)(W_r+\delta)\sigma^2
    +\mathcal O(\sigma^3).
\end{align}

We now use the intrinsic estimate. By
\Cref{thm:intrinsic_integrated},
\begin{equation}
\label{eq:effective_gain_intrinsic_integrated}
    W_r^2
    \leq
    {W_0^2-\gamma_nr+B_{\mathrm{int}}r^2}.
\end{equation}
Because $W_0>0$, concavity of the square-root function gives
\begin{equation}
\label{eq:effective_gain_W_expansion}
    W_r
    \leq
    {W_0-\frac{\gamma_n}{2W_0}r
    +\frac{B_{\mathrm{int}}}{2W_0}r^2}.
\end{equation}
Compactness and the short-time second-moment estimate for Brownian motion
give a constant $H_M<\infty$, independent of $d$, such that
$W_{2,M}(\mu_r,\hat\mu_n)\leq H_Mr$ for all sufficiently small $r$.
Consequently $|W_r-W_0|\leq H_Mr$. This also makes the remainder
explicit. Write
\[
 A:=\sqrt{m\varepsilon_{\mathrm{iso}}^2
           +L_{\mathrm{orth}}^2(d-m)},
 \qquad B:=C_{\mathrm{geom}}\rho_\parallel^2+G_2,
\]
where $G_2$ is defined in \eqref{eq:G_2}.
The discarded terms in the square of $T_4$ are at most
$2AG_2\sigma^3+G_2^2\sigma^4$; those from $T_2^2$ and the
$T_2,T_4$ cross term are at most
$2C_{\mathrm{geom}}\rho_\parallel^2A\sigma^3
 +(C_{\mathrm{geom}}^2\rho_\parallel^4
   +2C_{\mathrm{geom}}\rho_\parallel^2G_2)\sigma^4$.
Replacing $W_r$ by $W_0$ in the term already multiplied by $\sigma^2$
adds at most $2BH_M\rho_\parallel\sigma^3$. Hence the full
remainder is bounded above by
\begin{equation}\label{eq:latent_remainder_explicit}
 2B(A+H_M\rho_\parallel)\sigma^3+B^2\sigma^4.
\end{equation}
In particular, $G_2=(C_g/2)\sqrt{d(d+2)}$, so the cubic coefficient
can grow with the latent dimension; the $\mathcal O(\sigma^3)$ in the
theorem is uniform in $\sigma$ for fixed $d$, not uniform in $d$.
Substituting $r=\rho_\parallel\sigma$ into
\Cref{eq:effective_gain_pre_intrinsic} and using the preceding bounds
initially gives the favorable linear coefficient
$-\rho_\parallel\gamma_n(1+\varepsilon_{\mathrm{iso}}
+\delta/W_0)$. Discarding the additional nonpositive term
$-\rho_\parallel\gamma_n\varepsilon_{\mathrm{iso}}\sigma$ yields the
slightly weaker but simpler bound
\begin{align*}
    W_2^2(\mu_{\mathrm{out}},\mu)
    \leq{}&
    (1+\varepsilon_{\mathrm{iso}})W_0^2
    +2\delta W_0+\delta^2
    -
    \rho_\parallel\gamma_n
    \left(1+\frac{\delta}{W_0} \right)\sigma\\
    &+\Bigg[
        L_{\mathrm{orth}}^2(d-m)
        +\varepsilon_{\mathrm{iso}}^2m
        +\rho_\parallel^2 B_{\mathrm{int}}
        \left(
            1+\varepsilon_{\mathrm{iso}}
            +\frac{\delta}{W_0}
        \right)\\
    &\hspace{1.5cm}
        +m\varepsilon_{\mathrm{iso}}
        +2\left(
            C_{\mathrm{geom}}\rho_\parallel^2+G_2
        \right)(W_0+\delta)
    \Bigg]\sigma^2
    +\mathcal O(\sigma^3).
\end{align*}
The quadratic bracket is
$L_{\mathrm{orth}}^2(d-m)+m\varepsilon_{\mathrm{iso}}^2
+(\rho_\parallel^2B_{\mathrm{int}}+m)\varepsilon_{\mathrm{iso}}+Q$
by \eqref{eq:latent_Q_main}. Recalling
$\gamma_n=C_{\mathrm{int}}n^{-1/m}$ and the definitions of
$\mathcal E_{\mathrm{const}}$, $\Gamma_{\mathrm{eff}}$, and
$L_{\mathrm{latent}}$ proves
\Cref{eq:effective_gain_final_bound}.

Because the first and second derivatives of $g$ are uniformly bounded
on all of $\mathbb R^d$, the Taylor estimate
\Cref{eq:decoder_taylor_remainder_global} holds for every Gaussian
realization. Thus no truncation event or separate Gaussian tail term is
needed under the assumptions of the theorem.
\end{proof}

\section{Numerical Experiments}
\label{sec:experiments}
We conduct numerical experiments on synthetic manifolds to examine the trends suggested by \Cref{thm:decomposition,thm:intrinsic_integrated,thm:latent_decomposition}. We compare positively curved round spheres with a product of spheres that has flat mixed directions, studying ambient-dimensional effects and encoder--decoder fidelity under the same intrinsic heat-flow input. The theory guarantees its intrinsic decrease only below a manifold-
and sample-dependent bandwidth threshold. The entropy asymptotic used
in the proof does not certify that every plotted bandwidth is below it;
the experiments therefore examine trends rather than verify the local
inequality pointwise.
Code related to the numerical experiments is available at
\url{https://github.com/wonjunee/intrinsic_smoothing}.

\subsection{Ambient Smoothing on Round and Product Manifolds}
We examine how the ambient smoothing error $W_2^2(k_\sigma * \hat{\mu}_n,\mu)$ and the empirically selected bandwidth $\hat\sigma_{\mathrm{amb}}$ depend on sample size $n$ and ambient dimension $D$. The minimizer of the quadratic upper-bound surrogate in \eqref{eq:psi_amb} is
\[
\sigma^*_{\mathrm{amb}}=\frac{\alpha_n}{2\beta_{\mathrm{amb}}},\qquad
\alpha_n=C_{\mathrm{int}}n^{-1/m},\qquad
\beta_{\mathrm{amb}}=D+\frac{\kappa R^2}{2}+\frac{4m}{\tau}W_2(\hat\mu_n,\mu).
\]
For fixed manifold and sample, this surrogate has a $D^{-1}$ large-$D$ trend. Its $n^{-1/m}$ factor suggests a sample-size reference slope when the remaining coefficient varies little with $n$. These are comparisons for the surrogate, not exact formulas for the empirical minimizer.

We use the positively curved round sphere $\mathbb{S}^4$ and the product $\mathbb{S}^1\times\mathbb{S}^2$, whose mixed tangent planes have zero sectional curvature. The product provides a compact smooth example beyond strict positive sectional curvature. The lower-Ricci-bound theory also permits negative curvature, but the present experiments do not examine it.
In all these examples, the corresponding target measure $\mu$ is taken to be the uniform distribution supported on the manifold.
For each manifold tested, we compute the empirical error $W_2^2(k_\sigma *\hat{\mu}_n, \mu)$ over 81 quadratically spaced bandwidths $\sigma \in [0, 0.6]$ with
\[
\sigma_p = 0.6 \left(
\frac{p}{80}
\right)^2, ~ p = 0, 1,\dotsc, 80
\]
over five trials for each parameter pair $(n,D)$.

\subsection*{Optimal transport computation via discretization of marginal distributions}
For positive bandwidths, we approximate
$W_2^2(k_\sigma * \hat{\mu}_n,\mu)$ using
$N_{\mathrm{eval}}=5000$ samples from each marginal. At $\sigma=0$,
the source is the original empirical measure $\hat\mu_n$, with its
$n$ anchors carrying weights $1/n$; only the target $\mu$ is
approximated by $N_{\mathrm{eval}}$ samples. Writing
$A=k_\sigma * \hat{\mu}_n$, let $\widehat A$ denote the source
measure used in the computation and
$\widehat\mu_{N_{\mathrm{eval}}}$ the empirical target measure.
The reported $\widehat W_2^2$ is their discrete
Kantorovich cost, computed using the Python POT solver
\texttt{ot.emd2} \citep{flamary2021pot}. To relate this empirical comparison to the
underlying distributions, the triangle inequality gives
\[
    \left|
    W_2(\widehat A,
        \widehat\mu_{N_{\mathrm{eval}}})
    -W_2(A,\mu)
    \right|
    \leq
    W_2(\widehat A,A)
    +W_2(\widehat\mu_{N_{\mathrm{eval}}},\mu).
\]
At $\sigma=0$, the first error term vanishes. For scale comparison,
$N_{\mathrm{eval}}^{-1/m}$ is approximately $0.058$ for $m=3$
and $0.119$ for $m=4$. These nominal intrinsic scales motivate
our evaluation size, although they do not provide a quantitative
accuracy guarantee: the actual errors depend on the distributions,
and intrinsic-dimensional scaling need not apply to the
ambient-smoothed marginal.

The following explains how we sample the marginal distributions for each independent trial.
\begin{itemize}
    \item
To discretize $\mu$ on $M=\mathbb{S}^4$, we sample $N_{\mathrm{eval}}$ iid standard Gaussian vectors in $\mathbb{R}^5$ and normalize each to unit norm.
For $M = \mathbb{S}^1\times \mathbb{S}^2$, we independently sample
\[
\{ a_i\}_{i=1}^{N_{\mathrm{eval}}} \overset{iid}{\sim}
\operatorname{Unif}(\mathbb{S}^1),\quad \{ b_i\}_{i=1}^{N_{\mathrm{eval}}} \overset{iid}{\sim} \operatorname{Unif}(\mathbb{S}^2)
\]
and use
$
\{ (a_i,b_i) \}_{i=1}^{N_{\mathrm{eval}}} \subset \mathbb{R}^5
$
as iid samples from $\mu = \operatorname{Unif}(M)$. We then embed these low-dimensional samples into $\mathbb{R}^D$ by zero-padding the remaining coordinates, denoting the resulting samples by $\{z_j\}_{j=1}^{N_{\mathrm{eval}}}\overset{iid}{\sim} \mu$. The discrete approximation of $\mu$ is then
\[
\hat{\mu} = \frac{1}{N_{\mathrm{eval}}}\sum_{j=1}^{N_{\mathrm{eval}}} \delta_{z_j}.
\]
\item
To discretize $k_\sigma *\hat{\mu}_n$ at different bandwidth $\sigma_p$, we first draw $n$ samples $\{x_i\}_{i=1}^{n} \overset{iid}{\sim} \mu$, embed them into $\mathbb{R}^D$ by zero-padding the remaining coordinates, and form $\hat{\mu}_n = \frac{1}{n}\sum_{i=1}^n \delta_{x_i}$.
For each empirical measure $\hat{\mu}_n$, we independently draw
\[
\{i(j)\}_{j=1}^{N_{\mathrm{eval}}} \overset{iid}{\sim}\operatorname{Unif}\{ 1,2,\dotsc,n\}, \quad \{\eta_j\}_{j=1}^{N_{\mathrm{eval}}}\overset{iid}{\sim} \mathcal{N}(0, \textbf{I}_D)
\]
and reuse these draws across positive bandwidths $\sigma_p \in (0,0.6]$ to reduce Monte Carlo variability in the resulting $W_2^2$ curves.
For each positive bandwidth $\sigma_p$,
we approximate $k_{\sigma_p} * \hat{\mu}_n$ by
\[
\hat{\nu}_{\sigma_p} = \frac{1}{N_{\mathrm{eval}}}\sum_{j=1}^{N_{\mathrm{eval}}} \delta_{y_j^{(p)}}, \quad \text{where}~ y_j^{(p)} = x_{i(j)} + \sigma_p \eta_j, ~~\forall j=1,2,\dotsc,N_{\mathrm{eval}}.
\]
At $\sigma_0=0$, we instead use $\hat\nu_{\sigma_0}=\hat\mu_n$
with its original weights $1/n$.

\end{itemize}
At each grid point $\sigma_p$, we estimate the transport error by
$W_2^2(\hat\nu_{\sigma_p},\hat\mu)$, computed with
\texttt{ot.emd2}.

\subsubsection{Impact of Ambient Dimension \texorpdfstring{$(D)$}{(D)}}
\label{sec:impact of ambient D}


To isolate the impact of ambient dimension $D$, we fix the sample size at $n = 50$ and vary $D$. We examine the mean $W_2^2(k_\sigma*\hat{\mu}_n, \mu)$ error curves as functions of $\sigma $ and the scaling of $\sigma^*_{\mathrm{amb}}$ with respect to $D$.

For each of the five trials, we independently draw an empirical measure $\hat{\mu}_n$ and an $N_{\mathrm{eval}}$-point discretization of $\mu$. Both point clouds are held fixed across all $D$ within each trial and embedded into $\mathbb{R}^D$ by zero-padding. For each $D$, we separately discretize $k_\sigma *\hat{\mu}_n$, following the procedure described above.

Let $C_{t,D}(\sigma)$ denote the computed $W_2^2(k_\sigma *\hat{\mu}_n, \mu)$ errors in trial $t$ at ambient dimension $D$.
We define the empirical optimal bandwidth for the mean $W_2^2$ curve by
\[
\hat{\sigma}_{\mathrm{amb}}(D) \in \underset{\sigma\in\{\sigma_0,\ldots,\sigma_{80}\}}{\arg\min}
\frac{1}{5}\sum_{t = 1}^5
C_{t,D}(\sigma)
\] as a numerical proxy for $\sigma^*_{\mathrm{amb}}(D)$.

For each manifold tested, we use two collections of ambient dimensions, $L_1$ and $L_2$, for different purposes: $L_1$ illustrates how $W_2^2(k_{\sigma}*\hat{\mu}_n, \mu)$ varies with bandwidth $\sigma$, as well as how the optimal bandwidth and minimum $W_2^2$ error vary with $D$, whereas $L_2$ is used to estimate the slope of $\log(\sigma^*_{\mathrm{amb}})$ on $\log(D)$. Specifically,
\[
L_1 := \{6,8,10,12,14\},\quad \text{and}\quad L_2 := \{10\cdot 2^k: k=0,1,\dotsc,5\}.
\]
For fixed $n$, $\beta_{\mathrm{amb}}=D+\mathcal O(1)$ as $D$ grows, so the surrogate bandwidth has reference log-log slope $-1$. We compare the empirical slopes with this reference.

First, for $D\in L_1$, we plot the mean curves of $W_2^2(k_{\sigma}*\hat{\mu}_n, \mu) $ versus $\sigma$, with $\pm 1$-standard deviation band among 5 trials. Although the experiments cover $\sigma \in [0, 0.6]$, we restrict the displayed curves to $\sigma \in [0, 0.2]$ or $[0, 0.3]$ for better visualization.
Figure~\ref{fig:5-1-1_dimension_combined} (top row) compares the round sphere with the product manifold.
For clarity,
we display the mean $W_2^2(k_\sigma * \hat{\mu}_n, \mu)$ curves over $\sigma \in [0,0.2]$ for $\mathbb{S}^4$, and over $\sigma \in [0, 0.3]$ for $\mathbb{S}^1\times \mathbb{S}^2$. 
For comparison, reference dashed lines indicate the mean unsmoothed  errors $W_2^2(\hat{\mu}_n,\mu)$ and $W_{2,M}^2(\hat{\mu}_n,\mu)$.
The results suggest that for a fixed $n$, increasing $D$ shifts the empirical optimal bandwidth toward zero while increasing the minimum achievable $W_2^2$ error. Thus, ambient smoothing becomes progressively less effective at reducing $W_2^2(k_\sigma *\hat{\mu}_n, \mu)$ as $D$ increases.

Second, for $D\in L_2$, we plot $\hat{\sigma}_{\mathrm{amb}}$ against $D$ on log-log axes and regress $\log(\hat{\sigma}_{\mathrm{amb}})$ on $\log(D)$. Figure~\ref{fig:5-1-1_dimension_combined} (bottom row) gives fitted slopes of $-1.0494$ for $\mathbb{S}^4$ and $-0.9711$ for $\mathbb{S}^1\times\mathbb{S}^2$, both close to the surrogate reference slope $-1$.

\begin{figure}[t]
    \centering

    \begin{subfigure}[t]{0.38\textwidth}
        \centering
        \includegraphics[width=\linewidth]{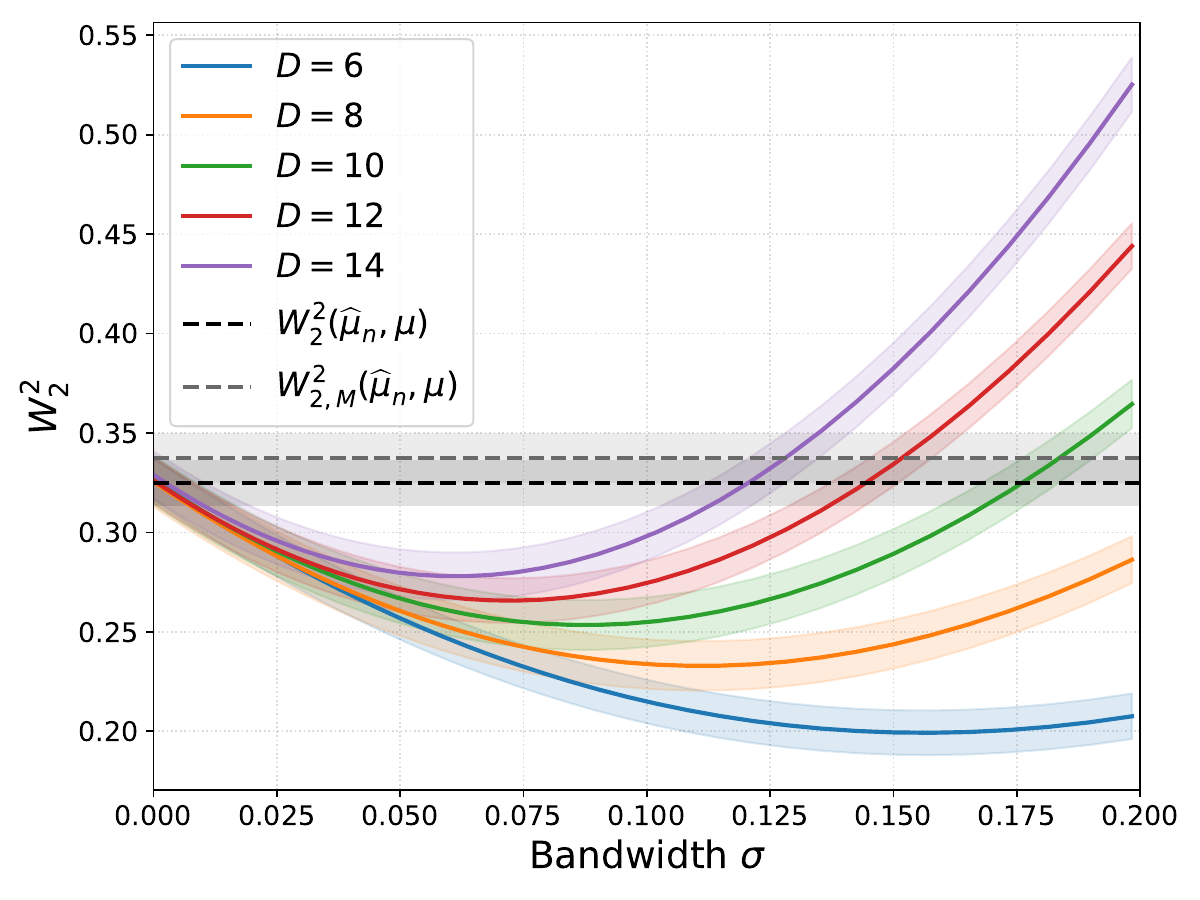}
        \caption{$\mathbb{S}^4$: empirical curves}
    \end{subfigure}
    \begin{subfigure}[t]{0.38\textwidth}
        \centering
        \includegraphics[width=\linewidth]{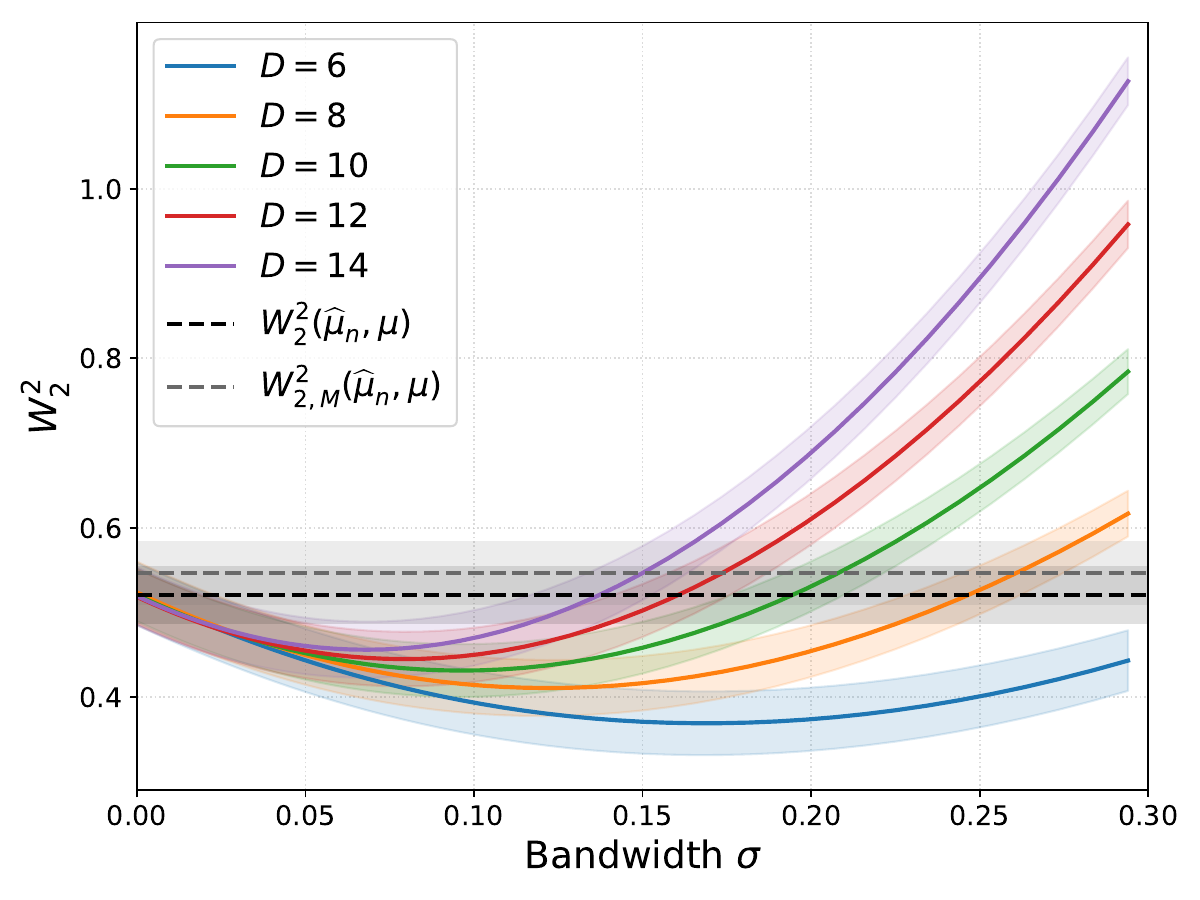}
        \caption{$\mathbb{S}^1\times\mathbb{S}^2$: empirical curves}
    \end{subfigure}

    \medskip

    \begin{subfigure}[t]{0.38\textwidth}
        \centering
        \includegraphics[width=\linewidth]{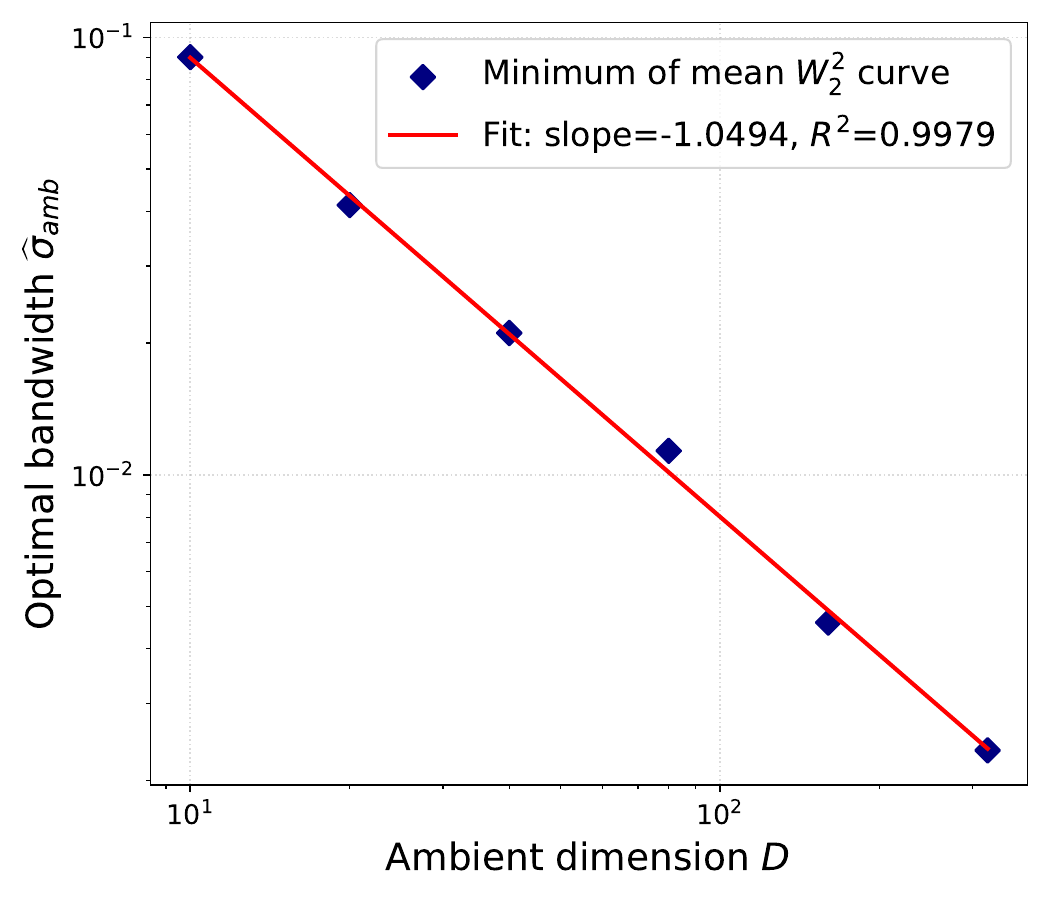}
        \caption{$\mathbb{S}^4$: fitted slope $-1.0494$}
    \end{subfigure}
    \begin{subfigure}[t]{0.38\textwidth}
        \centering
        \includegraphics[width=\linewidth]{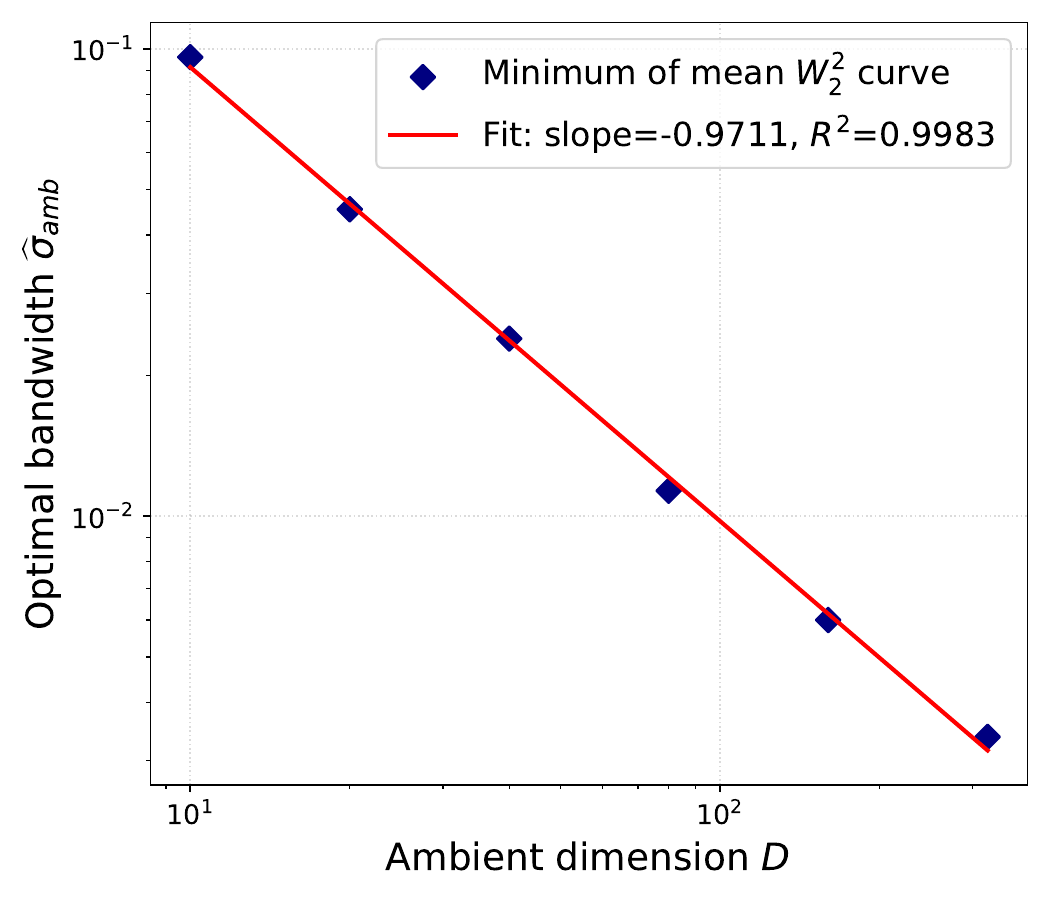}
        \caption{$\mathbb{S}^1\times\mathbb{S}^2$: fitted slope $-0.9711$}
    \end{subfigure}

    \caption{
    Dependence of ambient smoothing on the ambient dimension $D$ for the
    positively curved round sphere $\mathbb{S}^4$
    and the product manifold $\mathbb{S}^1\times\mathbb{S}^2$, which
    has flat mixed tangent planes.
    \textbf{Top row:} Mean empirical
    $W_2^2(k_\sigma*\hat{\mu}_n,\mu)$ as functions of
    $\sigma$, with $\pm1$ standard deviation bands over 5 trials.
    Colors indicate different ambient dimensions, and the horizontal
    reference lines show the corresponding unsmoothed errors.
    \textbf{Bottom row:} Log-log plots of the empirical bandwidth
    $\hat{\sigma}_{\mathrm{amb}}$ against
    $D\in\{10,20,40,80,160,320\}$, with linear regression fits in
    red. From left to right, the fitted slopes are $-1.0494$
    and $-0.9711$, close to the surrogate's $D^{-1}$ reference trend.
    }
    \label{fig:5-1-1_dimension_combined}
\end{figure}


\subsubsection{Impact of Sample Size \texorpdfstring{$(n)$}{(n)}}
\label{sec:impact of sample size n}



To isolate the effect of sample size $n$, we fix the ambient dimension at $D = 15$ and vary $n$. We examine the mean $W_2^2(k_\sigma *\hat{\mu}_n, \mu)$ error curves as functions of $\sigma$ and the scaling of $\sigma^*_{\mathrm{amb}}$ with $n$.

For each of the five trials and each sample size $n$, we independently draw a new empirical measure $\hat{\mu}_n$ and an $N_{\mathrm{eval}}$-point discretization of $\mu$. We then discretize $k_\sigma * \hat{\mu}_n$ accordingly.

Let $R_{t,n}(\sigma)$ denote the computed $W_2^2(k_{\sigma}*\hat{\mu}_n, \mu)$ errors at the trial $t$.
We define the corresponding optimal bandwidth for the mean $W_2^2(k_\sigma *\hat{\mu}_n, \mu)$ curves by
\[
\hat{\sigma}_{\mathrm{amb}}(n) := \arg\min_{ \sigma_p \in \{\sigma_0,\sigma_1,\dotsc, \sigma_{80} \}
} \frac{1}{5}\sum_{t=1}^5 R_{t,n}(\sigma)
\] as a numerical proxy for $\sigma_{\mathrm{amb}}^*(n)$.
At fixed $D$, the surrogate gives
\[
\sigma^*_{\mathrm{amb}}(n)
=\frac{C_{\mathrm{int}}n^{-1/m}}{2\beta_{\mathrm{amb}}(n)}.
\]
Because $\beta_{\mathrm{amb}}(n)$ contains the empirical transport error, $-1/m$ is a reference log-log slope when that coefficient changes slowly over the tested sample sizes; it is not an exact slope for the empirical minimizer. We consider
$n \in \{5,10,20,40,80\}$.

First, we plot the mean $W_2^2(k_\sigma*\hat{\mu}_n,\mu)$ against $\sigma$, with $\pm1$ standard deviation bands. Figure~\ref{fig:5-1-2_sample_size_combined} (top row) shows the results for both manifolds. For visualization, we display the mean $W_2^2(k_\sigma *\hat{\mu}_n , \mu)$ curves over $\sigma \in [0,0.3]$ in Figure \ref{fig:5-1-2_sample_size_combined} (top row).
For comparison, the plots include reference lines indicating the mean unsmoothed $W_2^2(\hat{\mu}_n,\mu)$ and $W_{2,M}^2(\hat{\mu}_n,\mu)$ errors. Colors distinguish different sample sizes $n$.
The results suggest that, increasing $n$ flattens the initial decrease in $W_2^2(k_\sigma*\hat{\mu}_n,\mu)$ and shifts the empirical optimal bandwidth toward zero.

Second, we examine the scaling of $\hat{\sigma}_{\mathrm{amb}}$ with $n$ by regressing $\log(\hat{\sigma}_{\mathrm{amb}})$ on $\log(n)$. Figure~\ref{fig:5-1-2_sample_size_combined} (bottom row) gives fitted slopes of $-0.2591$ for $\mathbb{S}^4$ and $-0.3317$ for $\mathbb{S}^1\times\mathbb{S}^2$, both close to their $-1/m$ reference values.

\begin{figure}[t]
    \centering
\renewcommand{\legenditem}[2]{%
    \tikz[baseline=-0.5ex]{
        \draw[#1, line width=1.0pt] (0,0) -- (0.58,0);
    }%
    \hspace{0.2em}#2%
}

\begingroup
\setlength{\fboxsep}{2pt}
\setlength{\fboxrule}{0pt}
\setlength{\tabcolsep}{4pt}
\fcolorbox{black!25}{white}{%
    \footnotesize
    \begin{tabular}{@{}ccccccc@{}}
         \legenditem{mplblue}{$n=5$} &
        \legenditem{mplorange}{$n=10$} &
        \legenditem{mplgreen}{$n=20$} &
        \legenditem{mplred}{$n=40$} &
        \legenditem{mplpurple}{$n=80$} &
        \legenditem{black,dashed}{$W_2^2(\hat{\mu}_n,\mu)$} &
        \legenditem{black,dotted}{$W_{2,M}^2(\hat{\mu}_n,\mu)$}
    \end{tabular}%
}

\endgroup

\par\vspace{0.4em}
    \begin{subfigure}[t]{0.38\textwidth}
        \centering
        \includegraphics[width=\linewidth]{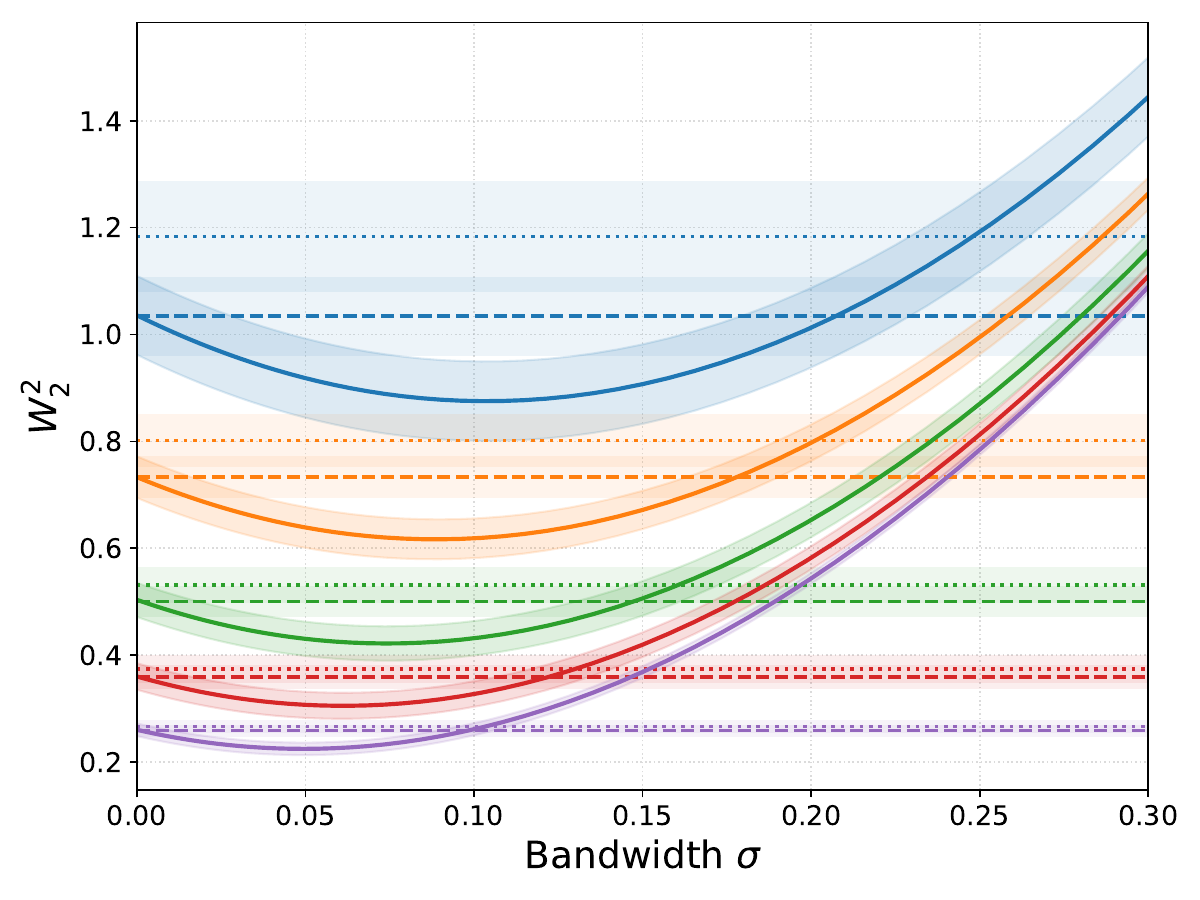}
        \caption{$\mathbb{S}^4$: empirical curves}
    \end{subfigure}
    \begin{subfigure}[t]{0.38\textwidth}
        \centering
        \includegraphics[width=\linewidth]{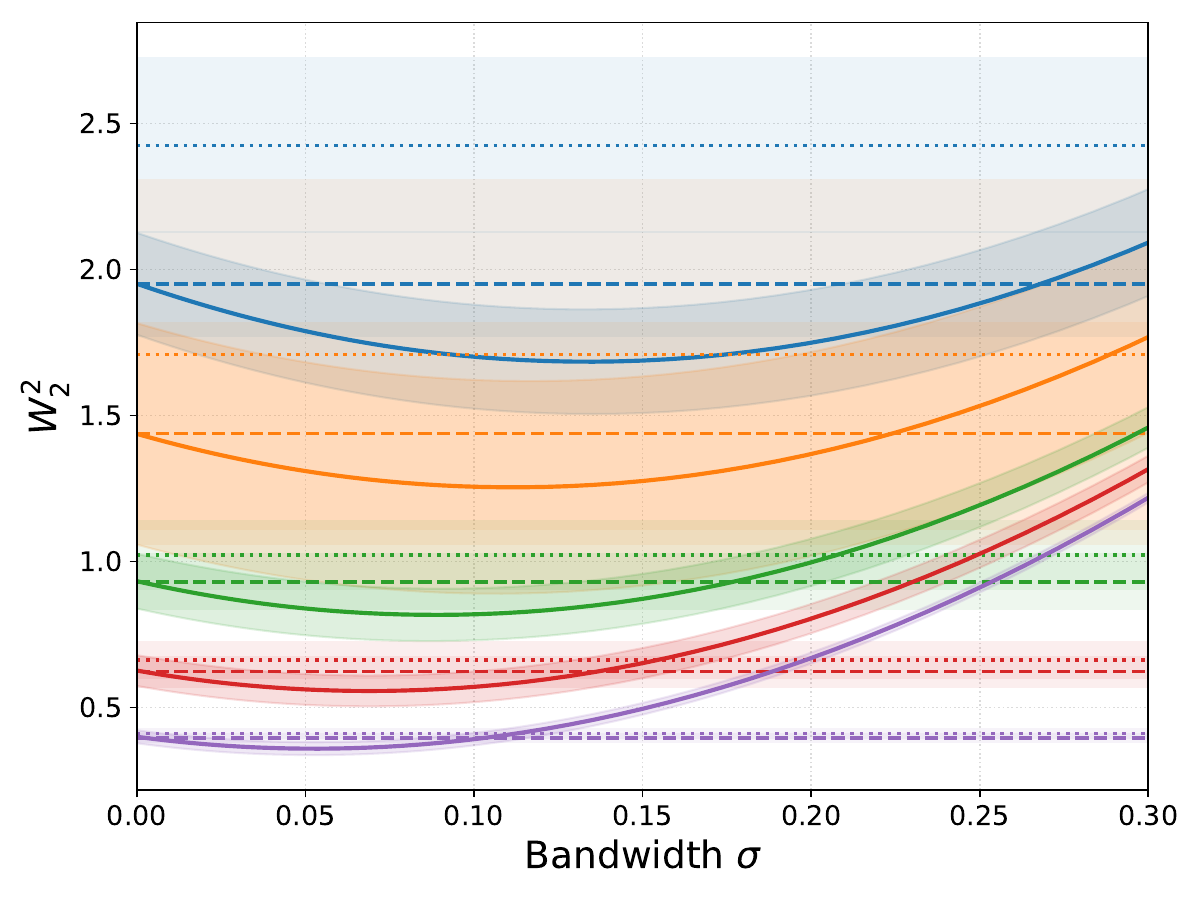}
        \caption{$\mathbb{S}^1\times\mathbb{S}^2$: empirical curves}
    \end{subfigure}

    \medskip

    \begin{subfigure}[t]{0.38\textwidth}
        \centering
        \includegraphics[width=\linewidth]{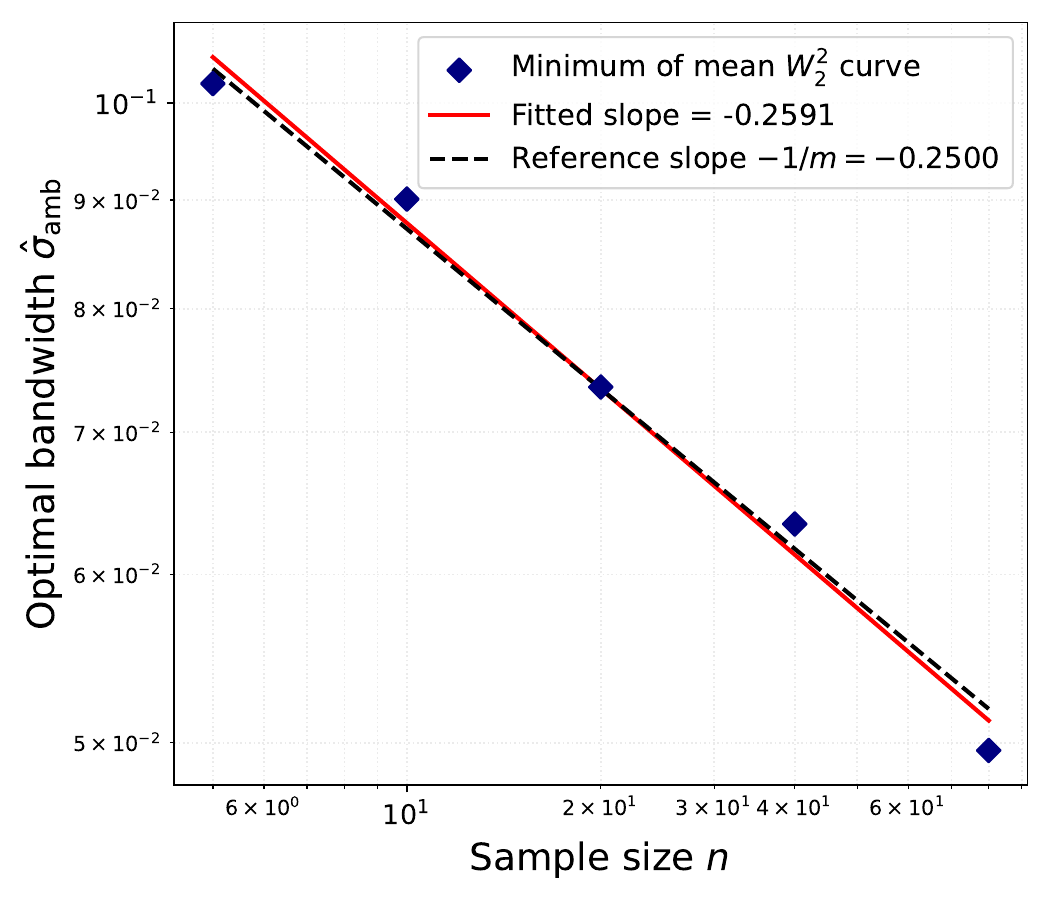}
        \caption{$\mathbb{S}^4$: fitted slope $-0.2591$}
    \end{subfigure}
    \begin{subfigure}[t]{0.38\textwidth}
        \centering
        \includegraphics[width=\linewidth]{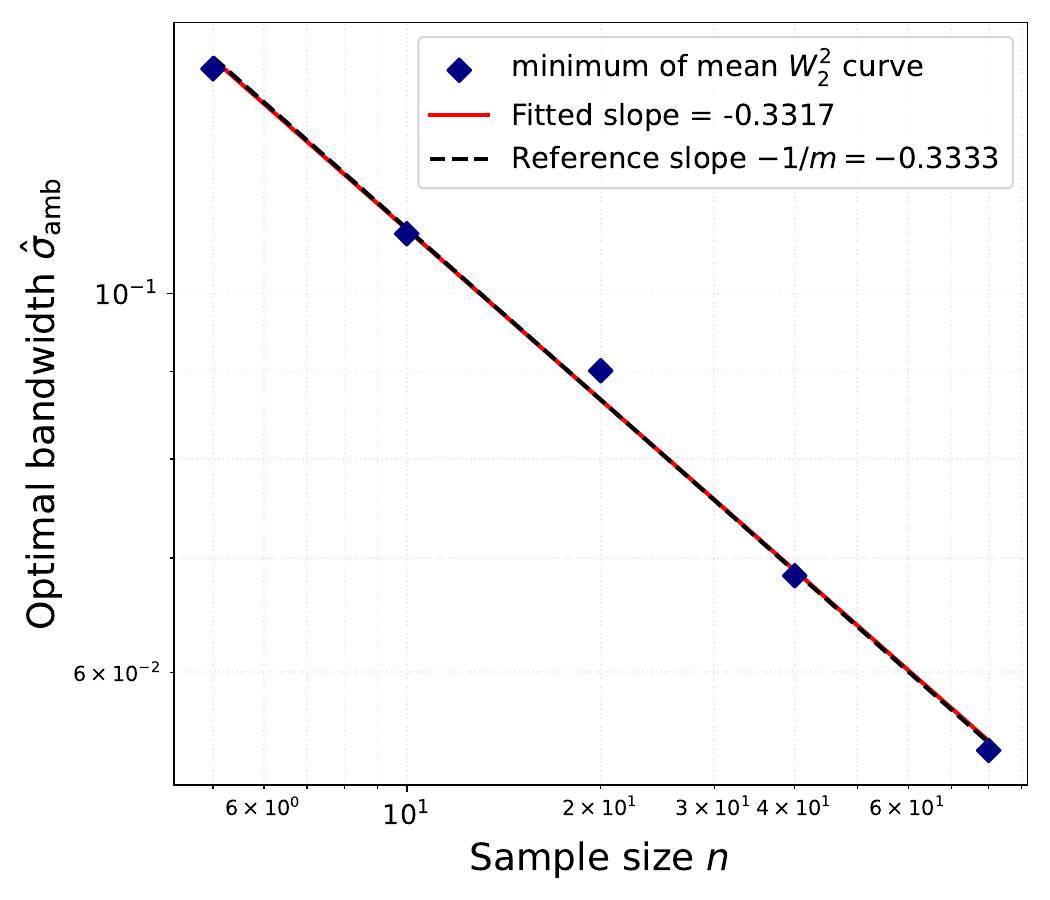}
        \caption{$\mathbb{S}^1\times\mathbb{S}^2$: fitted slope $-0.3317$}
    \end{subfigure}

    \caption{
    Dependence of ambient smoothing on the empirical sample size $n$ at
    fixed ambient dimension $D=15$. Columns correspond, from left to
    right, to the positively curved round sphere $\mathbb{S}^4$
    and the product manifold
    $\mathbb{S}^1\times\mathbb{S}^2$ with flat mixed directions.
    \textbf{Top row:} Mean empirical
    $W_2^2(k_\sigma*\hat{\mu}_n,\mu)$ as functions of
    $\sigma$, with $\pm1$ standard deviation bands over five trials and
    colors indicating different sample sizes. The horizontal reference lines indicate the unsmoothed errors $W_2^2(\hat{\mu}_n,\mu)$ and $W_{2,M}^2(\hat{\mu}_n,\mu)$.
    \textbf{Bottom row:} Log-log plots of the empirical bandwidth
    $\hat{\sigma}_{\mathrm{amb}}$ against $n$. Red lines show the fitted
    linear regressions, and the reference lines have slope $-1/m$.
    The fitted slopes are $-0.2591$ and $-0.3317$,
    respectively, close to that surrogate reference trend.
    }
    \label{fig:5-1-2_sample_size_combined}
\end{figure}

\clearpage

\subsection{Latent Smoothing Across Manifold Geometries and Encoder Quality}\label{sec:latent_smoothing_experiment}

We examine latent smoothing on a positively curved round sphere and on a
product manifold with flat mixed tangent planes. Four sweeps isolate the
tangential gain $\rho_{\parallel}$, the orthogonal decoder response
$L_{\mathrm{orth}}$, the number $d-m$ of orthogonal latent directions, and
the reconstruction error $\delta$. A separate sweep varies the tangential
isometry defect $\varepsilon_{\mathrm{iso}}$ in
\Cref{thm:latent_decomposition}. On both manifolds, the observed preferred
bandwidth increases with $\rho_{\parallel}$ and decreases as
$L_{\mathrm{orth}}$, $d-m$, or $\varepsilon_{\mathrm{iso}}$ increases.

Each sweep uses explicit global encoder and decoder maps, rather than
inserting theorem parameters directly into a sampling formula. Their local
Jacobians can be calculated analytically and checked by finite differences.
The main figure compares preferred bandwidths across both manifolds. The complete
Wasserstein curves and finite-difference parameter checks are grouped by
manifold in \Cref{app:latent_smoothing_diagnostics}.

\subsubsection{Manifolds and target distributions}
\label{subsec:latent_smoothing_design}

We use the unit sphere $\mathbb S^2\subset\mathbb R^3$ and the product
$\mathbb S^1\times\mathbb S^2\subset\mathbb R^2\times\mathbb R^3$.
Both are embedded in the same ambient space $\mathbb R^{100}$ by appending
zero coordinates.  Hence the intrinsic dimensions are $m=2$ and $m=3$,
respectively, while the ambient dimension is fixed at $D=100$.

On $\mathbb S^2$, we sample from a two-cluster density.  For a unit vector
$q\in\mathbb S^2$ and $\lambda>0$, let
\[
    p_{q,\lambda}(x)
    :=
    \frac{\exp(\lambda\langle q,x\rangle)}
    {\int_{\mathbb S^2}\exp(\lambda\langle q,u\rangle)\,d\vol(u)}.
\]
We take
\[
    \mu_{\mathbb S^2}
    =
    0.62\,p_{q_1,8}+0.38\,p_{q_2,11},
    \qquad
    q_1=(0,0,1),
    \qquad
    q_2=
    \frac{(0.88,0.15,-0.45)}{\|(0.88,0.15,-0.45)\|}.
\]
Thus the two components are centered at different points of the sphere and
have different concentrations.

For the circle factor, write a point as $(\cos\theta,\sin\theta)$ and,
for $\lambda>0$, set
\[
    p_{\theta_0,\lambda}(\theta)
    :=
    \frac{\exp\{\lambda\cos(\theta-\theta_0)\}}
    {\int_0^{2\pi}\exp\{\lambda\cos(s-\theta_0)\}\,ds}.
\]
We use
\[
    \mu_{\mathbb S^1}
    =
    0.58\,p_{0.25,6}+0.42\,p_{3.65,9},
    \qquad
    \mu_{\mathrm{prod}}
    =
    \mu_{\mathbb S^1}\otimes\mu_{\mathbb S^2}.
\]
The round sphere has positive sectional curvature, whereas
$\mathbb S^1\times\mathbb S^2$ has zero sectional curvature on mixed
tangent planes. Both satisfy the compact smooth manifold assumptions of
the intrinsic heat-flow input.
The product tests the latent bound beyond strict positive sectional
curvature, while neither example tests negative Ricci curvature.

\subsubsection{Explicit encoder and decoder}
\label{subsec:explicit_encoder_decoder}

For each run of the first four sweeps, we draw $n=50$ independent anchors
$x_1,\ldots,x_n\sim\mu$ and form
$\hat\mu_n=n^{-1}\sum_{i=1}^n\delta_{x_i}$.  Let $d_0=3$ on
$\mathbb S^2$ and $d_0=5$ on
$\mathbb S^1\times\mathbb S^2$.  For any $d_0\leq d\leq D$, the encoder
is the coordinate projection
\begin{equation}
\label{eq:explicit_coordinate_encoder}
    f_d(x_1,\ldots,x_D)
    =
    (x_1,\ldots,x_d).
\end{equation}
Because the final $D-d_0$ coordinates of the embedded manifolds vanish,
$J_{f_d}$ preserves every tangent-vector norm exactly.

Let $E_i\in\mathbb R^{D\times m}$ contain an orthonormal basis of
$T_{x_i}M$, and let $\widetilde E_i\in\mathbb R^{d\times m}$ be obtained
by retaining its first $d$ rows.  Then
$\widetilde E_i^{\top}\widetilde E_i=I_m$.  Define
\[
    c_i=f_d(x_i),
    \qquad
    P_i=\widetilde E_i\widetilde E_i^{\top},
\]
so $P_i$ projects onto the encoded tangent space.  On $\mathbb S^2$, the
columns of $E_i$ are obtained by normalizing $x_i\times r_i$ and then
taking its cross product with $x_i$.  We use $r_i=(1,0,0)$ unless this is
nearly parallel to $x_i$, in which case we use $r_i=(0,1,0)$.  On the
product manifold, we add the circle tangent
$(-x_{i,2},x_{i,1})$ to the two sphere tangents.

Let $\iota_d:\mathbb R^d\to\mathbb R^D$ be coordinate inclusion and let
$I_d$ denote the identity on $\mathbb R^d$.  The local
decoder centered at $c_i$ is
\begin{equation}
\label{eq:explicit_local_decoder}
    h_i(z)
    :=
    \exp_{x_i}^{M}\!\left(
        \rho_{\parallel}E_i\widetilde E_i^{\top}(z-c_i)
    \right)
    +
    L_{\mathrm{orth}}\iota_d(I_d-P_i)(z-c_i).
\end{equation}
For $M=\mathbb S^1\times\mathbb S^2$, the exponential map is applied
separately to the two factors.  To join the local maps, let $P_{d_0}$
denote projection onto the first $d_0$ coordinates and define the cardinal
weights
\begin{equation}
\label{eq:cardinal_shepard_weights}
    w_i(z)
    :=
    \frac{
        \prod_{k\neq i}\|P_{d_0}(z-c_k)\|^4
    }{
        \sum_{j=1}^n
        \prod_{k\neq j}\|P_{d_0}(z-c_k)\|^4
    }.
\end{equation}
For distinct anchors, these weights satisfy
$w_i(c_i)=1$, $w_j(c_i)=0$ for $j\neq i$, and
$Dw_j(c_i)=0$ for every $j$.  Finally, choose a fixed unit vector $b$ and
set
\begin{equation}
\label{eq:explicit_global_decoder}
    g(z)
    :=
    \sum_{i=1}^n w_i(z)h_i(z)+\delta b.
\end{equation}
In the implementation, $b$ is the unit vector opposite to the population
mean.  This fixes one reproducible direction for the reconstruction
displacement without changing the decoder Jacobian.

\begin{proposition}[Prescribed local decoder responses]
\label{prop:explicit_decoder_responses}
The maps in \eqref{eq:explicit_coordinate_encoder} and
\eqref{eq:explicit_global_decoder} satisfy, at every empirical anchor,
\[
    \|g(f_d(x_i))-x_i\|=\delta,
\]
and
\[
    J_g(c_i)
    =
    \rho_{\parallel}E_i\widetilde E_i^{\top}
    +
    L_{\mathrm{orth}}\iota_d(I_d-P_i).
\]
Consequently, for every $v\in T_{x_i}M$ and
$v^\perp\in\operatorname{range}(P_i)^\perp$,
\[
    J_g(f_d(x_i))J_{f_d}(x_i)v
    =
    \rho_{\parallel}v,
    \qquad
    \|J_g(f_d(x_i))v^\perp\|
    =
    L_{\mathrm{orth}}\|v^\perp\|.
\]
Thus this construction realizes the tangential gain, reconstruction error,
orthogonal response, and latent complement dimension appearing in
\Cref{thm:latent_decomposition}, with zero tangential isometry defect. The
separate defect sweep below uses another explicit map pair.
\end{proposition}

\begin{proof}
At $z=c_i$, cardinality of the weights gives
$g(c_i)=h_i(c_i)+\delta b=x_i+\delta b$.  Their first derivatives vanish
at the centers, so differentiating only $h_i$ and using
$d(\exp_{x_i}^{M})_0=I_{T_{x_i}M}$ gives the displayed Jacobian.  Since
$\widetilde E_i^{\top}\widetilde E_i=I_m$ and $P_i$ is an orthogonal
projection, the two response identities follow.
\end{proof}

\begin{proposition}[Global regularity of the explicit decoder]
\label{prop:explicit_decoder_global_c2}
For pairwise distinct anchors, the decoder in
\eqref{eq:explicit_global_decoder} belongs to
$C^2(\mathbb R^d;\mathbb R^D)$ and has uniformly bounded first and
second derivatives. Thus it satisfies the global regularity assumption
of \Cref{thm:latent_decomposition} for each fixed set of anchors.
\end{proposition}

\begin{proof}
Write $z=u+v$ with $u=P_{d_0}z$ and
$v=(I_d-P_{d_0})z$. Since $c_i$ and the range of $P_i$ lie in
the first $d_0$ coordinates, $w_i(z)=w_i(u)$ and
\[
 h_i(u+v)=h_i(u)+L_{\mathrm{orth}}\iota_d v.
\]
The weights sum to one, so
$g(u+v)=\sum_iw_i(u)h_i(u)+L_{\mathrm{orth}}\iota_dv+\delta b$.
This identity separates the unbounded affine part from the part on
$\mathbb R^{d_0}$.

For distinct anchors, the denominator in
\eqref{eq:cardinal_shepard_weights} is positive everywhere: at a center
exactly one numerator is nonzero, while away from all centers every
numerator is positive. The weights are therefore smooth. As
$r=\|u\|\to\infty$, set
$p_i(u):=\prod_{k\ne i}\|u-c_k\|^4$ and $q:=4(n-1)$.
Because the centers are fixed, uniformly over directions $u/r$,
$p_i(u)=r^q(1+O(r^{-1}))$ and
$\|D^jp_i(u)\|=O(r^{q-j})$ for $j=1,2$.
Thus $\sum_jp_j(u)=nr^q(1+O(r^{-1}))$ is bounded below by a
positive multiple of $r^q$ for large $r$. Applying the quotient rule
to $w_i=p_i/\sum_jp_j$ gives
$|w_i(u)|=O(1)$, $\|Dw_i(u)\|=O(r^{-1})$, and
$\|D^2w_i(u)\|=O(r^{-2})$.
For a round sphere, the exponential map is
$\exp_x(v)=\cos(\|v\|)x+
\frac{\sin(\|v\|)}{\|v\|}v$; the product exponential is applied
factorwise. These maps and their first two derivatives are bounded
globally. In \eqref{eq:explicit_local_decoder}, the exponential term
is therefore bounded, while the affine term has norm at most
$L_{\mathrm{orth}}\|u-c_i\|=O(1+r)$. Hence
$\|h_i(u)\|=O(1+r)$ and its first two derivatives are bounded.
The product rule now shows that the first two derivatives of
$\sum_iw_i(u)h_i(u)$ are bounded as $r\to\infty$; they are bounded
on compact sets by continuity. The affine $v$ term has constant first
derivative and zero second derivative. This proves the claim.
\end{proof}

\subsubsection{Numerical protocol}
\label{subsec:latent_numerical_protocol}

For each manifold, we perform five independent runs. In the first four
sweeps, each run draws a new set of $n=50$ anchors, an independent target
sample of size
$N_{\mathrm{OT}}=500$, and new Gaussian coordinates.  For a given run,
the same anchors, target sample, empirical-mixture indices, and Gaussian
coordinates are reused across all values in a parameter sweep.  This
common-random-number construction reduces the variability of comparisons
within a run without treating repeated evaluations of one empirical measure
as independent data.

For each of the first four sweeps and every $\sigma$, we generate
$N_{\mathrm{OT}}$ points of the form
\[
    g\bigl(f_d(x_{I_j})+\sigma\xi_j\bigr),
    \qquad
    I_j\sim\operatorname{Unif}\{1,\ldots,n\},
    \qquad
    \xi_j\sim\mathcal N(0,I_d).
\]
The separate defect sweep uses the same sampling scheme with
$f_{\varepsilon}$ and $g_{\varepsilon}$ defined below. We compare each
generated cloud with the independent $N_{\mathrm{OT}}$-point target sample
using the exact discrete squared Euclidean transport cost computed by
\texttt{ot.emd2} \citep{flamary2021pot}. At $\sigma=0$, the plotted cost
still uses the $N_{\mathrm{OT}}$ sampled anchor indices $I_j$; it is
therefore a Monte Carlo approximation of the zero-noise pushforward cost,
rather than the cost from the original $n$ anchors with weights $1/n$.
Reusing these indices across bandwidths supports paired comparisons
within each run.

The bandwidth grid is
\[
    \sigma_j
    =
    0.1\left(\frac{j}{80}\right)^2,
    \qquad j=0,\ldots,80,
\]
which places more grid points near zero. For the $L_{\mathrm{orth}}$
sweep only, we retain these points and append
$\sigma_{80+\ell}=0.1+0.005\ell$ for $\ell=1,\ldots,80$, reaching
$\sigma=0.5$. In each run, we select the
bandwidth grid point with the smallest computed transport
cost as its estimated optimal bandwidth.

Optimal-bandwidth plots show means of the five runwise grid minimizers,
with $95\%$ Student-$t$ confidence intervals.
Normalized optimal-bandwidth comparisons use ratios of mean optima, with
paired first-order ratio intervals. This avoids dividing by an individual
run's reference optimum when its estimated minimum occurs at $\sigma=0$.
We check the decoder identities in \Cref{prop:explicit_decoder_responses}
and the corresponding defect-sweep identities below by central finite
differences at ten anchors in every run.

The theorem gives a small-bandwidth upper bound, not an exact formula for the
minimizing bandwidth. We do not numerically certify its anchor-dependent
$C_g$, $W_0$, or admissible bandwidth threshold. Accordingly, the plots test
qualitative parameter trends rather than the sharpness of the bound's
constants.

\subsubsection{Sweep parameters and expected trends}
\label{subsec:latent_smoothing_sweeps}

The first three sweeps use the zero-defect maps in
\Cref{prop:explicit_decoder_responses}. We vary one quantity at a time:
\begin{align*}
\rho_{\parallel}&\in\{0.8,0.9,1.0,1.1,1.2\}
  &&(\delta=0,\ L_{\mathrm{orth}}=1,\ d-m=80),\\
L_{\mathrm{orth}}&\in\{0,0.15,0.25,0.35,0.5\}
  &&(\rho_{\parallel}=1,\ \delta=0,\ d-m=80),\\
d-m&\in\{10,25,40,60,80\}
  &&(\rho_{\parallel}=1,\ \delta=0,\ L_{\mathrm{orth}}=0.5).
\end{align*}
The $\rho_{\parallel}$ sweep makes the orthogonal quadratic penalty
$L_{\mathrm{orth}}^2(d-m)=80$ large relative to the changing tangential
terms. Its dominant-term prediction is
\begin{equation}
\label{eq:rho_dominant_prediction}
\frac{\sigma_{\mathrm{quad}}^*(\rho_{\parallel})}
     {\sigma_{\mathrm{quad}}^*(1)}\approx\rho_{\parallel}.
\end{equation}
The plotted gray curve is a descriptive fit of
$\rho_{\parallel}(1+\beta)/(\rho_{\parallel}^2+\beta)$, with
$\beta\geq0$ fitted. In the second sweep, the explicit orthogonal penalty
is $80L_{\mathrm{orth}}^2\sigma^2$. The anchored reciprocal
comparison plotted in gray is
\begin{equation}
\label{eq:lorth_reciprocal_prediction}
\hat\sigma_{\mathrm{comp}}(L_{\mathrm{orth}})
=\frac{a}{80L_{\mathrm{orth}}^2+q_{\mathrm{orth}}},
\end{equation}
where $a/q_{\mathrm{orth}}$ matches the observed mean optimum at
$L_{\mathrm{orth}}=0$ and $q_{\mathrm{orth}}>0$ is fitted. This is a
descriptive fit, not an exact formula for the quadratic surrogate's
minimizer. In the third sweep, the explicit penalty is
$0.25(d-m)\sigma^2$, and the gray comparison curve is
$a/[0.25(d-m)+q_d]$ with a fitted $q_d\geq0$. The red reference curves
show the anchored dominant scalings $L_{\mathrm{orth}}^{-2}$ and
$(d-m)^{-1}$. The remainder $Q$ can change with $L_{\mathrm{orth}}$ or
$d$ through the decoder derivative bound $C_g$, so these comparisons
isolate the explicit penalties rather than give full formulas for the
minimizing bandwidth.

A fourth zero-defect sweep varies the reconstruction error
$\delta\in\{0,0.2,0.4,0.6,0.8\}$ with
$\rho_{\parallel}=1$, $L_{\mathrm{orth}}=0.5$, and $d-m=8$.
We report its full curves in \Cref{subsec:latent_smoothing_delta}, since
$\delta$ affects the baseline and multiple bandwidth-dependent terms.

For the separate tangential-isometry sweep, we use $n=50$ anchors,
$N_{\mathrm{OT}}=500$, and a bandwidth grid reaching $0.3$. The settings
are $\rho_{\parallel}=1$, $\delta=L_{\mathrm{orth}}=0$, $d=m+2$, and
$\varepsilon_{\mathrm{iso}}\in\{0,0.25,0.5,1,2,3\}$. The two extra latent
coordinates are ignored by the decoder, which isolates the tangential
defect from the orthogonal-noise response.

To prescribe the defect, choose distinct lattice points
$a_i\in\mathbb R^m$ and set $c_i=(a_i,0,0)\in\mathbb R^{m+2}$.
Let $\omega_i$ be cardinal weights centered at $x_i$ in the ambient
space and $\bar w_i$ the analogous weights centered at $a_i$ in
$\mathbb R^m$, both using the fourth-power construction of
\eqref{eq:cardinal_shepard_weights}. With $E_i$ as above, define
\begin{align*}
 f_{\varepsilon}(x)
 &=\sum_{i=1}^n\omega_i(x)
   \bigl(a_i+E_i^{\top}(x-x_i),0,0\bigr),\\
 g_{\varepsilon}(z)
 &=\sum_{i=1}^n\bar w_i(z_{1:m})
   \exp_{x_i}^{M}\!\left(
       E_i B_{\varepsilon}(z_{1:m}-a_i)
   \right),
\end{align*}
where
\[
 B_{\varepsilon}
 =I_m+s_{\varepsilon}e_1e_2^{\top},
 \qquad
 s_{\varepsilon}
 =\frac{\varepsilon(\varepsilon+2)}{\varepsilon+1}.
\]
The cardinal weights have zero first derivative at their centers. Hence
$f_{\varepsilon}(x_i)=c_i$, $g_{\varepsilon}(c_i)=x_i$, and, in the
tangent frame $E_i$,
\[
 J_{g_{\varepsilon}}(c_i)
 J_{f_{\varepsilon}}(x_i)E_i
 =E_iB_{\varepsilon}.
\]
The restriction of $J_{f_{\varepsilon}}(x_i)$ to $T_{x_i}M$ is an
isometry, while $J_{g_{\varepsilon}}(c_i)$ annihilates both additional
latent directions. The extreme singular values of $B_{\varepsilon}$
are $1+\varepsilon$ and $(1+\varepsilon)^{-1}$, so its operator-norm
distance to the nearest orthogonal map is exactly $\varepsilon$.
Thus \eqref{eq:def-isometry-defect} equals the prescribed value at every
anchor. The decoder has bounded first and second derivatives by the same
cardinal-weight argument as in
\Cref{prop:explicit_decoder_global_c2}.

For these manifolds $\kappa=0$. Under the settings above, the explicit
defect contribution to $L_{\mathrm{latent}}$ in
\Cref{thm:latent_decomposition} is
$2m\varepsilon_{\mathrm{iso}}+m\varepsilon_{\mathrm{iso}}^2$.
The dashed comparison in the bandwidth plots is an anchored descriptive fit
\[
\hat\sigma_{\mathrm{trend}}(\varepsilon)
=\hat\sigma(0)\frac{q}{q+2m\varepsilon+m\varepsilon^2}
\]
with one fitted $q$. It tests the predicted direction of the defect
effect; $q$ is not a measurement of the theorem's $Q$.

\subsubsection{Preferred bandwidths on the sphere and product manifold}
\label{subsec:latent_smoothing_results}

\Cref{fig:latent_smoothing_optima} compares preferred bandwidths on
$\mathbb S^2$ and $\mathbb S^1\times\mathbb S^2$, whose mixed tangent
planes have zero sectional curvature. On both manifolds, the observed
tangential-gain ratios increase in the direction predicted by
\eqref{eq:rho_dominant_prediction}, whereas increasing the orthogonal
response or latent complement dimension decreases the preferred
bandwidth. In the separate defect sweep, the mean preferred bandwidth
decreases from $0.200$ to $0.0595$ on $\mathbb S^2$ and from $0.279$ to
$0.0957$ on $\mathbb S^1\times\mathbb S^2$ as
$\varepsilon_{\mathrm{iso}}$ increases from $0$ to $3$. The complete
Wasserstein curves are in \Cref{app:latent_smoothing_diagnostics}, and
the five parameter checks are in
\Cref{subsec:latent_parameter_verification}.

\begin{figure}[t]
\centering
\begin{minipage}[t]{0.24\textwidth}
\centering
\includegraphics[width=\linewidth]{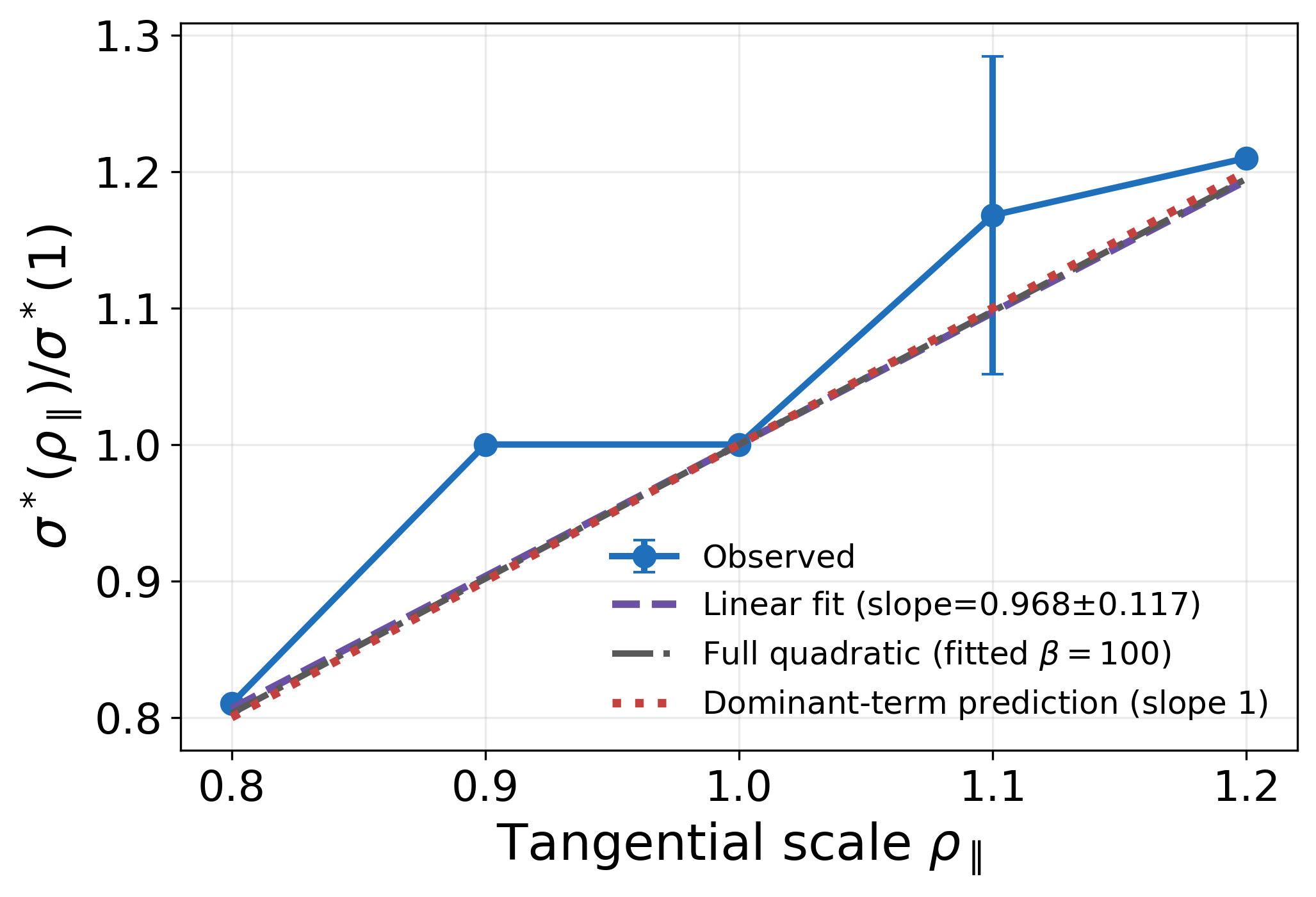}
\end{minipage}\hfill
\begin{minipage}[t]{0.24\textwidth}
\centering
\includegraphics[width=\linewidth]{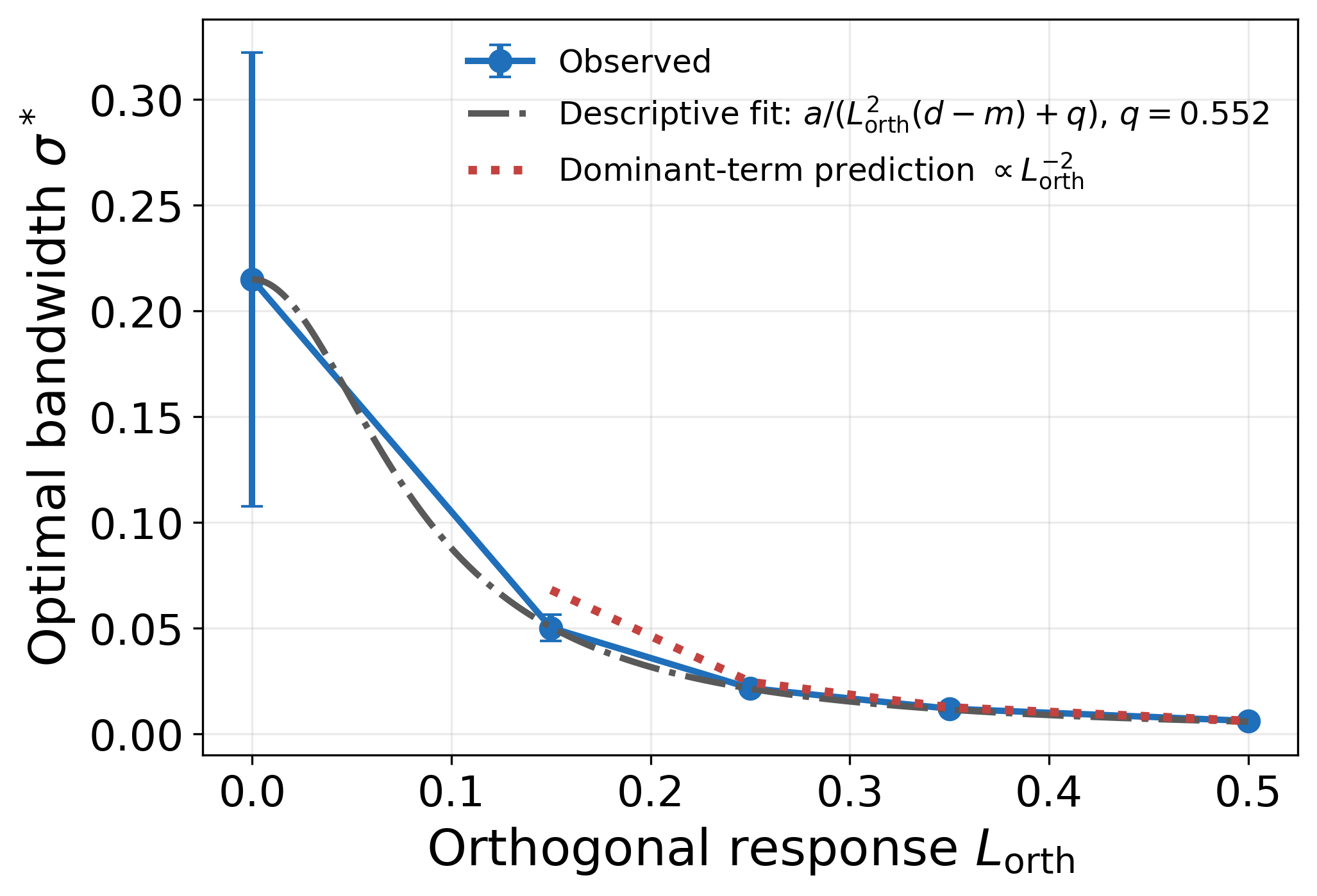}
\end{minipage}\hfill
\begin{minipage}[t]{0.24\textwidth}
\centering
\includegraphics[width=\linewidth]{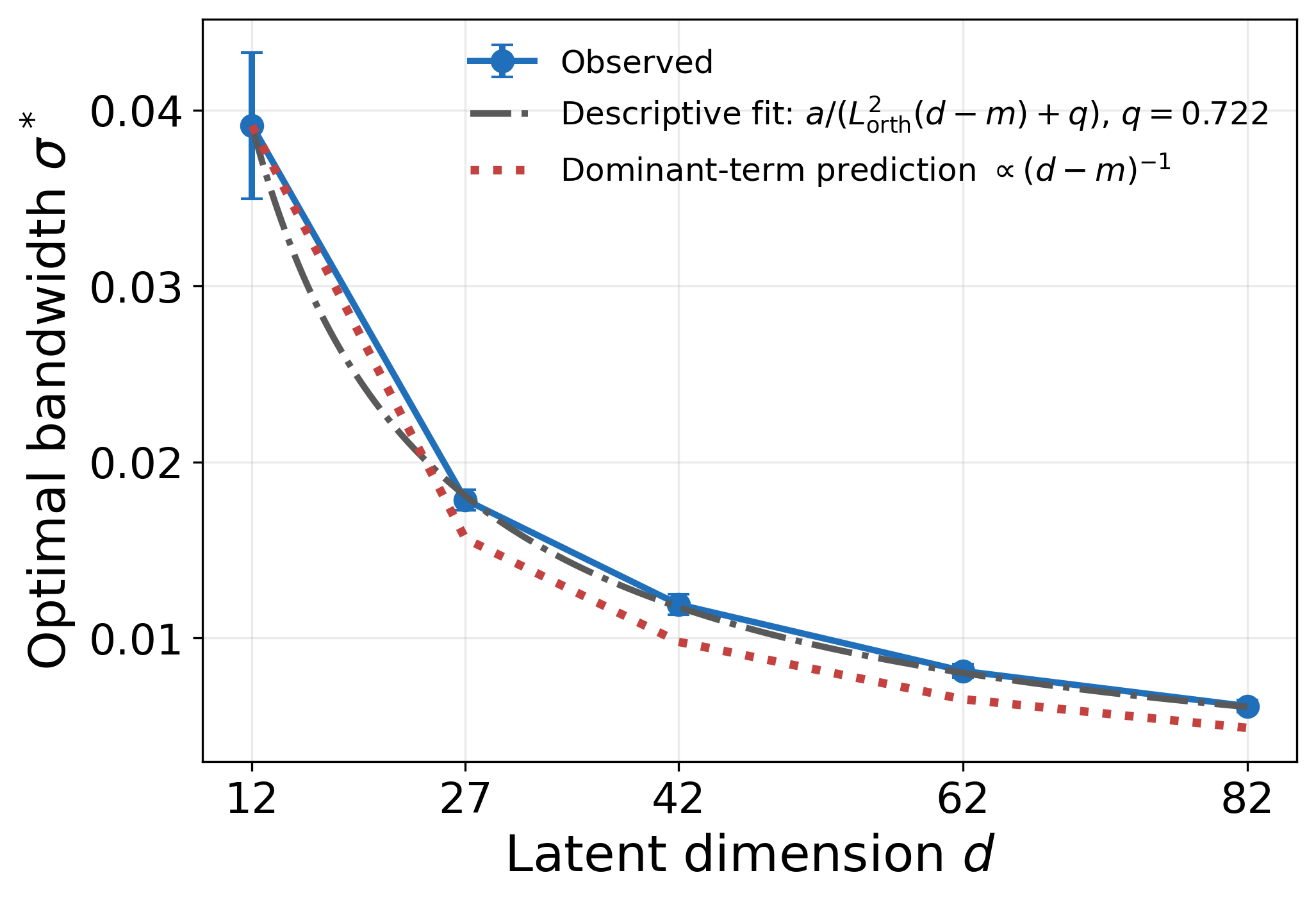}
\end{minipage}\hfill
\begin{minipage}[t]{0.24\textwidth}
\centering
\includegraphics[width=\linewidth,trim=0bp 0bp 0bp 26bp,clip]{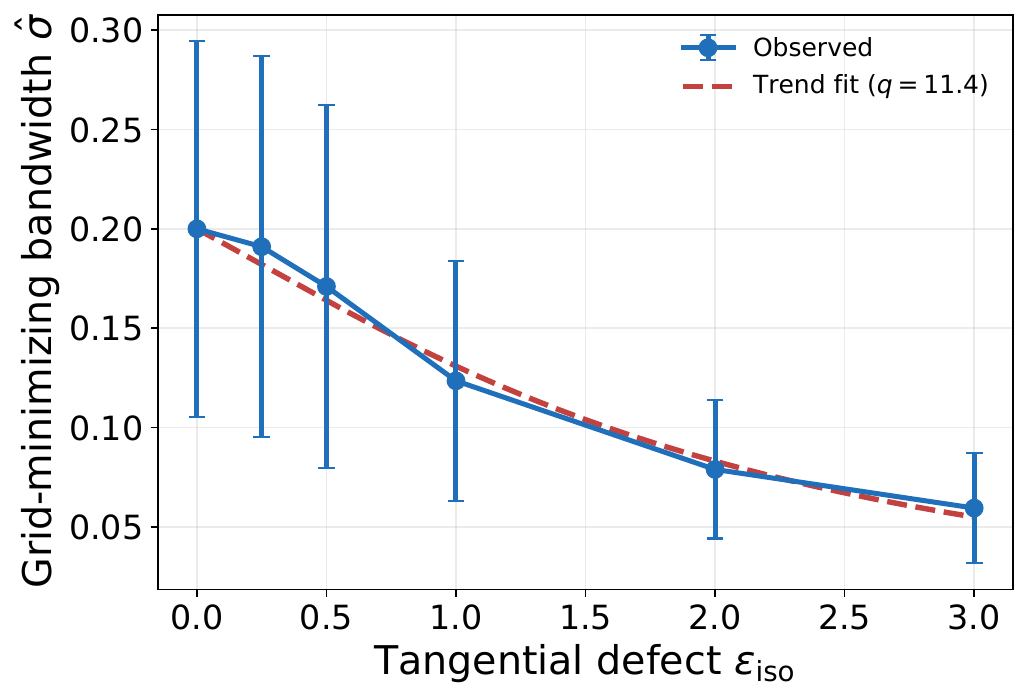}
\end{minipage}

\medskip
\begin{minipage}[t]{0.24\textwidth}
\centering
\includegraphics[width=\linewidth]{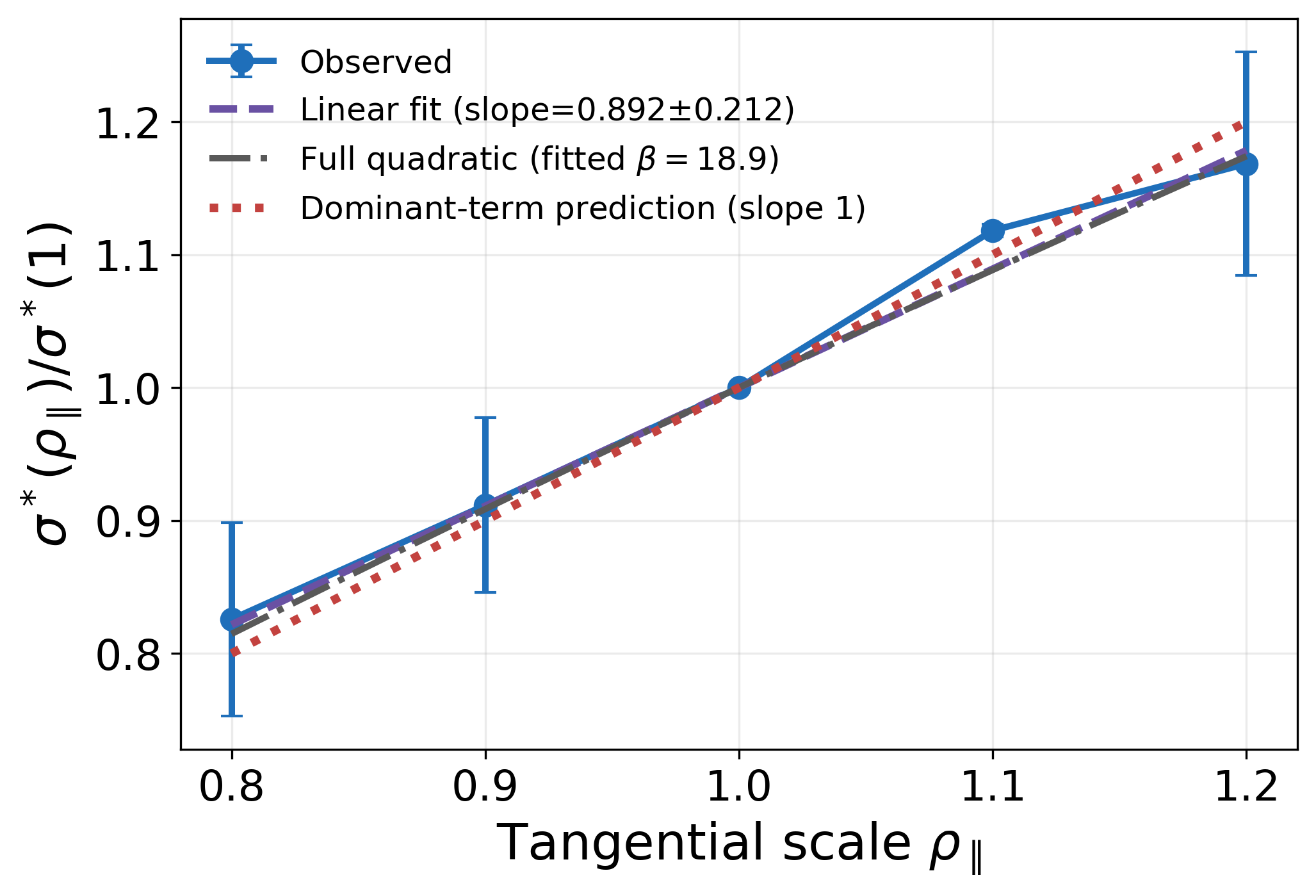}
\end{minipage}\hfill
\begin{minipage}[t]{0.24\textwidth}
\centering
\includegraphics[width=\linewidth]{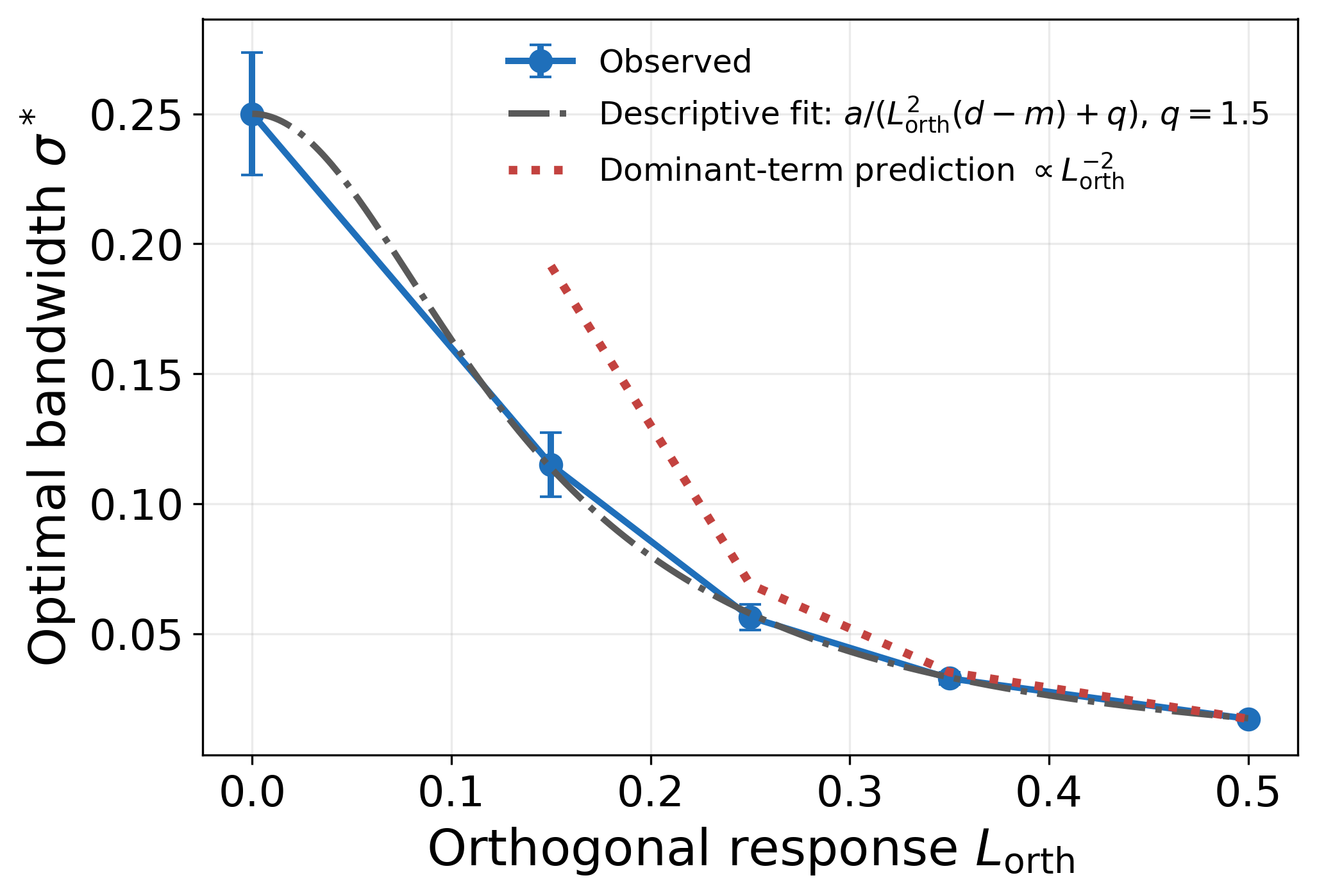}
\end{minipage}\hfill
\begin{minipage}[t]{0.24\textwidth}
\centering
\includegraphics[width=\linewidth]{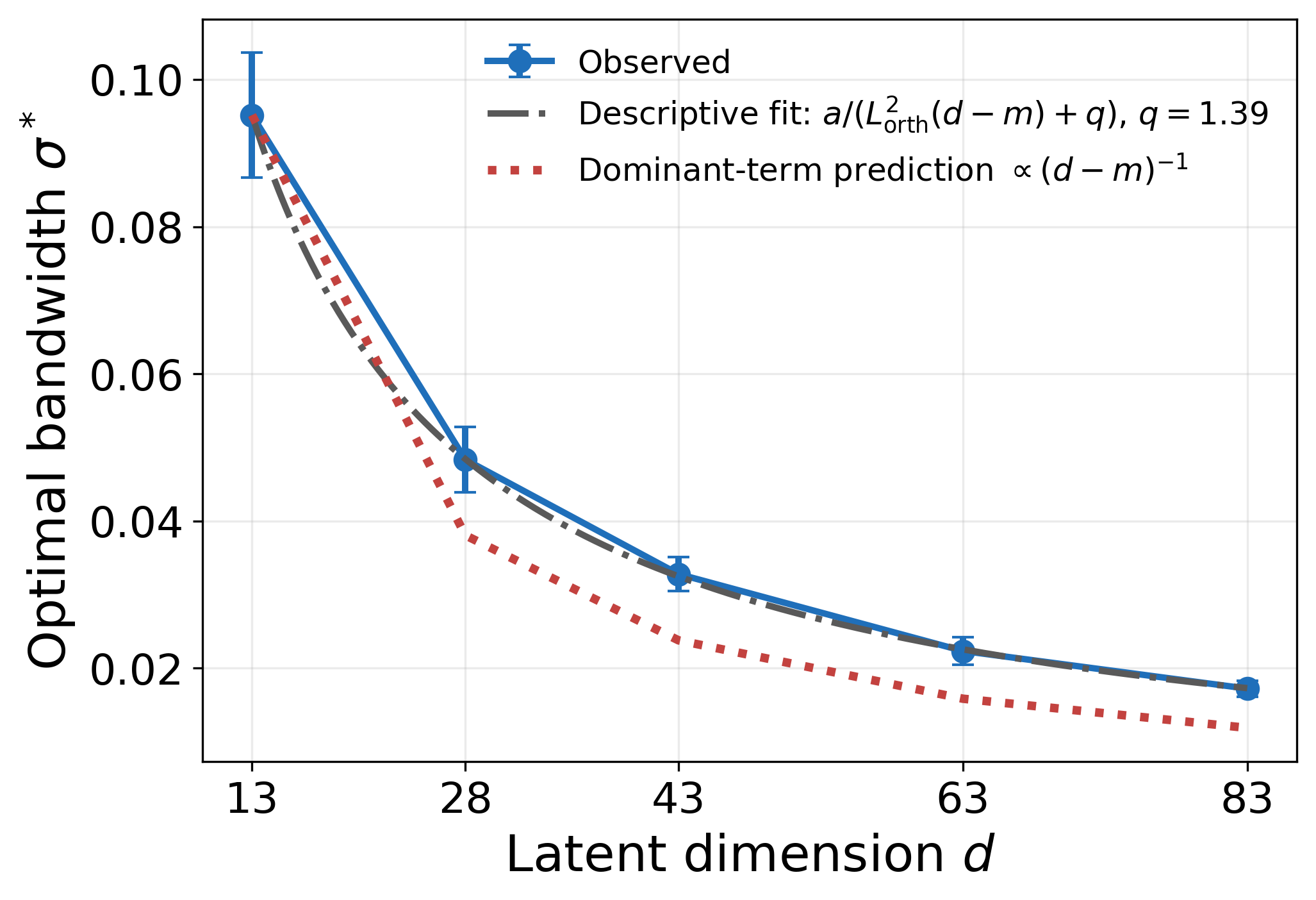}
\end{minipage}\hfill
\begin{minipage}[t]{0.24\textwidth}
\centering
\includegraphics[width=\linewidth,trim=0bp 0bp 0bp 26bp,clip]{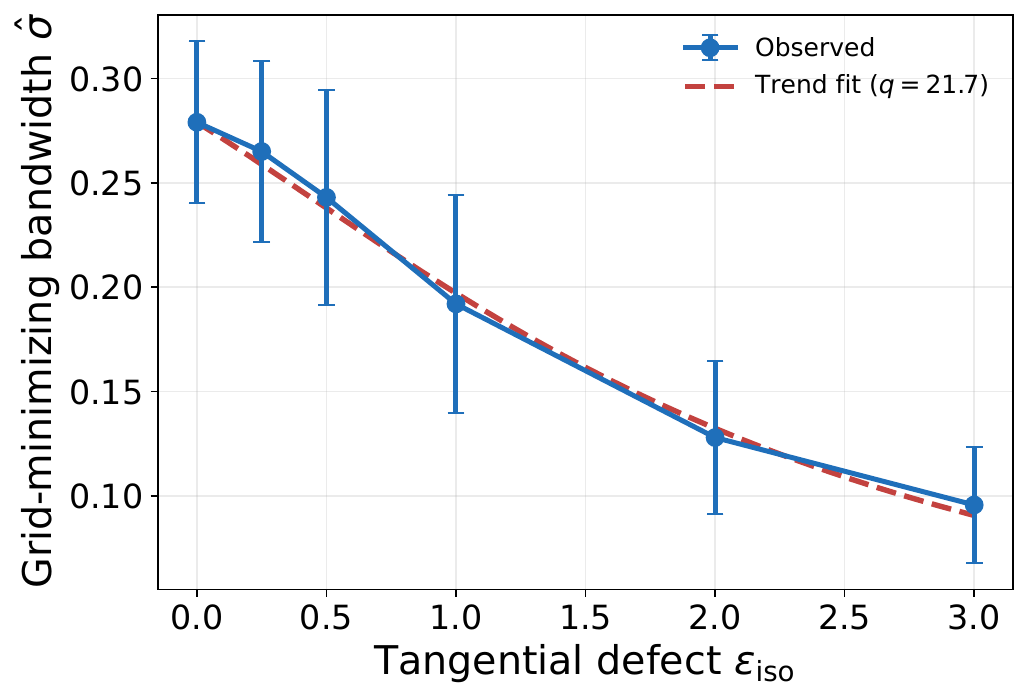}
\end{minipage}
\caption{Preferred bandwidths on $\mathbb S^2$ (top row) and
$\mathbb S^1\times\mathbb S^2$ (bottom row). Left to right:
normalized bandwidth versus tangential gain $\rho_{\parallel}$,
bandwidth versus orthogonal response $L_{\mathrm{orth}}$,
latent dimension $d$ (varying the complement dimension $d-m$), and
tangential isometry defect $\varepsilon_{\mathrm{iso}}$. Points are five-run means with $95\%$
Student-$t$ intervals; the first column uses paired ratio intervals.
Gray curves are descriptive fits. Red curves show the slope-one
dominant prediction for $\rho_{\parallel}$ and anchored dominant
scalings for $L_{\mathrm{orth}}$ and $d-m$. All sweeps use $n=50$
anchors. }
\label{fig:latent_smoothing_optima}
\end{figure}

\paragraph{Scope of the numerical evidence.}
The experiments verify the local parameter identities of the explicit
encoder--decoder maps and test trends suggested by the small-bandwidth
bound. Some runs, including low-defect configurations, minimize at the
upper grid endpoint. The finite transport discretization and the
uncertified small-bandwidth threshold mean that the fitted curves are
qualitative comparisons, not exact or certified bandwidth predictions.

\subsection{Exploratory MNIST Illustration}
\label{subsec:mnist_benchmark}

We use MNIST to illustrate latent perturbations on image data, separate from
our tests of the geometric bounds. MNIST has no known smooth manifold,
intrinsic metric, or tangent spaces that would verify the hypotheses of
\Cref{thm:latent_decomposition}. For one balanced sample of $n=20$ images
(two per digit), we train the Geometry-Preserving Encoder--Decoder (GPE)
of \citet{lee2025geometry}, with encoder
$f_\theta:\mathbb R^{784}\to\mathbb R^{15}$ and decoder
$g_\phi:\mathbb R^{15}\to\mathbb R^{784}$. We choose it for its small
reconstruction loss on these anchors, which is relevant to the
reconstruction term in \Cref{thm:latent_decomposition}; the measured
displacement is reported in \Cref{tab:network_properties_app}. We compare
decoded latent Gaussian perturbations with Gaussian pixel perturbations
clipped to $[0,1]^{784}$. Both are evaluated against a fixed
$5000$-image MNIST reference using a debiased Sinkhorn divergence, rather
than exact $W_2^2$.

We sweep a dimensionless control $\lambda$: the latent noise standard
deviation is $\lambda s_z$, where $s_z$ is the median pairwise distance
between encoded training images, while the pixel noise standard deviation
is $\lambda$. Thus the two $\lambda$ values are not a common physical
bandwidth. On this single trained split, and after optimizing each control
separately, the decoded latent perturbations attain a lower minimum
Sinkhorn transport value (approximately $54$) than clipped pixel
perturbations (approximately $61$); see \Cref{fig:w2_comparison}. This
exploratory comparison concerns the full encoder--decoder procedure. It
does not isolate a dimensional effect or verify the geometric assumptions
of \Cref{thm:latent_decomposition}, but it motivates studying whether
learned latent smoothing yields similar gains across independent splits
and other image data.
Implementation details, a PCA visualization, and representative images
appear in \Cref{app:mnist_details}.

\begin{figure}[!htbp]
\centering
\begin{tikzpicture}
\node[anchor=south west,inner sep=0] (plot) at (0,0)
{\includegraphics[width=0.65\textwidth]{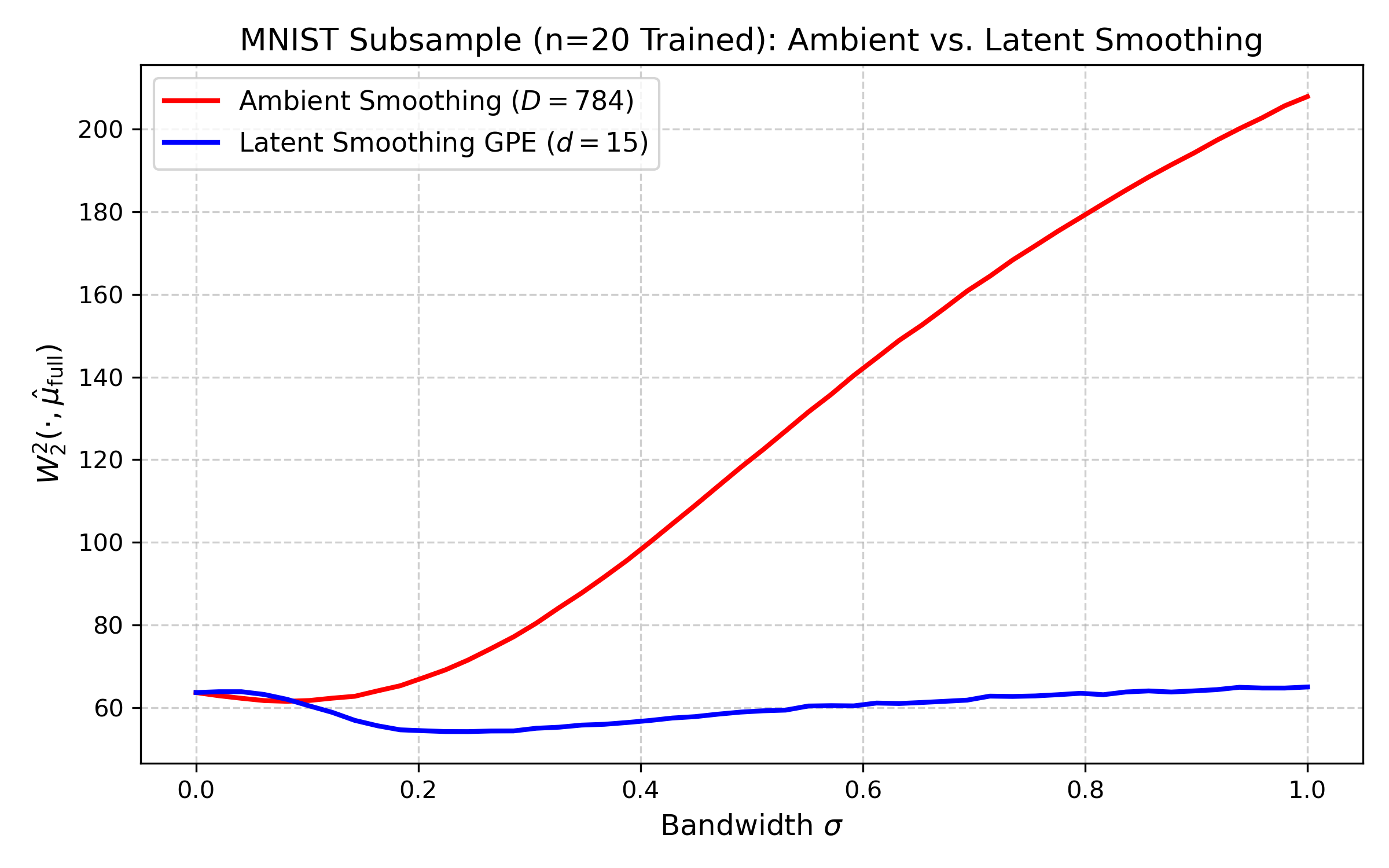}};
\begin{scope}[x={(plot.south east)},y={(plot.north west)}]
\fill[white] (0.28,0.00) rectangle (0.76,0.078);
\node[font=\footnotesize] at (0.52,0.042)
{Dimensionless perturbation scale $\lambda$};
\fill[white] (0.00,0.25) rectangle (0.050,0.75);
\node[rotate=90,font=\footnotesize] at (0.022,0.50)
{Sinkhorn transport value};
\end{scope}
\end{tikzpicture}
\caption{Exploratory MNIST comparison for one trained $n=20$ split.
Blue: Gaussian perturbations in $d=15$ latent coordinates followed by
decoding. Red: Gaussian perturbations in $D=784$ pixel coordinates followed
by clipping. The vertical axis is twice the debiased \texttt{GeomLoss}
Sinkhorn divergence with \texttt{p=2}, not exact $W_2^2$.
The horizontal axis is a dimensionless control with different noise
scalings for the two procedures.}
\label{fig:w2_comparison}
\end{figure}

\section{Conclusion}
\label{sec:conclusion}
We studied how smoothing an empirical measure on a data manifold changes Wasserstein error when noise is added intrinsically, in the ambient space, or in an encoder--decoder representation. On a compact connected smooth boundaryless manifold with a smooth positive target density, the intrinsic improvement follows from a known dimensional heat-flow inequality and heat-kernel entropy asymptotics. Our ambient decomposition isolates normal noise; after combining the intrinsic and normal contributions, the quadratic coefficient contains $D$ and an intrinsic-curvature correction.

The latent bound is the main representation-level result. It tracks reconstruction error, tangential distortion, and decoder amplification of orthogonal noise. Its second-order surrogate gives local bandwidth rules and conditions under which the optimized latent upper-bound surrogate is smaller than the ambient counterpart. These conditions compare bounds, not the exact Wasserstein errors, and their usefulness requires the selected bandwidths to lie in the small-bandwidth regime.

Synthetic experiments on spheres and a product of spheres examine the predicted scaling and parameter trends. A single-split MNIST illustration suggests that decoded latent smoothing merits further study on image data; it does not establish the theorem's geometric assumptions for learned representations or a general performance advantage. Extending the bounds to manifolds with boundary or to nonsmooth metric-measure spaces also requires additional analysis.

\section*{Acknowledgments and Disclosure of Funding}

Wonjun Lee acknowledges support from a startup fund at The Ohio State University.
Computational resources were provided by The Ohio State University.

\appendix
\section{Appendix}
\label{sec:appendix}
{This appendix provides additional technical details, proofs, and
numerical diagnostics for the main results.}

\subsection{A Gaussian tail bound}

\begin{proposition}[Gaussian Tail Integral Bounds]\label{lemma:gaussian_tail}
For any radius $r> \sigma \sqrt{D}$, let $\gamma_\sigma = \mathcal{N}(0,\sigma^2 \mathbf{I}_D)$ denote the isotropic Gaussian measure in $\mathbb{R}^D$ and take $\varepsilon\sim \gamma_\sigma$. The quadratic and zeroth moment tail integrals over the exterior domain $\{\|\varepsilon\| > r\}$ satisfy
\begin{align}
   \mathbb{E} \left[\|\varepsilon\|^2
   \mathbf{1}(\|\varepsilon\|>r) \right]&\le \sqrt{3}\sigma^2 
    D\exp\left(
        -\frac{(r - \sigma\sqrt{D})^2}{4\sigma^2}
    \right), \label{eq:tail_quad} \\
    \mathbb{P}\left( \|\varepsilon\|>r \right) &\le \exp\left(
        -\frac{(r - \sigma\sqrt{D})^2}{2\sigma^2}
    \right).\label{eq:tail_zeroth}
\end{align}
\end{proposition}

\begin{proof}
Let $\varepsilon \sim \gamma_\sigma$. By the {Cauchy--Schwarz} inequality,
   $\mathbb{E}\|\varepsilon\|\le \sqrt{\mathbb{E}\| \varepsilon\|^2} = \sigma\sqrt{D}$.
To prove \eqref{eq:tail_zeroth},
\begin{align*}
   \mathbb{P}(\|\varepsilon\|>r) &= \mathbb{P}(\|\varepsilon\| - \sigma\sqrt{D} > r - \sigma\sqrt{D}) \\
   & \le  \mathbb{P}(\|\varepsilon\| - \mathbb{E}\|\varepsilon\| > r - \sigma\sqrt{D}) \\
   & \le \exp\left(
        - \frac{(r- \sigma\sqrt{D})^2}{2\sigma^2}
   \right),
\end{align*}
where the first inequality follows from $\mathbb{E}\|\varepsilon\|\le \sigma \sqrt{D}$ and
the second inequality follows from the Gaussian concentration inequality
{\citep[Theorem~5.6]{boucheron2013concentration}}. To prove \eqref{eq:tail_quad},
\begin{align*}
    \mathbb{E}[\|\varepsilon\|^2 \mathbf{1}(\|\varepsilon\|> r) ] & = \sum_{i=1}^D\mathbb{E} [|\varepsilon_i|^2  \mathbf{1}(\|\varepsilon\|> r)] \\
    & \le \sum_{i=1}^D  \sqrt{\mathbb{E}|\varepsilon_i|^4}~ \sqrt{\mathbb{P}(\|\varepsilon\|>r)} \\
    & = D\sigma^2\sqrt{3}~ \sqrt{\mathbb{P}(\|\varepsilon\|>r) }\\
    & \le D\sigma^2 \sqrt{3}\exp\left(
     - \frac{(r- \sigma\sqrt{D})^2}{4\sigma^2}
    \right)
\end{align*}
where the first inequality follows from the Cauchy--Schwarz
inequality, the second equality follows from the fourth-moment
formula for a centered Gaussian random variable, and the last inequality
follows from \eqref{eq:tail_zeroth}.

\end{proof}

\section{Additional diagnostics for the latent-smoothing experiment}
\label{app:latent_smoothing_diagnostics}

We collect all five Wasserstein curves by manifold, then discuss
reconstruction-error sensitivity and gather the five parameter checks.
Curves are means over five independent runs, with pointwise
$\pm1$ standard-deviation bands. Circular markers show the mean of the
runwise grid-minimizing bandwidths, located by interpolation on the mean
curves. Insets show
\[
\Delta\widehat W_2^2(\sigma)
:=\widehat W_2^2(\sigma)-\widehat W_2^2(0)
\]
on a symmetric logarithmic vertical scale, retaining the sign of each
change and resolving small improvements near zero. The centering is done
within each run before averaging.

\subsection{Round-sphere Wasserstein curves}
\label{subsec:app_latent_sphere_diagnostics}

\Cref{fig:app_latent_sphere_curves} groups the five sweeps on
$\mathbb S^2$. The first four panels include the direct ambient-smoothing
reference $\hat\mu_n\ast\mathcal N(0,\sigma^2I_D)$; the defect panel
shows latent smoothing only.

\begin{figure}[htbp]
\centering
\begin{minipage}[t]{0.31\textwidth}
\centering
\includegraphics[width=\linewidth]{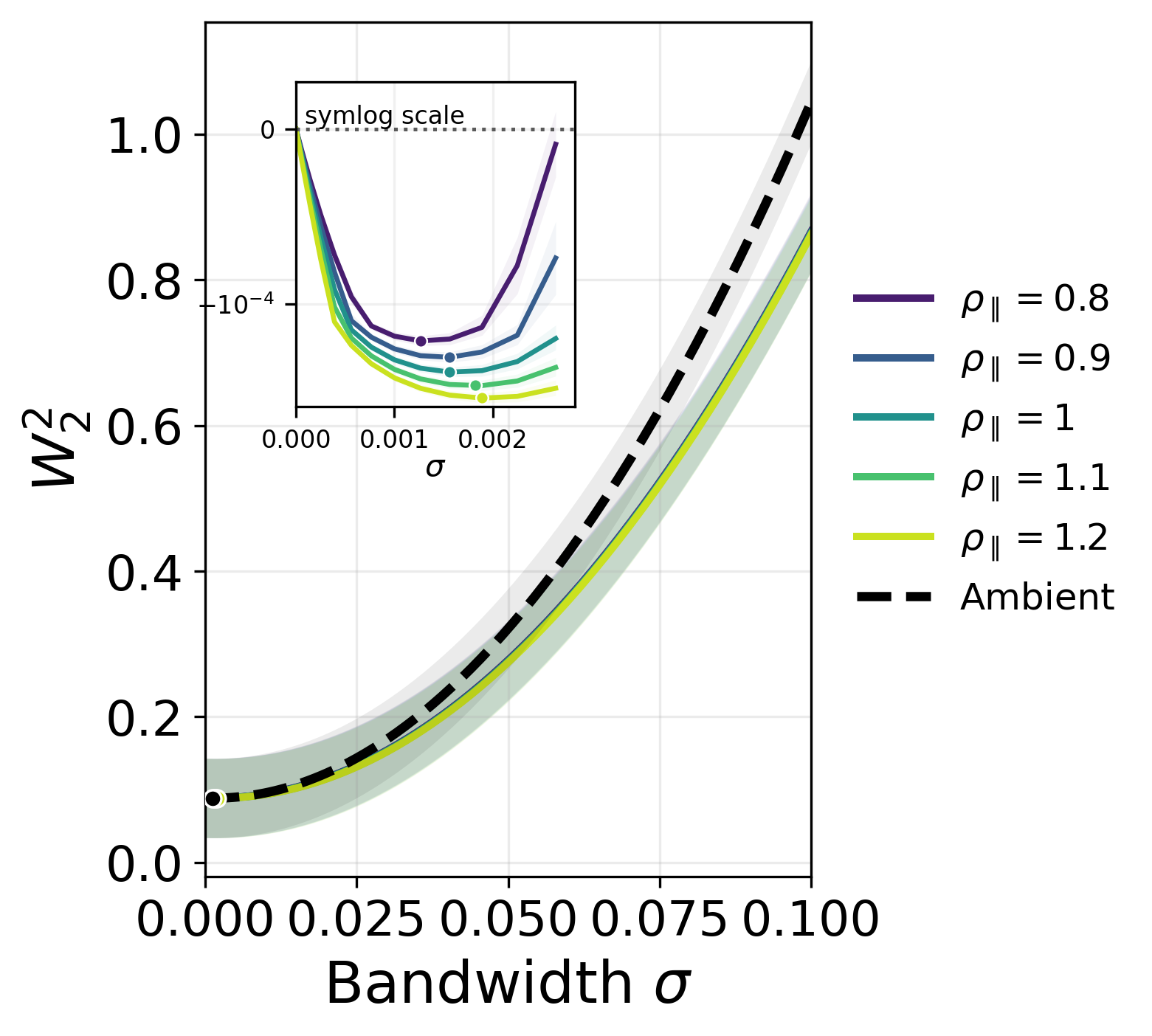}
\end{minipage}\hfill
\begin{minipage}[t]{0.31\textwidth}
\centering
\includegraphics[width=\linewidth]{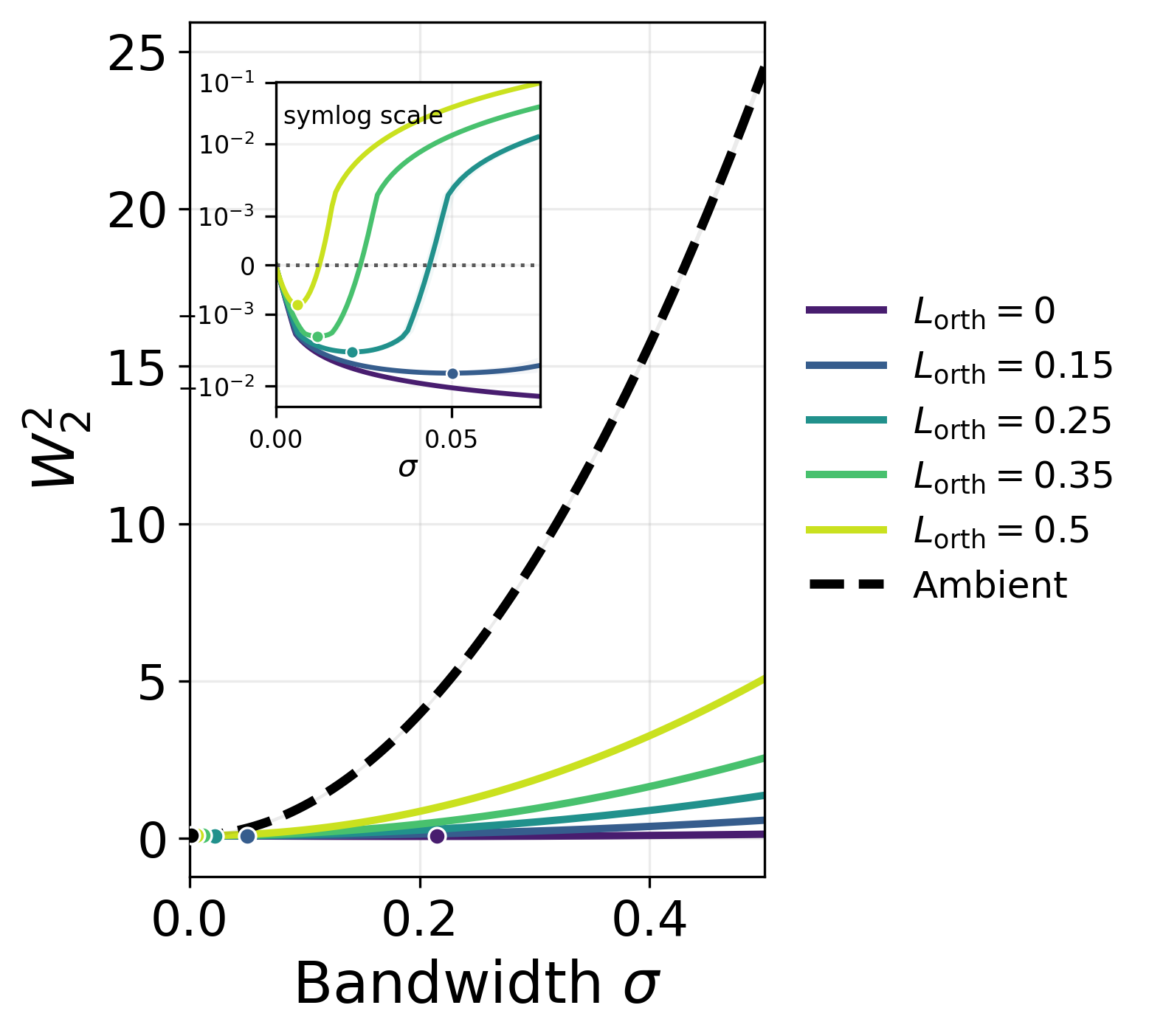}
\end{minipage}\hfill
\begin{minipage}[t]{0.31\textwidth}
\centering
\includegraphics[width=\linewidth]{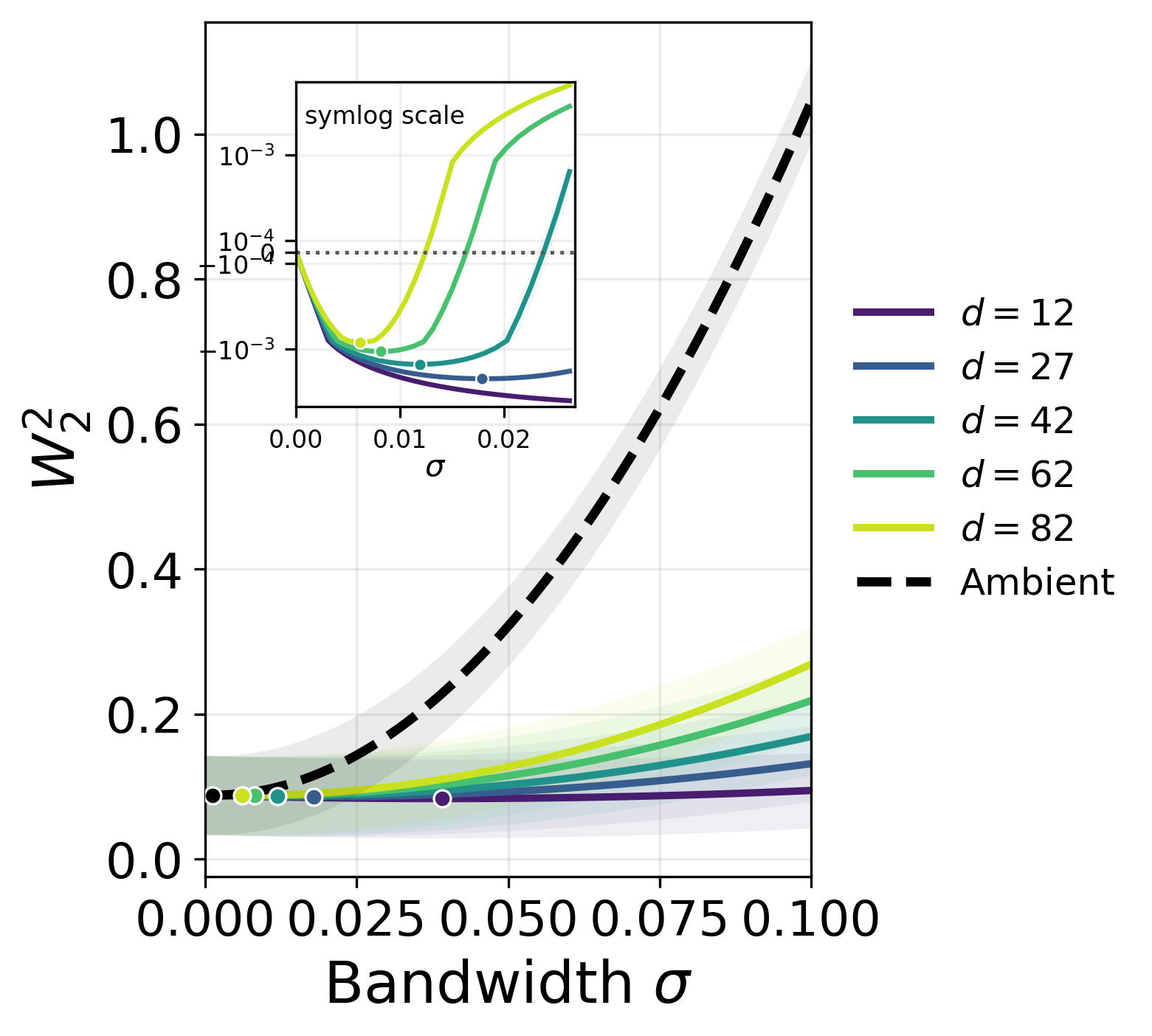}
\end{minipage}
\medskip
\begin{minipage}[t]{0.31\textwidth}
\centering
\includegraphics[width=\linewidth]{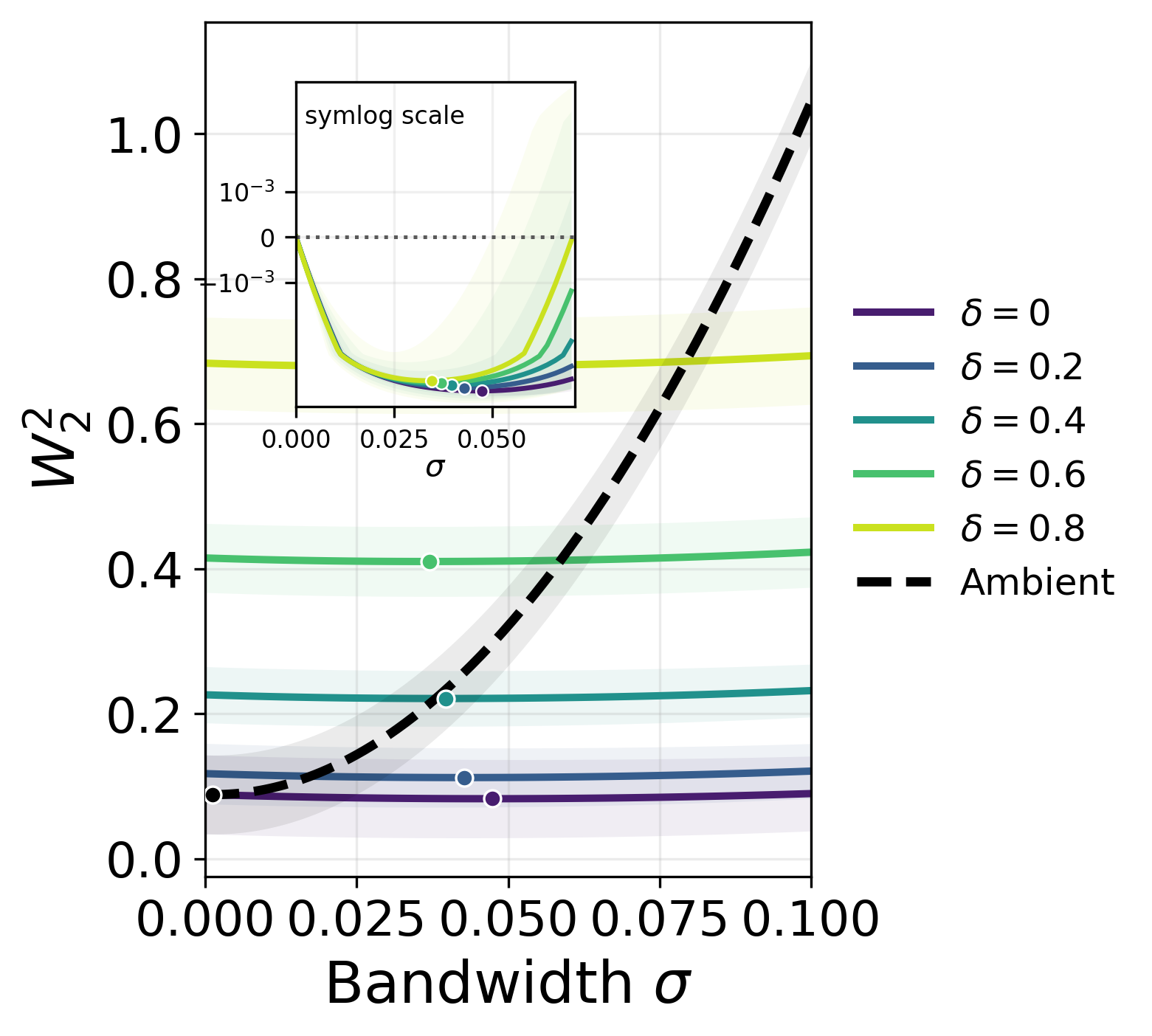}
\end{minipage}\hfill
\begin{minipage}[t]{0.31\textwidth}
\centering
\includegraphics[width=\linewidth]{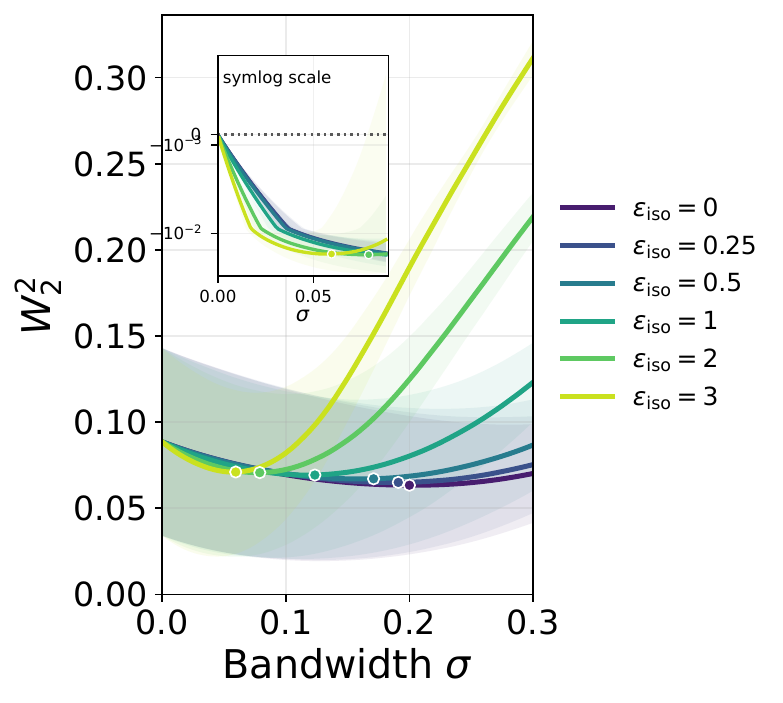}
\end{minipage}\hfill
\begin{minipage}[t]{0.31\textwidth}\mbox{}\end{minipage}
\caption{Empirical $W_2^2$ versus bandwidth on $\mathbb S^2$. Top row, left to right: tangential gain $\rho_{\parallel}$, orthogonal response $L_{\mathrm{orth}}$, and latent dimension $d$ (varying $d-m$). Bottom row: reconstruction error $\delta$ and tangential isometry defect $\varepsilon_{\mathrm{iso}}$. All sweeps use $n=50$ anchors. The first four include a direct ambient-smoothing reference. The defect sweep uses $d=m+2$, $L_{\mathrm{orth}}=0$, and latent smoothing only. The $L_{\mathrm{orth}}$ grid reaches $0.5$, the defect grid reaches $0.3$, and the other grids reach $0.1$. Curves are five-run means with pointwise $\pm1$ standard-deviation bands.}
\label{fig:app_latent_sphere_curves}
\end{figure}

\subsection{Product-sphere Wasserstein curves}
\label{subsec:app_latent_product_diagnostics}

\Cref{fig:app_latent_product_curves} groups the same five sweeps on
$\mathbb S^1\times\mathbb S^2$, using the same panel order and plotting
conventions.

\begin{figure}[htbp]
\centering
\begin{minipage}[t]{0.31\textwidth}
\centering
\includegraphics[width=\linewidth]{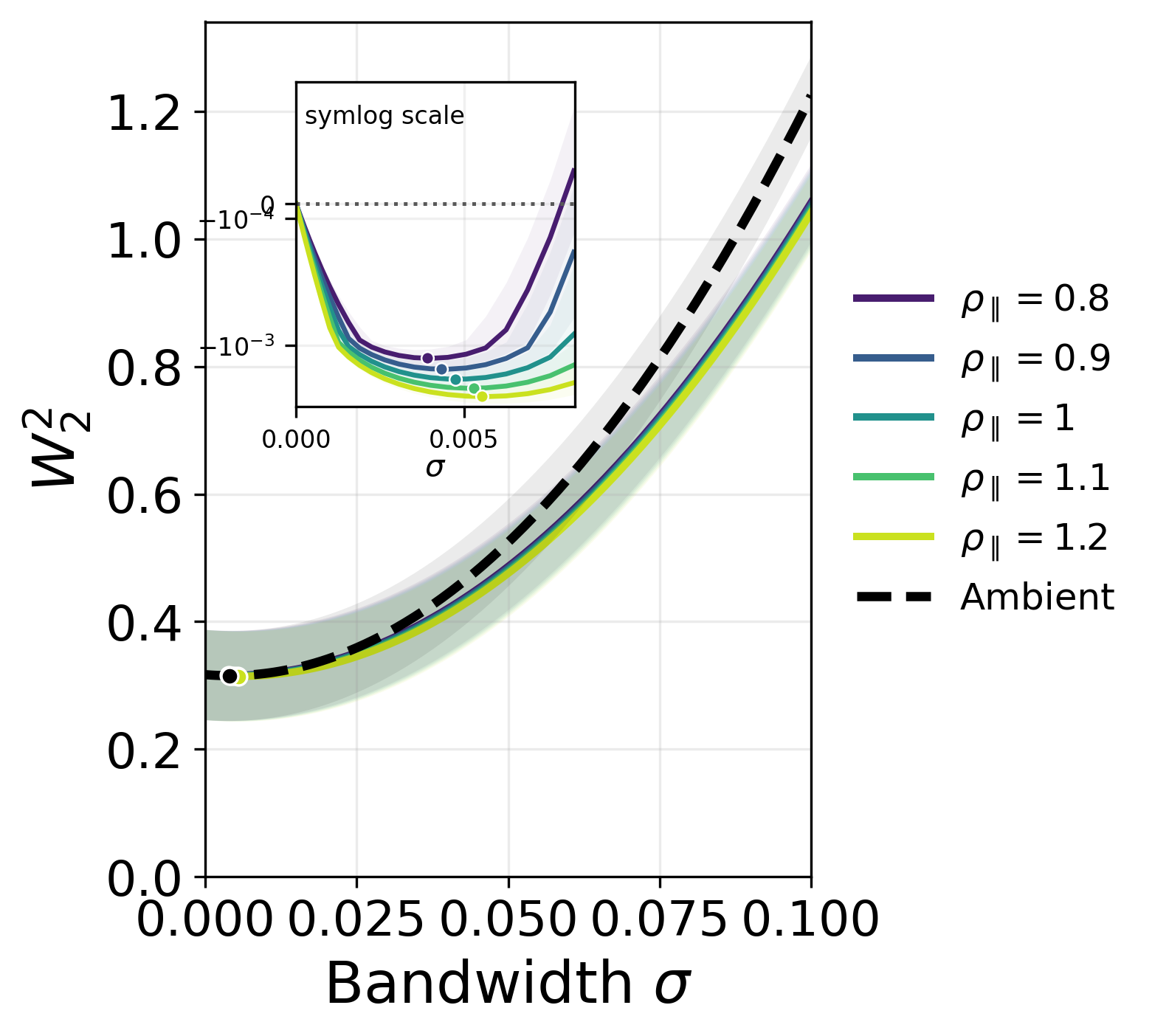}
\end{minipage}\hfill
\begin{minipage}[t]{0.31\textwidth}
\centering
\includegraphics[width=\linewidth]{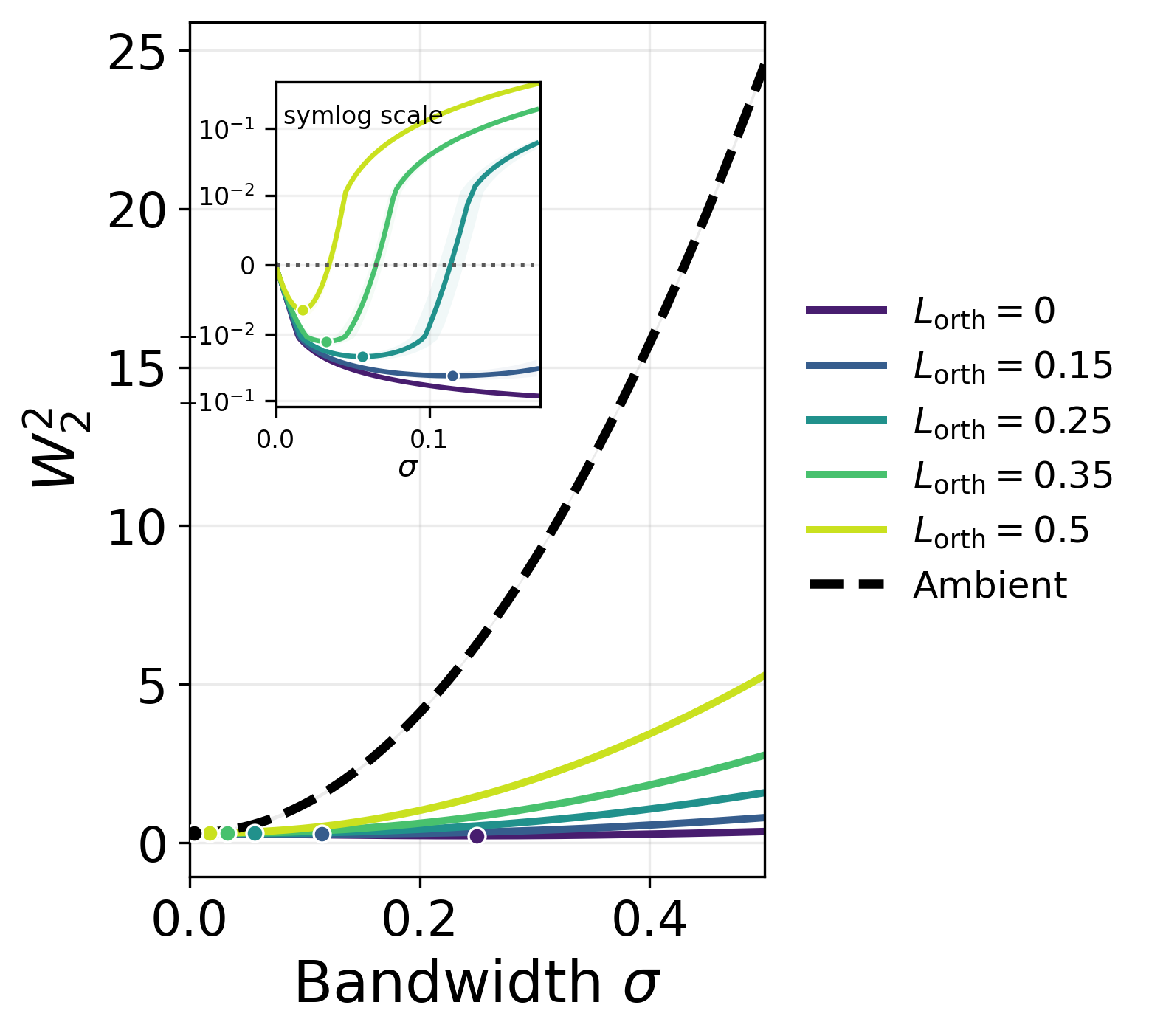}
\end{minipage}\hfill
\begin{minipage}[t]{0.31\textwidth}
\centering
\includegraphics[width=\linewidth]{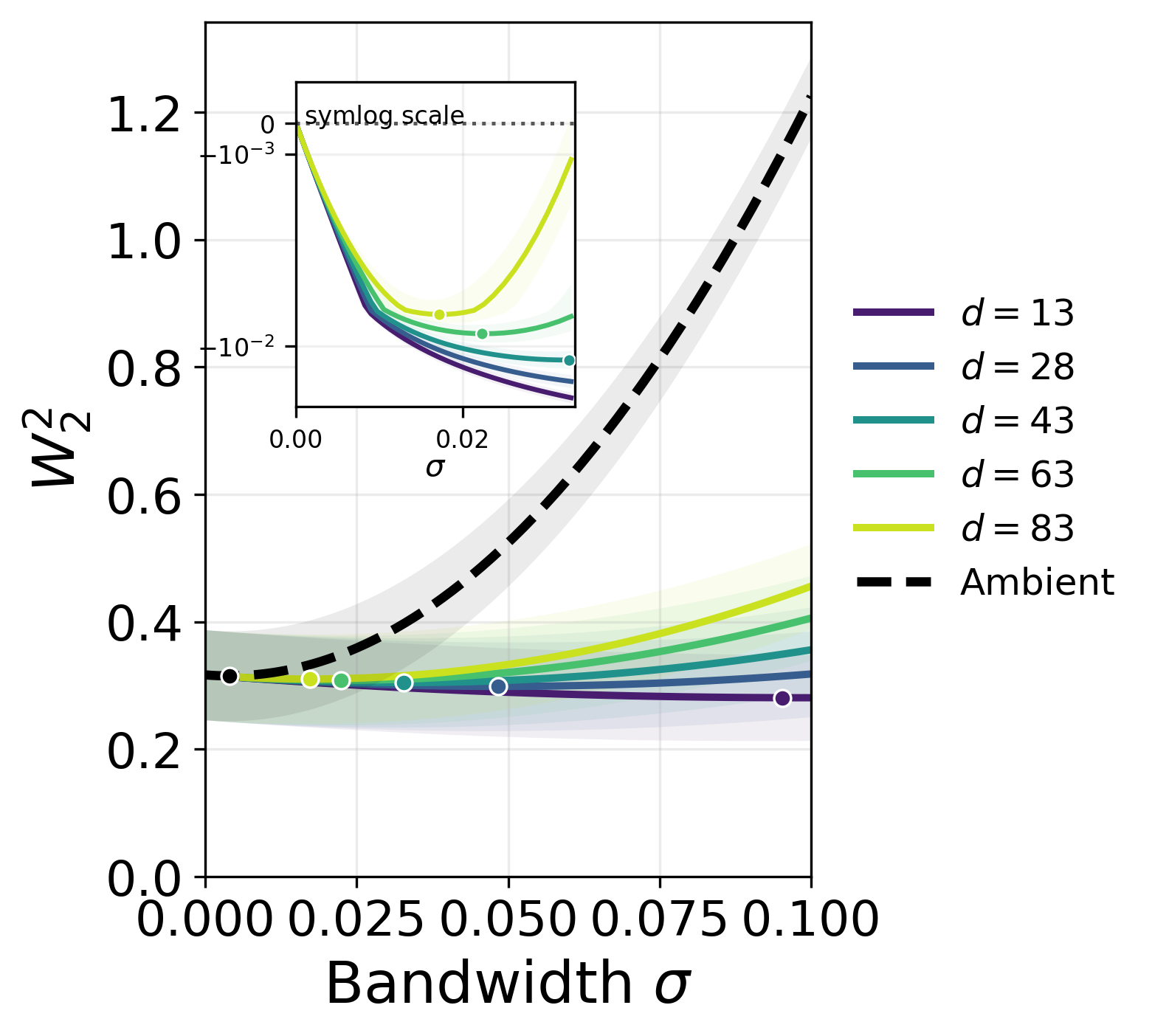}
\end{minipage}
\medskip
\begin{minipage}[t]{0.31\textwidth}
\centering
\includegraphics[width=\linewidth]{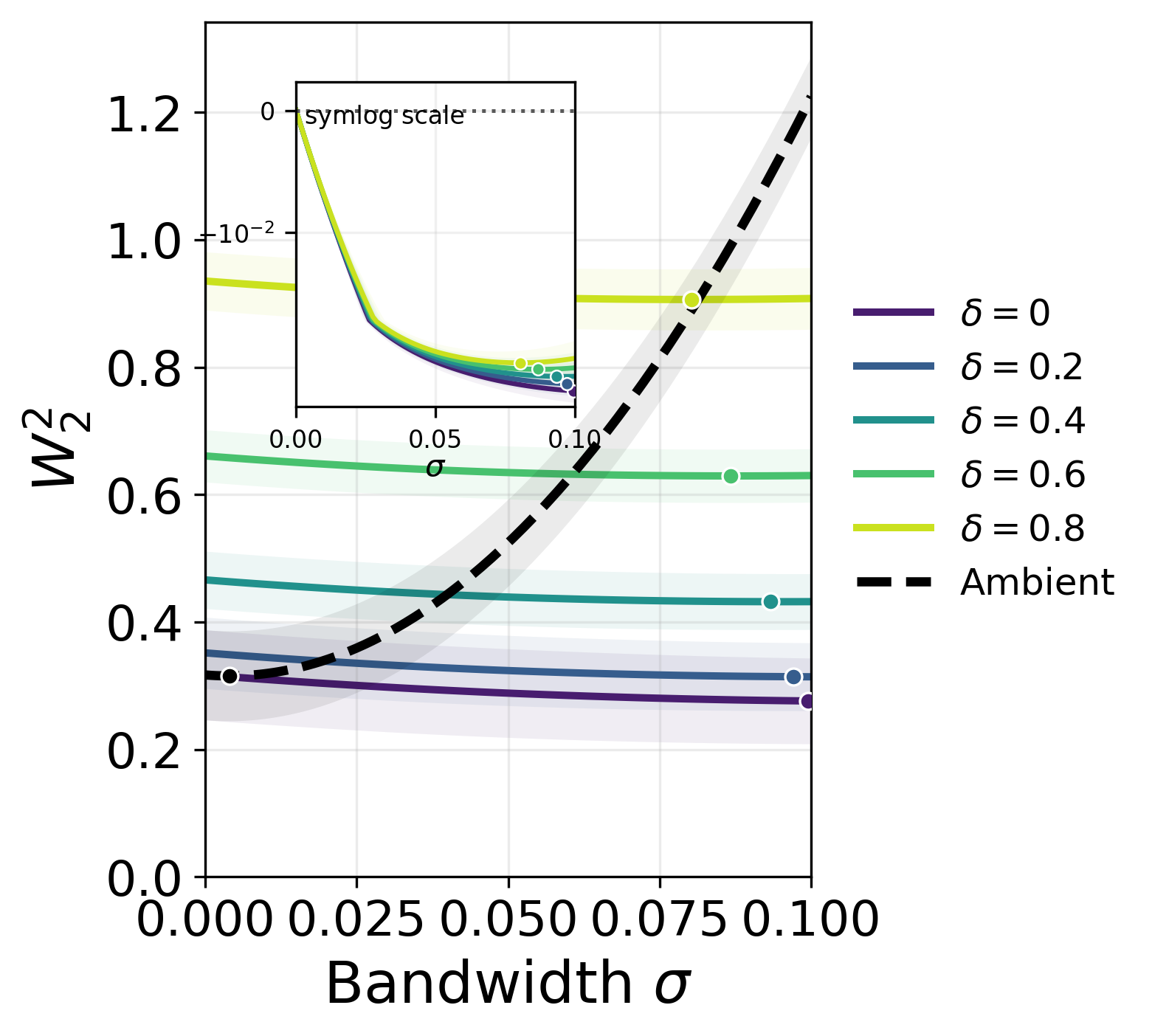}
\end{minipage}\hfill
\begin{minipage}[t]{0.31\textwidth}
\centering
\includegraphics[width=\linewidth]{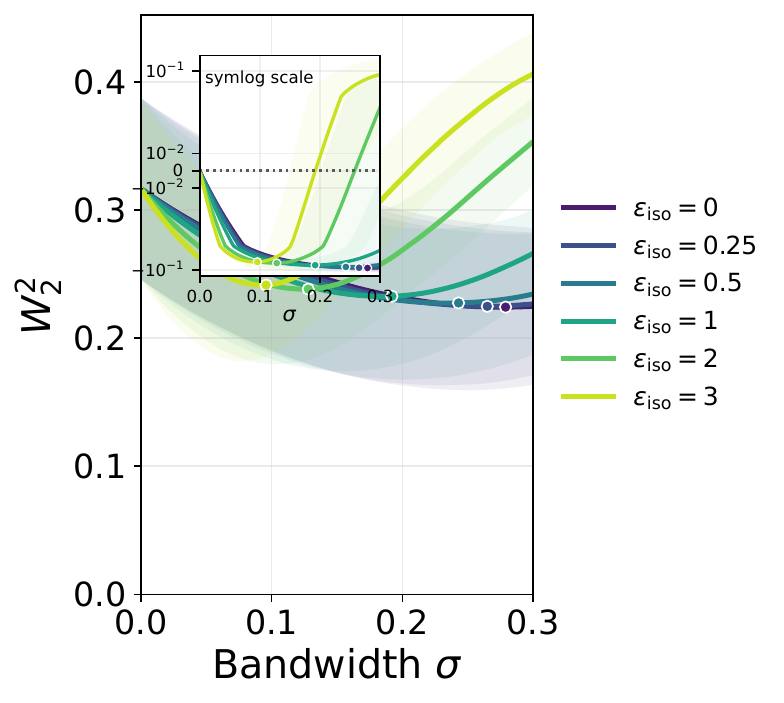}
\end{minipage}\hfill
\begin{minipage}[t]{0.31\textwidth}\mbox{}\end{minipage}
\caption{Empirical $W_2^2$ versus bandwidth on $\mathbb S^1\times\mathbb S^2$. Top row, left to right: tangential gain $\rho_{\parallel}$, orthogonal response $L_{\mathrm{orth}}$, and latent dimension $d$ (varying $d-m$). Bottom row: reconstruction error $\delta$ and tangential isometry defect $\varepsilon_{\mathrm{iso}}$. All sweeps use $n=50$ anchors. The first four include a direct ambient-smoothing reference. The defect sweep uses $d=m+2$, $L_{\mathrm{orth}}=0$, and latent smoothing only. The $L_{\mathrm{orth}}$ grid reaches $0.5$, the defect grid reaches $0.3$, and the other grids reach $0.1$. Curves are five-run means with pointwise $\pm1$ standard-deviation bands.}
\label{fig:app_latent_product_curves}
\end{figure}

\subsection{Reconstruction-error sensitivity}
\label{subsec:latent_smoothing_delta}

The lower-left panel of each figure shows the reconstruction-error sweep.
Changing $\delta$ translates every decoded output by the same vector
$\delta b$ without changing the decoder Jacobian. Among the zero-defect
sweeps, it produces the largest change in the vertical level of the
Wasserstein curves. This agrees with the leading term $(W_0+\delta)^2$
in \Cref{thm:latent_decomposition}: reconstruction error is already
present at $\sigma=0$. It also enters the bandwidth-dependent
coefficients, so these curves provide a sensitivity comparison rather
than a universal law for $\sigma^*(\delta)$.

\subsection{Verification of the five prescribed parameters}
\label{subsec:latent_parameter_verification}

\Cref{fig:app_latent_parameter_verification}
collects the checks of $\delta$, $\rho_{\parallel}$,
$L_{\mathrm{orth}}$, $d-m$, and $\varepsilon_{\mathrm{iso}}$ for each
manifold. The first four checks use the zero-defect construction; the
fifth uses the separate defect construction. Jacobian-based responses
are measured by central finite differences at the empirical anchors.

\begin{figure}[htbp]
\centering
\begin{minipage}[t]{0.31\textwidth}
\centering
\includegraphics[width=\linewidth,trim=0bp 277bp 705bp 0bp,clip]{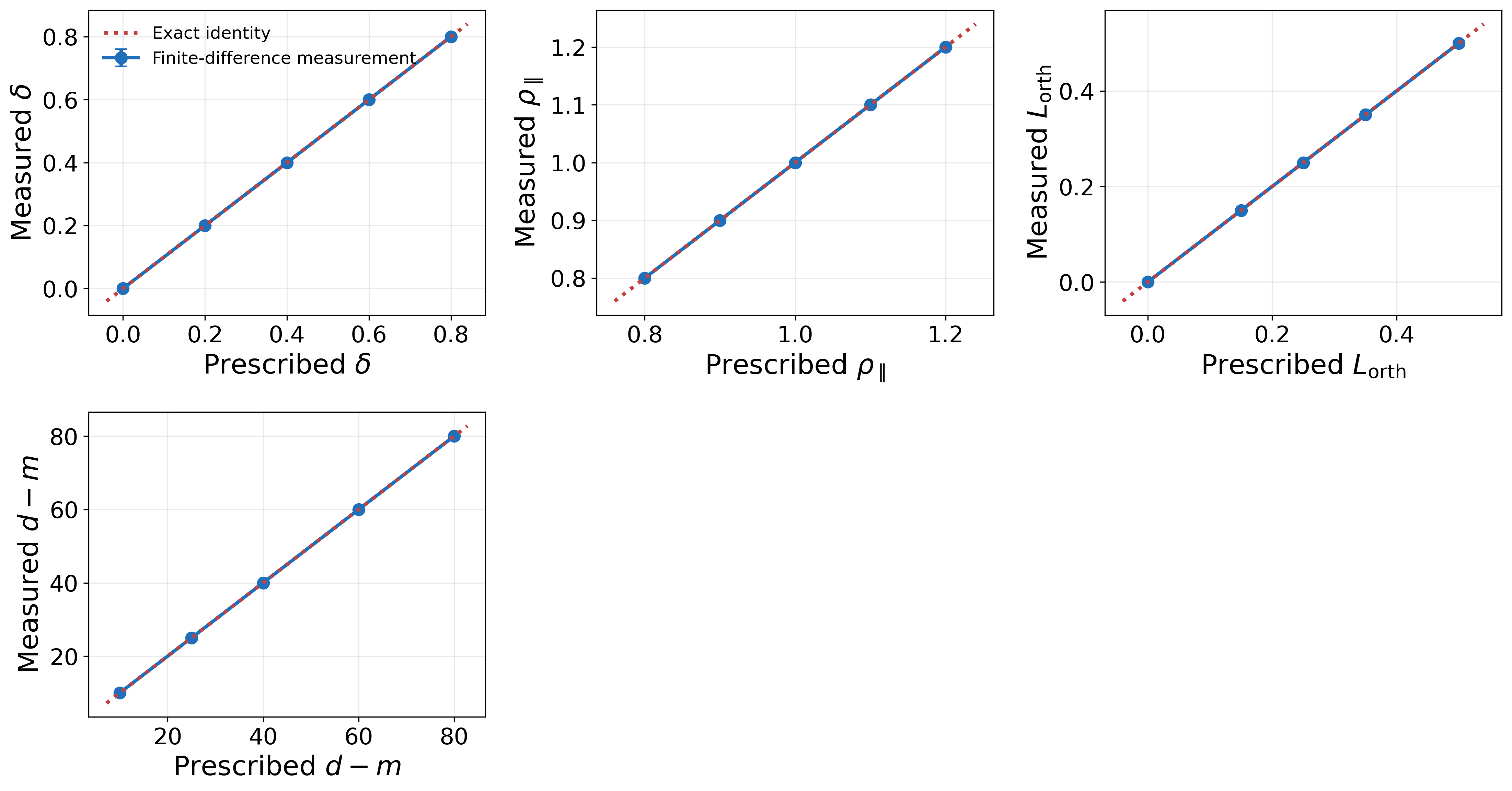}
\end{minipage}\hfill
\begin{minipage}[t]{0.31\textwidth}
\centering
\includegraphics[width=\linewidth,trim=352bp 277bp 352bp 0bp,clip]{figures/explicit_global_encoder_decoder_results_new/sphere/sphere_explicit_parameter_verification.png}
\end{minipage}\hfill
\begin{minipage}[t]{0.31\textwidth}
\centering
\includegraphics[width=\linewidth,trim=705bp 277bp 0bp 0bp,clip]{figures/explicit_global_encoder_decoder_results_new/sphere/sphere_explicit_parameter_verification.png}
\end{minipage}
\medskip
\begin{minipage}[t]{0.31\textwidth}
\centering
\includegraphics[width=\linewidth,trim=0bp 0bp 705bp 276bp,clip]{figures/explicit_global_encoder_decoder_results_new/sphere/sphere_explicit_parameter_verification.png}
\end{minipage}\hfill
\begin{minipage}[t]{0.31\textwidth}
\centering
\includegraphics[width=\linewidth,trim=0bp 245bp 670bp 0bp,clip]{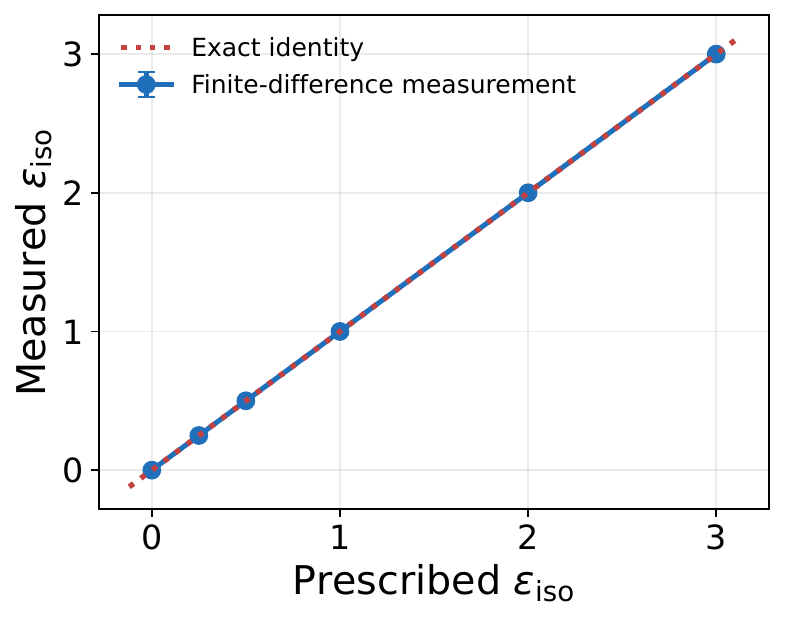}
\end{minipage}\hfill
\begin{minipage}[t]{0.31\textwidth}\mbox{}\end{minipage}
\par\medskip
\begin{minipage}[t]{0.31\textwidth}
\centering
\includegraphics[width=\linewidth,trim=0bp 277bp 705bp 0bp,clip]{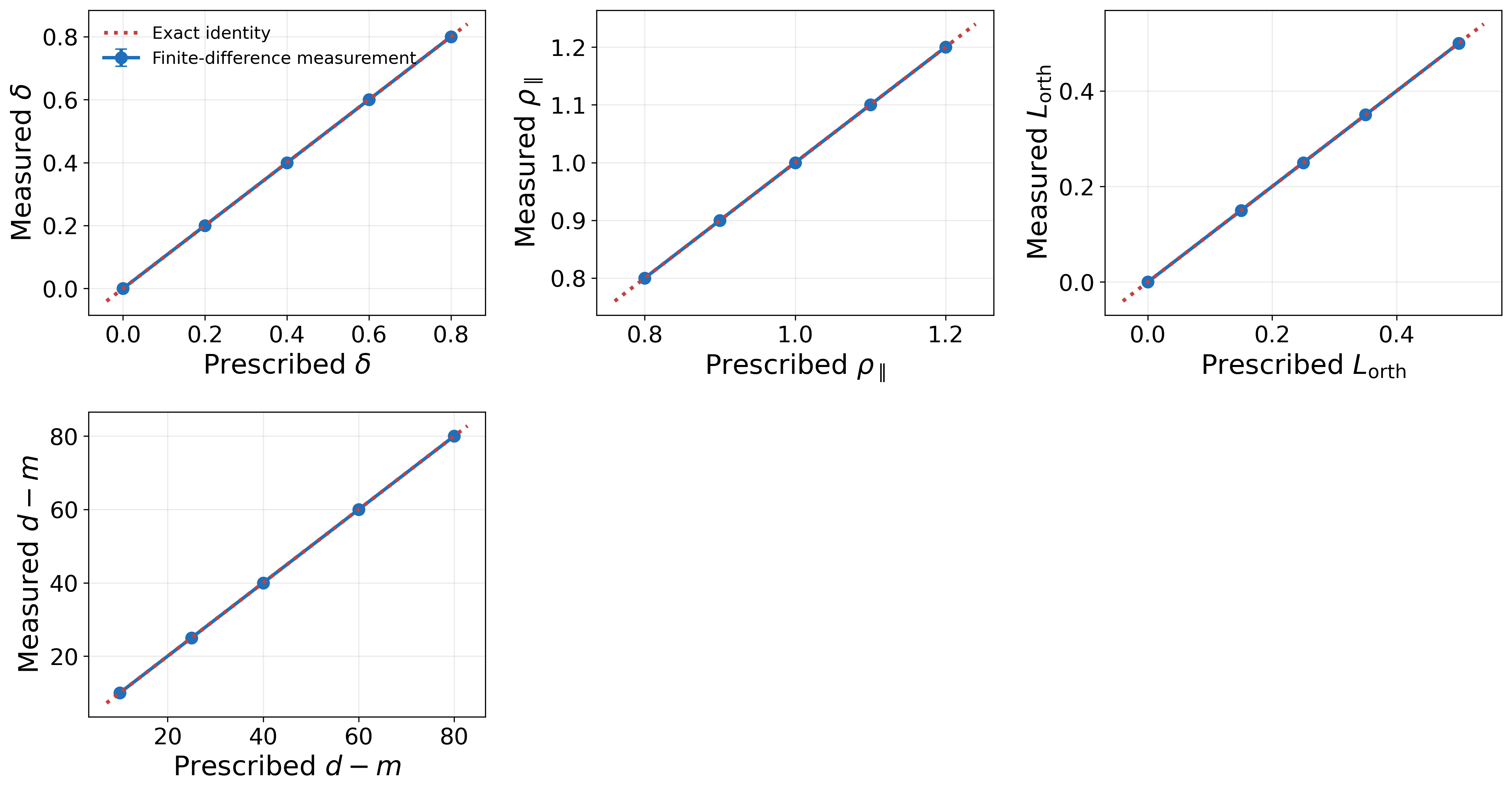}
\end{minipage}\hfill
\begin{minipage}[t]{0.31\textwidth}
\centering
\includegraphics[width=\linewidth,trim=352bp 277bp 352bp 0bp,clip]{figures/explicit_global_encoder_decoder_results_new/product/product_explicit_parameter_verification.png}
\end{minipage}\hfill
\begin{minipage}[t]{0.31\textwidth}
\centering
\includegraphics[width=\linewidth,trim=705bp 277bp 0bp 0bp,clip]{figures/explicit_global_encoder_decoder_results_new/product/product_explicit_parameter_verification.png}
\end{minipage}
\medskip
\begin{minipage}[t]{0.31\textwidth}
\centering
\includegraphics[width=\linewidth,trim=0bp 0bp 705bp 276bp,clip]{figures/explicit_global_encoder_decoder_results_new/product/product_explicit_parameter_verification.png}
\end{minipage}\hfill
\begin{minipage}[t]{0.31\textwidth}
\centering
\includegraphics[width=\linewidth,trim=0bp 245bp 670bp 0bp,clip]{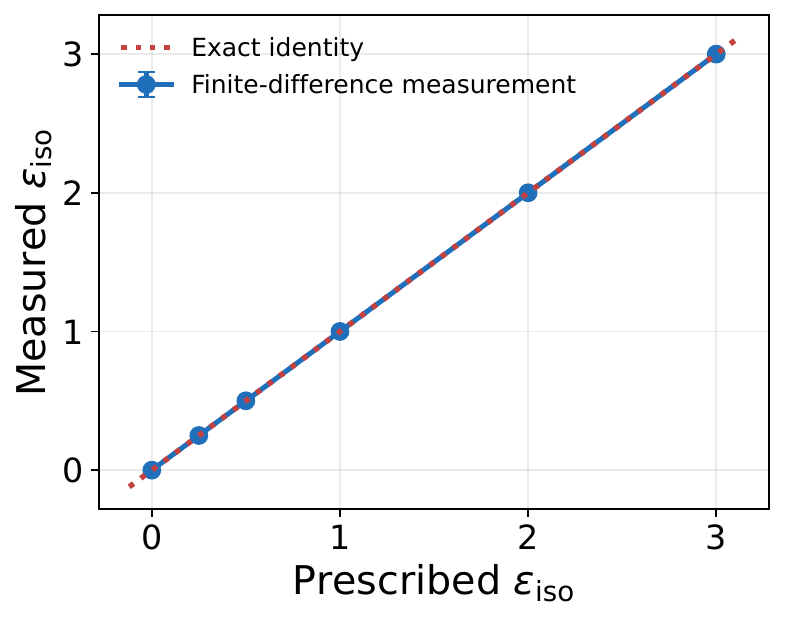}
\end{minipage}\hfill
\begin{minipage}[t]{0.31\textwidth}\mbox{}\end{minipage}
\caption{Finite-difference verification of the five prescribed parameters on $\mathbb S^2$ (top two rows) and $\mathbb S^1\times\mathbb S^2$ (bottom two rows). Within each pair of rows, the first shows reconstruction error $\delta$, tangential gain $\rho_{\parallel}$, and orthogonal response $L_{\mathrm{orth}}$ from left to right; the second shows latent complement dimension $d-m$ and tangential isometry defect $\varepsilon_{\mathrm{iso}}$. All sweeps use $n=50$ anchors. The first four parameters use the zero-defect construction, while $\varepsilon_{\mathrm{iso}}$ uses the separate defect construction. Blue finite-difference measurements agree with the red identity lines.}
\label{fig:app_latent_parameter_verification}
\end{figure}

\section{MNIST Implementation Details and Diagnostic Metrics}
\label{app:mnist_details}

This section provides the architecture, training objectives, optimization
settings, local-frame construction, and diagnostic formulas used in the MNIST
study of \Cref{subsec:mnist_benchmark}. The synthetic experiments are used to
evaluate the geometric quantities in settings with known tangent spaces. For
MNIST, the local-frame calculations below are exploratory diagnostics of the
trained network relative to data-estimated PCA directions.

\subsection{Network Architecture and Training Protocol}

We use the Geometry-Preserving Encoder--Decoder (GPE) of
\citet{lee2025geometry} because it achieves a small reconstruction loss
with only $20$ training images. This matters for our latent bound, which
depends explicitly on the reconstruction displacement $\delta$; its
measured value is reported in \Cref{tab:network_properties_app}. We train
the GPE on the balanced empirical measure
\[
    \hat\mu_n
    =
    \frac1{20}\sum_{i=1}^{20}\delta_{x_i},
\]
which contains two images from each MNIST digit class. The sample is selected
using random seed $42$. The encoder and decoder are trained sequentially for
$2{,}000$ epochs each, using full batches and the Adam optimizer with learning
rate $10^{-3}$.

\paragraph{Encoder training.}
The encoder $f_\theta:\mathbb R^{784}\to\mathbb R^{15}$ is a fully connected
network with layer dimensions
\[
    784\longrightarrow512\longrightarrow256\longrightarrow128
    \longrightarrow15.
\]
Softplus activations are applied after the three hidden layers. The encoder
parameters are trained by minimizing
\begin{equation}
\label{eq:mnist_encoder_training_loss}
    \mathcal L_{\mathrm{enc}}(\theta)
    =
    \mathbb E_{x_i,x_j\sim\hat\mu_n}
    \left[
        \left(
            \log
            \frac{
                1+\|f_\theta(x_i)-f_\theta(x_j)\|_2^2
            }{
                1+\|x_i-x_j\|_2^2
            }
        \right)^2
    \right].
\end{equation}
This objective penalizes discrepancies between pairwise distances of the
training anchors and their latent representations. Because it is evaluated on
a finite set of anchors, it does not by itself determine the Jacobians
$J_f(x_i)$ or verify local isometry on an unknown MNIST tangent space.

\paragraph{Decoder training.}
After freezing the encoder, the decoder
$g_\phi:\mathbb R^{15}\to[0,1]^{784}$ is trained. Its layer dimensions are
\[
    15\longrightarrow128\longrightarrow256\longrightarrow512
    \longrightarrow784.
\]
The three hidden layers use Softplus activations, and the output layer uses a
sigmoid activation. At every training iteration, independent noise
\[
    \xi\sim\mathcal N(0,\sigma_{\mathrm{train}}^2I_{15}),
    \qquad
    \sigma_{\mathrm{train}}=0.05,
\]
is added to the latent representation. The decoder objective is
\begin{equation}
\label{eq:mnist_decoder_training_loss_app}
    \mathcal L_{\mathrm{dec}}(\phi)
    =
    \mathbb E_{x\sim\hat\mu_n,\,\xi}
    \left[
        \|g_\phi(f_\theta(x)+\xi)-x\|_2^2
    \right].
\end{equation}
The implementation averages the squared error over pixels, which differs from
\eqref{eq:mnist_decoder_training_loss_app} only by the fixed factor $784$ and
therefore has the same minimizers.

The latent noise is intended to limit the decoder's response to perturbations,
including directions orthogonal to an encoded tangent space. To see the local
mechanism, write $z_i=f_\theta(x_i)$ and $J_i=J_{g_\phi}(z_i)$. A Taylor
expansion at $z_i$, together with the bias--variance identity, gives
\[
\mathbb E_\xi\|g_\phi(z_i+\xi)-x_i\|_2^2
=\|g_\phi(z_i)-x_i\|_2^2
+\sigma_{\mathrm{train}}^2\bigl(
    \|J_i\|_{\mathrm F}^2
    +\langle g_\phi(z_i)-x_i,\Delta g_\phi(z_i)\rangle
  \bigr)
+o(\sigma_{\mathrm{train}}^2).
\]
In the manifold setting of \Cref{thm:latent_decomposition}, let
$P_{V_i^\perp}$ project onto the complement of the encoded tangent space
$V_i=J_f(x_i)(T_{x_i}M)$. Then
\[
\|J_i\|_{\mathrm F}^2
\geq\|J_iP_{V_i^\perp}\|_{\mathrm F}^2
\geq\|J_iP_{V_i^\perp}\|_{\mathrm{op}}^2.
\]
Thus, when reconstruction error is small, the positive Jacobian term
encourages a smaller decoder response in the directions that enter
$L_{\mathrm{orth}}$. This motivates adding $\xi$; finite-noise
training does not by itself guarantee a smaller worst-case
$L_{\mathrm{orth}}$, which must be assessed separately.

\subsection{Reconstruction and Local-Frame Diagnostics}

MNIST does not provide a known intrinsic dimension or ground-truth tangent
spaces. We therefore do not report an empirical encoder-isometry defect. The
quantities defined below measure the network response relative to specified
local PCA frames and are not used as estimates of the intrinsic constants in
\Cref{thm:latent_decomposition}.

\paragraph{Local PCA frames.}
For each training anchor $x_i$, we construct an orthonormal matrix
    $E_i\in\mathbb R^{784\times m}$
by anchor-centered, distance-weighted PCA using same-class neighbors from the
full MNIST training collection. The offset $x_j-x_i$ receives the weight
\[
    w_{ij}
    =
    \exp\!\left[
        -\frac12
        \left(
            \frac{\|x_j-x_i\|_2}{h_i}
        \right)^2
    \right].
\]
The primary configuration uses $m=10$, $k=30$ neighbors, and
\[
    h_i
    =
    0.75\,
    \operatorname{median}_{j\in\mathcal N_i}
    \|x_j-x_i\|_2.
\]
As a sensitivity check, we also use $k=50$ neighbors with
\[
    h_i
    =
    \operatorname{median}_{j\in\mathcal N_i}
    \|x_j-x_i\|_2.
\]
The larger MNIST collection is used only after training, for this local-frame
construction and for the finite reference measure in the transport
comparison.

\paragraph{Reconstruction displacement.}
The empirical RMS Euclidean reconstruction displacement is
\begin{equation}
\label{eq:mnist_delta_diagnostic}
    \widehat\delta
    :=
    \left(
        \frac1n\sum_{i=1}^n
        \|x_i-g_\phi(f_\theta(x_i))\|_2^2
    \right)^{1/2}.
\end{equation}
This Euclidean displacement should be distinguished from the pixel-averaged
mean squared error used during decoder training.

\paragraph{Encoded local frame.}
For each estimated frame, define 
$
    B_i
    :=
    J_f(x_i)E_i
    \in\mathbb R^{15\times m}.
$
We record the numerical rank and extreme singular values of $B_i$ to determine
whether the selected $m$-dimensional frame remains $m$-dimensional after
encoding. These values are numerical checks on the selected PCA frames, not
estimates of an intrinsic encoder-isometry constant.

\paragraph{Shared PCA-frame tangential gain and residual.}
The composed Jacobian acting on the estimated local frame is
\[
    A_i
    :=
    J_g(f_\theta(x_i))B_i
    =
    J_g(f_\theta(x_i))J_f(x_i)E_i
    \in\mathbb R^{784\times m}.
\]
Relative to these frames, we define the shared tangential gain and its
operator-norm residual by
\begin{align}
    \widehat\rho_\parallel
    &\in
    \arg\min_{\rho\geq0}
    \max_{1\leq i\leq n}
    \min_{O_i\in O(m)}
    \|A_i-\rho E_iO_i\|_{\mathrm{op}},
    \label{eq:mnist_global_gain_diagnostic}\\
    \widehat\varepsilon_{\mathrm{iso}}
    &:=
    \max_{1\leq i\leq n}
    \min_{O_i\in O(m)}
    \|A_i-\widehat\rho_\parallel E_iO_i\|_{\mathrm{op}}.
    \label{eq:mnist_global_isometry_diagnostic}
\end{align}
Both quantities remain conditional on the chosen PCA frames.

For the numerical optimization, write
\[
    E_i^\top A_i
    =
    U_i\Sigma_iW_i^\top
\]
and initialize the orthogonal factor by $O_{i,0}=U_iW_i^\top$. This is the
orthogonal Procrustes solution for the Frobenius norm and is used only as an
initialization for the operator-norm objective. At fixed $\rho$, we refine
$O_i$ through the parameterization
\[
    O_i
    =
    O_{i,0}\exp(K_i-K_i^\top).
\]
Both connected components of $O(m)$ are considered. The refinement uses
$300$ Adam iterations with learning rate $3\times10^{-2}$, together with one
random orthogonal restart. We alternate this refinement with a bounded
one-dimensional minimization over the shared variable $\rho$ for at most four
iterations. Final operator norms are evaluated by singular value
decomposition rather than by the smooth objective used during refinement.

As an independent numerical check, we approximate
$J_{g\circ f}(x_i)E_i$ by centered finite differences with step size
$10^{-2}$ and compare the result with the automatic-differentiation
Jacobian.

\paragraph{Latent-complement decoder response.}
Let
    $r_i:=\operatorname{rank}(B_i),$
let $Q_i\in\mathbb R^{15\times r_i}$ have orthonormal columns spanning
$\operatorname{range}(B_i)$, and let
$Q_i^\perp\in\mathbb R^{15\times(15-r_i)}$ span its orthogonal complement.
We define the PCA-frame latent-complement response by
\begin{equation}
\label{eq:mnist_orthogonal_response_diagnostic}
    \widehat L_{\mathrm{orth}}
    :=
    \max_{1\leq i\leq n}
    \left\|
        J_g(f_\theta(x_i))Q_i^\perp
    \right\|_{\mathrm{op}}.
\end{equation}
The complement is constructed from
$\operatorname{range}(J_f(x_i)E_i)$ rather than from the range of the full
encoder Jacobian. If $r_i=15$, the complement is empty and its contribution is
zero by convention.

\begin{table}[h]
\centering
\caption{Empirical diagnostics for the GPE trained on $\hat\mu_n$, with
$n=20$, $d=15$, and $D=784$. The local PCA calculations use the estimated
dimension $m=10$ and the primary $k=30$ neighborhood configuration. The
reconstruction displacement is computed directly at the training anchors.
The remaining quantities are conditional on the data-estimated PCA frames
and are not interpreted as intrinsic geometric constants in
\Cref{thm:latent_decomposition}.}
\label{tab:network_properties_app}
\begin{tabular}{lcc}
\hline
Metric & Symbol & Empirical value \\ \hline
RMS reconstruction displacement
& $\widehat\delta$
& $1.87\times10^{-3}$ \\
PCA-frame tangential gain
& $\widehat\rho_\parallel$
& $1.98\times10^{-3}$ \\
PCA-frame tangential residual
& $\widehat\varepsilon_{\mathrm{iso}}$
& $4.34\times10^{-3}$ \\
Latent-complement decoder response
& $\widehat L_{\mathrm{orth}}$
& $3.83\times10^{-3}$ \\ \hline
\end{tabular}
\end{table}

\subsection{Evaluation Protocol and Qualitative MNIST Diagnostics}

For the comparison in \Cref{fig:w2_comparison}, let
\[
    \hat\mu_{\mathrm{ref}}
    =\frac1{N_{\mathrm{eval}}}
    \sum_{j=1}^{N_{\mathrm{eval}}}\delta_{x_j^{\mathrm{ref}}},
    \qquad N_{\mathrm{eval}}=5000,
\]
be a fixed reference sample from the larger MNIST training collection.
We also generate $N_{\mathrm{eval}}$ samples from each perturbed
procedure. Because exact discrete optimal transport at this size is
computationally expensive, we evaluate with the debiased Sinkhorn
divergence implemented by \texttt{GeomLoss} \citep{feydy2019interpolating},
using \texttt{p=2}, \texttt{blur=0.01}, and \texttt{scaling=0.8}.
Since \texttt{GeomLoss} uses the ground cost $\frac12\|x-y\|_2^2$
for \texttt{p=2}, we multiply its output by two. We call this quantity
the \emph{Sinkhorn transport value}; it is not exact $W_2^2$.

The scale of the learned latent representation is arbitrary, so we set
\[
    s_z:=\operatorname{median}_{i<j}
    \|f_\theta(x_i)-f_\theta(x_j)\|_2
\]
and, at dimensionless control $\lambda$, generate samples using
\[
    z_i^\lambda=f_\theta(x_i)+\lambda s_z\zeta,
    \quad \zeta\sim\mathcal N(0,I_{15}),
    \qquad
    x_i^\lambda=\operatorname{clamp}(x_i+\lambda\eta,0,1),
    \quad \eta\sim\mathcal N(0,I_{784}).
\]
The latent sample is $g_\phi(z_i^\lambda)$. The two procedures have
different noise standard deviations at the same $\lambda$, and the
clipped pixel baseline differs from the exact ambient Gaussian
convolution in \Cref{thm:decomposition}. The encoder and decoder are
trained on the $n=20$ anchors; the larger MNIST collection is used for
the reference comparison and for the local PCA frames below.

\begin{figure}[t]
\centering
\includegraphics[width=0.79\textwidth]
{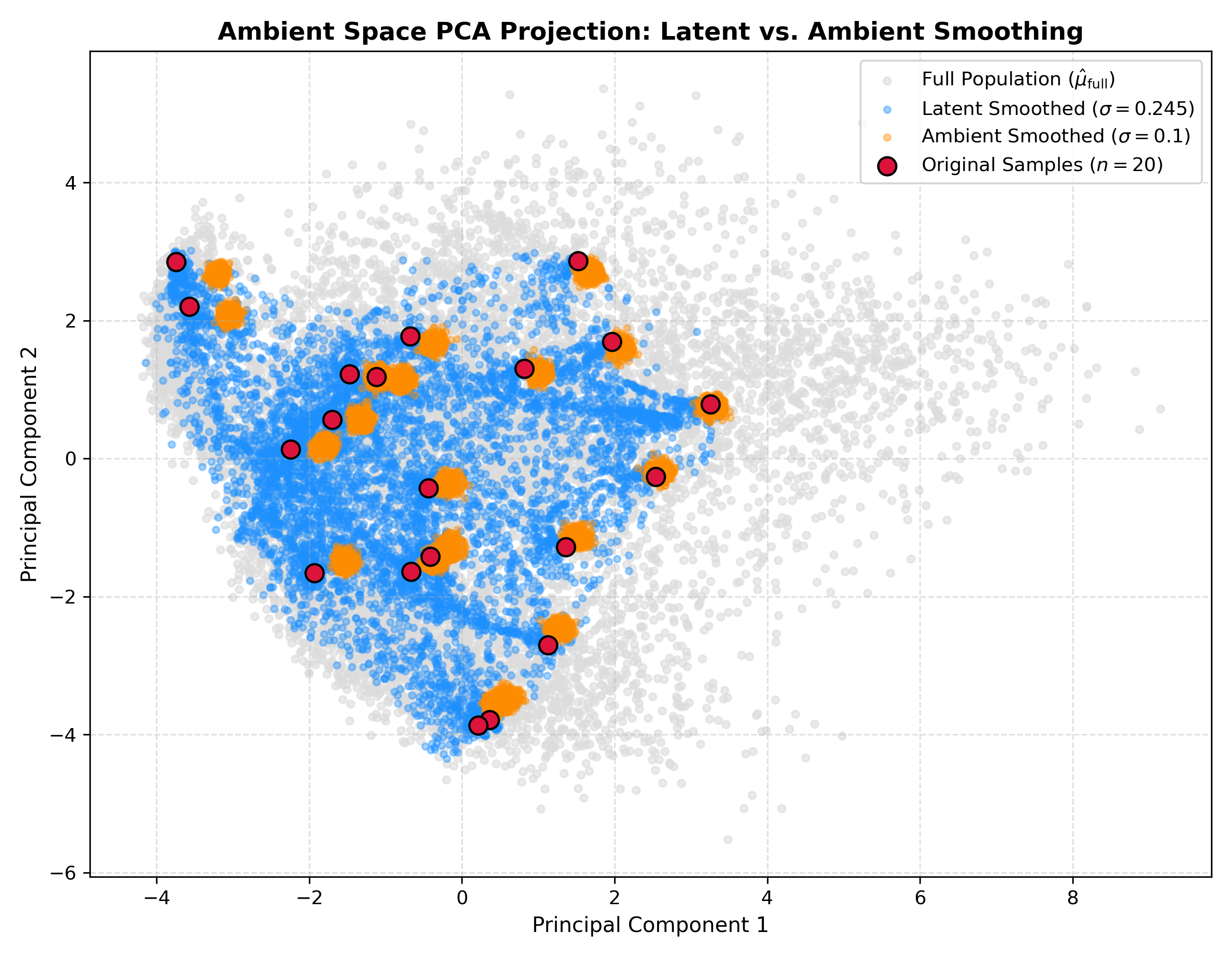}
\caption{Two-dimensional PCA visualization of the finite reference sample,
the $n=20$ training anchors, and generated samples at the empirically
selected controls. Gray: reference images. Crimson: training anchors.
Orange: clipped pixel perturbations at $\lambda_{\mathrm{pix}}=0.10$.
Light blue: decoded latent perturbations at
$\lambda_{\mathrm{lat}}=0.245$. The PCA projection illustrates coverage
and does not represent intrinsic geodesic coordinates.}
\label{fig:pca_projection}
\end{figure}

For each anchor, we generate $500$ pixel and $500$ decoded latent
perturbations at their respective empirically selected controls. In the
PCA projection, the pixel samples stay near the anchors, while the
decoded samples cover more of the reference-image region. This is a
visualization of projected coverage, not evidence of tangent or
geodesic motion.

\begin{figure}[t]
\centering
\includegraphics[width=\textwidth]
{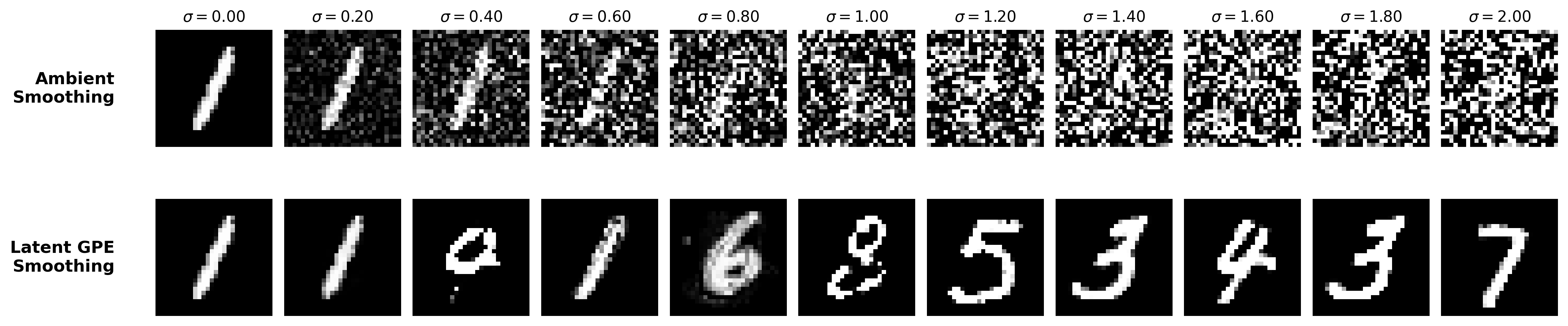}
\caption{Representative samples over increasing values of $\lambda$.
Upper row: Gaussian pixel perturbations followed by clipping. Lower row:
Gaussian latent perturbations followed by decoding. Noise is sampled
independently at each value, so columns are representative samples rather
than a coupled trajectory.}
\label{fig:sigma_sweep}
\end{figure}

The displayed pixel perturbations progressively obscure digit structure;
decoded latent perturbations remain recognizable over a wider displayed
range. These qualitative observations describe this trained model and do
not test the manifold assumptions of \Cref{thm:latent_decomposition}.

\bibliography{ref}

@article{fournier2015rate,
  title     = {On the rate of convergence in Wasserstein distance of the empirical measure},
  author    = {Fournier, Nicolas and Guillin, Arnaud},
  journal   = {Probability Theory and Related Fields},
  volume    = {162},
  number    = {3--4},
  pages     = {707--738},
  year      = {2015},
  publisher = {Springer}
}

@article{federer1959curvature,
  title     = {Curvature measures},
  author    = {Federer, Herbert},
  journal   = {Transactions of the American mathematical Society},
  volume    = {93},
  number    = {3},
  pages     = {418--491},
  year      = {1959}
}

@article{boissard2014mean,
  title     = {On the mean speed of convergence of empirical and occupation measures in Wasserstein distance},
  author    = {Boissard, Emmanuel and Le Gouic, Thibaut},
  journal   = {Annales de l'IHP Probabilit{\'e}s et statistiques},
  volume    = {50},
  number    = {2},
  pages     = {539--563},
  year      = {2014}
}

@article{weed2019sharp,
  title     = {Sharp asymptotic and finite-sample rates of convergence of empirical measures in Wasserstein distance},
  author    = {Weed, Jonathan and Bach, Francis},
  journal   = {Bernoulli},
  volume    = {25},
  number    = {4A},
  pages     = {2620--2648},
  year      = {2019}
}

@article{parzen1962estimation,
  title     = {On estimation of a probability density function and mode},
  author    = {Parzen, Emanuel},
  journal   = {The annals of mathematical statistics},
  volume    = {33},
  number    = {3},
  pages     = {1065--1076},
  year      = {1962}
}

@article{dudley1969speed,
  title     = {The speed of mean Glivenko-Cantelli convergence},
  author    = {Dudley, Richard Mansfield},
  journal   = {The Annals of Mathematical Statistics},
  volume    = {40},
  number    = {1},
  pages     = {40--50},
  year      = {1969}
}

@article{rosenblatt1956density,
  title   = {Remarks on some nonparametric estimates of a density function},
  author  = {Rosenblatt, Murray},
  journal = {The Annals of Mathematical Statistics},
  volume  = {27},
  number  = {3},
  pages   = {832--837},
  year    = {1956},
  doi     = {10.1214/aoms/1177728190}
}

@article{goldfeld2020convergence,
  title     = {Convergence of smoothed empirical measures with applications to entropy estimation},
  author    = {Goldfeld, Ziv and Greenewald, Kristjan and Niles-Weed, Jonathan and Polyanskiy, Yury},
  journal   = {IEEE Transactions on Information Theory},
  volume    = {66},
  number    = {7},
  pages     = {4368--4391},
  year      = {2020},
  publisher = {IEEE}
}

@article{goldfeld2024statistical,
  title     = {Statistical inference with regularized optimal transport},
  author    = {Goldfeld, Ziv and Kato, Kengo and Rioux, Gabriel and Sadhu, Ritwik},
  journal   = {Information and Inference: A Journal of the IMA},
  volume    = {13},
  number    = {1},
  pages     = {iaad056},
  year      = {2024},
  publisher = {Oxford University Press}
}

@article{block2026rate,
  title   = {Rate of convergence of the smoothed empirical Wasserstein distance},
  author  = {Block, Adam and Jia, Zeyu and Polyanskiy, Yury and Rakhlin, Alexander},
  journal = {Annales de l'Institut Henri Poincar{\'e}, Probabilit{\'e}s et Statistiques},
  year    = {2026},
  volume  = {62},
  number  = {3},
  pages   = {1848--1881},
  doi     = {10.1214/25-AIHP1571},
  url     = {https://arxiv.org/abs/2205.02128}
}

@article{chen2022asymptotics,
  title         = {Asymptotics of Smoothed Wasserstein Distances},
  author        = {Chen, Hong-Bin and Niles-Weed, Jonathan},
  journal       = {Potential Analysis},
  year          = {2022},
  volume        = {56},
  pages         = {571--595},
  doi           = {10.1007/s11118-020-09895-9},
  url           = {https://arxiv.org/abs/2005.00738}
}

@inproceedings{ding2022asymptotics,
  title     = {Asymptotics of Smoothed Wasserstein Distances in the Small Noise Regime},
  author    = {Ding, Yunzi and Niles-Weed, Jonathan},
  booktitle = {Advances in Neural Information Processing Systems},
  volume    = {35},
  year      = {2022},
  url       = {https://proceedings.neurips.cc/paper_files/paper/2022/hash/7a41a2bc087b6d1981235d1718e53f51-Abstract-Conference.html}
}

@article{borda2023empirical,
  title   = {Empirical measures and random walks on compact spaces in the quadratic Wasserstein metric},
  author  = {Borda, Bence},
  journal = {Annales de l'Institut Henri Poincar{\'e}, Probabilit{\'e}s et Statistiques},
  year    = {2023},
  volume  = {59},
  number  = {4},
  pages   = {2017--2035},
  doi     = {10.1214/22-AIHP1322},
  url     = {https://arxiv.org/abs/2110.00295}
}

@inproceedings{tolstikhin2018wasserstein,
  title     = {Wasserstein Auto-Encoders},
  author    = {Tolstikhin, Ilya and Bousquet, Olivier and Gelly, Sylvain and Sch{\"o}lkopf, Bernhard},
  booktitle = {International Conference on Learning Representations},
  year      = {2018},
  url       = {https://openreview.net/forum?id=HkL7n1-0b}
}

@inproceedings{chakrabarty2021statistical,
  title     = {Statistical Regeneration Guarantees of the Wasserstein Autoencoder with Latent Space Consistency},
  author    = {Chakrabarty, Anish and Das, Swagatam},
  booktitle = {Advances in Neural Information Processing Systems},
  volume    = {34},
  year      = {2021},
  url       = {https://papers.nips.cc/paper_files/paper/2021/hash/8e036cc193d0af59aa9b22821248292b-Abstract.html}
}

@inproceedings{mbacke2023statistical,
  title     = {Statistical Guarantees for Variational Autoencoders using {PAC-Bayesian} Theory},
  author    = {Mbacke, Sokhna Diarra and Clerc, Florence and Germain, Pascal},
  booktitle = {Advances in Neural Information Processing Systems},
  volume    = {36},
  year      = {2023},
  url       = {https://papers.nips.cc/paper_files/paper/2023/hash/b29500824d22ee9bbd25e4cd97c49b55-Abstract-Conference.html}
}

@article{loaiza2024manifold,
  title   = {Deep Generative Models through the Lens of the Manifold Hypothesis: A Survey and New Connections},
  author  = {Loaiza-Ganem, Gabriel and Ross, Brendan Leigh and Hosseinzadeh, Rasa and Caterini, Anthony L. and Cresswell, Jesse C.},
  journal = {Transactions on Machine Learning Research},
  year    = {2024},
  url     = {https://openreview.net/forum?id=a90WpmSi0I}
}

@misc{fu2026intrinsic,
  title         = {Intrinsic Wasserstein Rates for Score-Based Generative Models on Smooth Manifolds},
  author        = {Fu, Guoji and Suzuki, Taiji and Lee, Wee Sun and Nitanda, Atsushi},
  year          = {2026},
  eprint        = {2605.15822},
  archivePrefix = {arXiv},
  primaryClass  = {cs.LG},
  url           = {https://arxiv.org/abs/2605.15822}
}

@inproceedings{nietert2021smooth,
  title        = {Smooth $ p $-Wasserstein distance: structure, empirical approximation, and statistical applications},
  author       = {Nietert, Sloan and Goldfeld, Ziv and Kato, Kengo},
  booktitle    = {International conference on machine learning},
  pages        = {8172--8183},
  year         = {2021},
  organization = {PMLR}
}

@article{tenenbaum2000global,
  title     = {A global geometric framework for nonlinear dimensionality reduction},
  author    = {Tenenbaum, Joshua B and de Silva, Vin and Langford, John C},
  journal   = {science},
  volume    = {290},
  number    = {5500},
  pages     = {2319--2323},
  year      = {2000},
  publisher = {American Association for the Advancement of Science}
}

@article{belkin2003laplacian,
  title     = {Laplacian eigenmaps for dimensionality reduction and data representation},
  author    = {Belkin, Mikhail and Niyogi, Partha},
  journal   = {Neural computation},
  volume    = {15},
  number    = {6},
  pages     = {1373--1396},
  year      = {2003},
  publisher = {MIT Press}
}

@article{miyato2018spectral,
  title   = {Spectral normalization for generative adversarial networks},
  author  = {Miyato, Takeru and Kataoka, Toshiki and Koyama, Masanori and Yoshida, Yuichi},
  journal = {arXiv preprint arXiv:1802.05957},
  year    = {2018}
}

@inproceedings{vincent2008extracting,
  title     = {Extracting and composing robust features with denoising autoencoders},
  author    = {Vincent, Pascal and Larochelle, Hugo and Bengio, Yoshua and Manzagol, Pierre-Antoine},
  booktitle = {Proceedings of the 25th international conference on Machine learning},
  pages     = {1096--1103},
  year      = {2008}
}

@article{alain2014regularized,
  title   = {What regularized auto-encoders learn from the data-generating distribution},
  author  = {Alain, Guillaume and Bengio, Yoshua},
  journal = {Journal of Machine Learning Research},
  volume  = {15},
  number  = {110},
  pages   = {3743--3773},
  year    = {2014}
}

@inproceedings{rombach2022high,
  title        = {High-resolution image synthesis with latent diffusion models},
  author       = {Rombach, Robin and Blattmann, Andreas and Lorenz, Dominik and Esser, Patrick and Ommer, Bj{\"o}rn},
  booktitle    = {2022 IEEE/CVF conference on computer vision and pattern recognition (CVPR)},
  pages        = {10684--10695},
  year         = {2022},
  organization = {ieee}
}

@article{lee2025geometry,
  title   = {Geometry-preserving encoder/decoder in latent generative models},
  author  = {Lee, Wonjun and O'Neill, Riley CW and Zou, Dongmian and Calder, Jeff and Lerman, Gilad},
  journal = {arXiv preprint arXiv:2501.09876},
  year    = {2025}
}

@book{boucheron2013concentration,
  author    = {Boucheron, St{\'e}phane and Lugosi, G{\'a}bor and Massart, Pascal},
  title     = {Concentration Inequalities: A Nonasymptotic Theory of Independence},
  publisher = {Oxford University Press},
  address   = {Oxford},
  year      = {2013}
}

@article{flamary2021pot,
  title   = {Pot: Python optimal transport},
  author  = {Flamary, R{\'e}mi and Courty, Nicolas and Gramfort, Alexandre and Alaya, Mokhtar Z and Boisbunon, Aur{\'e}lie and Chambon, Stanislas and Chapel, Laetitia and Corenflos, Adrien and Fatras, Kilian and Fournier, Nemo and others},
  journal = {Journal of Machine Learning Research},
  volume  = {22},
  number  = {78},
  pages   = {1--8},
  year    = {2021}
}

@article{song2020score,
  title   = {Score-based generative modeling through stochastic differential equations},
  author  = {Song, Yang and Sohl-Dickstein, Jascha and Kingma, Diederik P and Kumar, Abhishek and Ermon, Stefano and Poole, Ben},
  journal = {arXiv preprint arXiv:2011.13456},
  year    = {2020}
}

@article{ho2020denoising,
  title   = {Denoising diffusion probabilistic models},
  author  = {Ho, Jonathan and Jain, Ajay and Abbeel, Pieter},
  journal = {Advances in neural information processing systems},
  volume  = {33},
  pages   = {6840--6851},
  year    = {2020}
}

@article{hyvarinen2005estimation,
  title   = {Estimation of non-normalized statistical models by score matching},
  author  = {Hyv{\"a}rinen, Aapo},
  journal = {Journal of Machine Learning Research},
  volume  = {6},
  number  = {24},
  pages   = {695--709},
  year    = {2005}
}

@article{vincent2011connection,
  title     = {A connection between score matching and denoising autoencoders},
  author    = {Vincent, Pascal},
  journal   = {Neural computation},
  volume    = {23},
  number    = {7},
  pages     = {1661--1674},
  year      = {2011},
  publisher = {MIT Press}
}

@article{pidstrigach2022score,
  title   = {Score-based generative models detect manifolds},
  author  = {Pidstrigach, Jakiw},
  journal = {Advances in Neural Information Processing Systems},
  volume  = {35},
  pages   = {35852--35865},
  year    = {2022}
}

@article{de2022convergence,
  title   = {Convergence of denoising diffusion models under the manifold hypothesis},
  author  = {De Bortoli, Valentin},
  journal = {arXiv preprint arXiv:2208.05314},
  year    = {2022}
}

@inproceedings{oko2023diffusion,
  title        = {Diffusion models are minimax optimal distribution estimators},
  author       = {Oko, Kazusato and Akiyama, Shunta and Suzuki, Taiji},
  booktitle    = {International Conference on Machine Learning},
  pages        = {26517--26582},
  year         = {2023},
  organization = {PMLR}
}

@inproceedings{feydy2019interpolating,
  title        = {Interpolating between optimal transport and mmd using sinkhorn divergences},
  author       = {Feydy, Jean and S{\'e}journ{\'e}, Thibault and Vialard, Fran{\c{c}}ois-Xavier and Amari, Shun-ichi and Trouv{\'e}, Alain and Peyr{\'e}, Gabriel},
  booktitle    = {The 22nd international conference on artificial intelligence and statistics},
  pages        = {2681--2690},
  year         = {2019},
  organization = {PMLR}
}

@article{niles2022minimax,
  title     = {Minimax estimation of smooth densities in Wasserstein distance},
  author    = {Niles-Weed, Jonathan and Berthet, Quentin},
  journal   = {The Annals of Statistics},
  volume    = {50},
  number    = {3},
  pages     = {1519--1540},
  year      = {2022}
}

@article{divol2022measure,
  title     = {Measure estimation on manifolds: an optimal transport approach},
  author    = {Divol, Vincent},
  journal   = {Probability Theory and Related Fields},
  volume    = {183},
  number    = {1},
  pages     = {581--647},
  year      = {2022},
  publisher = {Springer}
}

@article{erbar2015equivalence,
  author = {Erbar, Matthias and Kuwada, Kazumasa and Sturm, Karl-Theodor},
  title = {On the equivalence of the entropic curvature-dimension condition and Bochner's inequality on metric measure spaces},
  journal = {Inventiones Mathematicae},
  volume = {201},
  number = {3},
  pages = {993--1071},
  year = {2015},
  doi = {10.1007/s00222-014-0563-7}
}

@article{menkovski2024entropy,
  author = {Menkovski, Vlado and Portegies, Jacobus W. and Ravelonanosy, Mahefa Ratsisetraina},
  title = {Small time asymptotics of the entropy of the heat kernel on a Riemannian manifold},
  journal = {Applied and Computational Harmonic Analysis},
  volume = {71},
  pages = {101642},
  year = {2024},
  doi = {10.1016/j.acha.2024.101642}
}

@article{boissonnat2019reach,
  author = {Boissonnat, Jean-Daniel and Lieutier, Andr{\'e} and Wintraecken, Mathijs},
  title = {The reach, metric distortion, geodesic convexity and the variation of tangent spaces},
  journal = {Journal of Applied and Computational Topology},
  volume = {3},
  pages = {29--58},
  year = {2019},
  doi = {10.1007/s41468-019-00029-8}
}
\end{document}